\documentclass{amsart}
\usepackage{amsmath}
\usepackage{paralist}
\usepackage{amsfonts}
\usepackage{amssymb}
\usepackage{amsthm}
\usepackage{amscd}
\usepackage{float}
\usepackage{tikz}
\usepackage{graphicx}
\usepackage[colorlinks=true]{hyperref}
\hypersetup{urlcolor=blue, citecolor=red}
\usepackage{hyperref}

\newtheorem{theorem}{Theorem}[section]

\newtheorem{lemma}[theorem]{Lemma}
\newtheorem{proposition}[theorem]{Proposition}

\theoremstyle{definition}
\newtheorem{definition}[theorem]{Definition}
\newtheorem{remark}[theorem]{Remark}

\newcommand{\T}{\mathbb{T}}
\newcommand{\R}{\mathbb{R}}
\newcommand{\Z}{\mathbb{Z}}
\newcommand{\N}{\mathbb{N}}
\newcommand{\C}{\mathbb{C}}

\def\Xint#1{\mathchoice
{\XXint\displaystyle\textstyle{#1}}%
{\XXint\textstyle\scriptstyle{#1}}%
{\XXint\scriptstyle\scriptscriptstyle{#1}}%
{\XXint\scriptscriptstyle\scriptscriptstyle{#1}}%
\!\int}
\def\XXint#1#2#3{{\setbox0=\hbox{$#1{#2#3}{\int}$ }
\vcenter{\hbox{$#2#3$ }}\kern-.6\wd0}}

\def\dashint{\Xint-}

\begin{document}
\title[Oscillatory integral operators with homogeneous phase functions]{Oscillatory integral operators with homogeneous phase functions}
\author{Robert Schippa}
\address{Fakult\"at f\"ur Mathematik, Karlsruher Institut f\"ur Technologie,
Englerstrasse 2, 76131 Karlsruhe, Germany}
\email{robert.schippa@kit.edu}
\begin{abstract}
Oscillatory integral operators with $1$-homogeneous phase functions satisfying a convexity condition are considered. For these we show the $L^p - L^p$-estimates for the Fourier extension operator of the cone due to Ou--Wang via polynomial partitioning. For this purpose, we combine the arguments of Ou--Wang with the analysis of Guth--Hickman--Iliopoulou, who previously showed sharp $L^p-L^p$-estimates for non-homogeneous phase functions with variable coefficients under a convexity assumption. Furthermore, we provide examples exhibiting Kakeya compression, which shows the estimates to be sharp. We apply the oscillatory integral estimates to show new local smoothing estimates for wave equations on compact Riemannian manifolds $(M,g)$ with $\dim M \geq 3$. This generalizes the argument for the Euclidean wave equation due to Gao--Liu--Miao--Xi.
\end{abstract}

\maketitle

\section{Introduction}
\label{section:Introduction}

In the following we consider oscillatory integral operators which naturally generalize the Fourier extension operator for the cone
\begin{equation}
\label{eq:ConeExtensionOperator}
\mathcal{E} f(x) = \int_{A^{n-1}} e^{i (\langle x',\omega \rangle + x_n |\omega|)} f(\omega) d\omega.
\end{equation}
We consider operators with $\lambda \geq 1$,
\begin{equation}
\label{eq:OscillatoryIntegralOperators}
T^\lambda f(x) = \int e^{i \phi^\lambda(x;\omega)} a^\lambda(x;\omega) f(\omega) d\omega
\end{equation}
and $a \in C_c^\infty(\R^{n} \times \R^{n-1},\R)$, $\phi \in C^\infty(\R^n \times \R^{n-1} \backslash 0;\R)$, $\phi^\lambda(x;\omega) = \lambda \phi(x/\lambda;\omega)$, $a^\lambda(x;\omega) = a(x/\lambda;\omega)$. We suppose that $\phi$ is $1$-homogeneous in $\omega$: it holds for $\mu >0$
\begin{equation}
\label{eq:HomogeneityPhaseFunction}
\phi(x;\mu \omega) = \mu \phi(x;\omega).
\end{equation}
For the support of $a$ we suppose that
\begin{equation*}
\text{supp}(a) \subseteq A^{n-1} = B_{n-1}(0,2) \backslash B_{n-1}(0,1/2).
\end{equation*}
Above $B_d(c,r) \subseteq \R^d$ denotes the open ball with center $c$ and radius $r>0$.

We write $x=(x',x_n) \in \R^{n-1} \times \R$ and impose the following conditions on $\phi$ in $\text{supp}(a)$:
\begin{align*}
&C1) \quad \text{rank}(\partial^2_{x \omega} \phi) = n-1, \\
&C2^+) \quad \partial^2_{\omega \omega} \langle \partial_x \phi, G(x;\omega_0) \rangle \big|_{\omega = \omega_0} \text{ has } n-2 \text{ non-vanishing eigenvalues of the same sign},
\end{align*}
where $G$ denotes the Gauss map
\begin{equation}
\label{eq:GaussMap}
G_0(x;\omega) = \bigwedge_{j=1}^{n-1} \partial^2_{x \omega_j} \phi(x;\omega), \qquad G = G_0 / |G_0|
\end{equation}
of the embedded surface $\omega \mapsto \partial_x \phi(x;\omega)$. We identify $\bigwedge^{n-1} \R^n \simeq \R^n$.

In this note we prove new estimates
\begin{equation}
\label{eq:LpLpBoundVariableCoefficient}
\| T^\lambda f \|_{L^p(\R^n)} \lesssim_{\varepsilon,\phi,a} \lambda^\varepsilon \Vert f \Vert_{L^p(A^{n-1})}
\end{equation}
for operators \eqref{eq:OscillatoryIntegralOperators} like described above. Firstly, we recall that the conjectured range of $L^p$-estimates 
\begin{equation}
\label{eq:FourierExtensionCone}
\| \mathcal{E} f \|_{L^p(\R^n)} \lesssim \| f \|_{L^p(A^{n-1})}
\end{equation}
 is given by $p>\frac{2(n-1)}{n-2}$. This prominent open problem is known as \emph{restriction conjecture for the cone} and goes back to Stein. The corresponding conjecture for the paraboloid is given by
 \begin{equation*}
 \| \int_{B_{n-1}(0,1)} e^{i( \langle x',\omega \rangle + x_n |\omega|^2)} f(\omega) d\omega \|_{L^p(\R^n)} \lesssim \| f \|_{L^p(B_{n-1})}
 \end{equation*}
 for $p>\frac{2n}{n-1}$. Note that there is a shift of dimension by one, which is commonly explained by the number of principal curvatures being reduced by one.
 More pictorially, disregarding the null direction of the cone, the conical hypersurface in $n$ dimensions locally looks like a parabolic hypersurface in $n-1$ dimensions.
 
The restriction conjecture for the cone was solved for $n=3$ by Taberner \cite{Taberner1985}, for $n=4$ by Wolff \cite{Wolff2001} via bilinear estimates, and for $n=5$ by Ou--Wang \cite{OuWang2022} via polynomial partitioning. Let
\begin{equation}
\label{eq:PolynomialPartitioningRange}
p_n = \begin{cases}
&4, \quad n=3, \\
&2 \cdot \frac{3n+1}{3n-3}, \quad n > 3 \text{ odd}, \\
&2 \cdot \frac{3n}{3n-4}, \quad n>3 \text{ even}.
\end{cases}
\end{equation}

Ou--Wang showed \eqref{eq:FourierExtensionCone} for $p > p_n$, which is also currently the widest range in higher dimensions to the best of the author's knowledge. Notably, in the case of Carleson-Sj\"olin phase functions (cf. \cite{CarlesonSjoelin1972,Hoermander1973}), which are not $1$-homogeneous anymore, where $C2^+)$ is replaced with
\begin{equation*}
H2^+) \quad \partial^2_{\omega \omega} \langle \partial_x \phi(x;\omega) , G(x;\omega_0) \rangle \big|_{\omega = \omega_0} \text{ has } n-1 \text{ eigenvalues of the same sign},
\end{equation*}
Guth--Hickman--Iliopoulou \cite{GuthHickmanIliopoulou2019} showed the sharp range of $L^p - L^p$ estimates, in the sense that there are phase functions for which the estimate fails for lower values of $p$. The deviation from the corresponding generalized restriction conjecture for the paraboloid occurs due to \emph{Kakeya compression}. This means that wave packets, which we find after decomposing $T^\lambda f$ into non-oscillating components, can cluster in low-dimensional varieties. This is not known to happen in the constant-coefficient case: The restriction conjecture for the sphere implies the \emph{Kakeya conjecture}.

 Kakeya compression was initially observed by Bourgain \cite{Bourgain1991}, see also Wisewell \cite{Wisewell2005} and Bourgain--Guth \cite{BourgainGuth2011}. Related phenomena were discussed by Minicozzi--Sogge \cite{MinicozziSogge1997} and Sogge \cite{Sogge1999}. In this note we point out Kakeya compression for $1$-homogeneous phases with variable coefficients, which shows that the following $L^p$-estimates are sharp up to endpoints:
\begin{theorem}
\label{thm:LpLpEstimatesVariableCoefficients}
Let $\phi:\R^n \times \R^{n-1} \backslash \{ 0 \} \rightarrow \R$ be a $1$-homogeneous phase satisfying $C1)$ and $C2^+)$ and $a \in C^\infty_c(A^{n-1})$ be an amplitude. Then, the estimate \eqref{eq:LpLpBoundVariableCoefficient} holds for $p \geq p_n$ with $p_n$ as in \eqref{eq:PolynomialPartitioningRange}.
\end{theorem}
We remark that for $p>p_n$ the $\lambda^\varepsilon$-factor can be dropped. Guth--Hickman--Iliopoulou showed the $\varepsilon$-removal lemma for oscillatory integral operators in \cite[Section~12]{GuthHickmanIliopoulou2019}, albeit with a stronger non-degeneracy hypothesis than presently considered. The idea goes back to Tao \cite{Tao1998,Tao1999}. In Section \ref{section:epsRemoval} we prove the following global estimates for $p > p_n$ by a small variation of the argument in \cite{GuthHickmanIliopoulou2019}:
\begin{equation*}
\| T^\lambda f \|_{L^p(\R^n)} \lesssim_{\phi,a} \| f \|_{L^p(A^{n-1})}.
\end{equation*}

\bigskip

The proof of Theorem \ref{thm:LpLpEstimatesVariableCoefficients} combines ideas from the case of constant-coefficient homogeneous phases due to Ou--Wang \cite{OuWang2022} and Gao--Liu--Miao--Xi \cite{GaoLiuMiaoXi2023} and variable-coefficient non-homogeneous phases due to Guth--Hickman--Iliopoulou \cite{GuthHickmanIliopoulou2019}.

\medskip

 We digress for a moment to describe the tools we will use and put them into context. Bennett--Carbery--Tao \cite{BennettCarberyTao2006} delivered an important contribution with sharp $n$-multilinear restriction estimates:
The $k$-restriction conjecture for the cone reads
\begin{equation}
\label{eq:kRestrictionConjecture}
\big\| \prod_{i=1}^k |\mathcal{E} f_i |^{\frac{1}{k}} \big\|_{L^{p_k}(B_n(0,R))} \lesssim_\varepsilon R^\varepsilon  \prod_{i=1}^k \| f_i \|^{\frac{1}{k}}_2
\end{equation}
with $|\mathfrak{n}(\xi_1) \wedge \ldots \wedge \mathfrak{n}(\xi_k)| \gtrsim 1$ with $\mathfrak{n}(\xi) = (-\frac{\xi}{|\xi|},1)$, $\xi_i \in \text{supp}(f_i) \subseteq A_{n-1}$, and $p_k = \frac{2(n+k)}{n+k-2}$. For $k=2$, this is due to Wolff \cite{Wolff2001} (see Tao \cite{Tao2001} for the endpoint without $R^\varepsilon$-loss), and for $k=n$ this is an instance of the Bennett--Carbery--Tao theorem \cite{BennettCarberyTao2006}. There are only few partial results for $3 \leq k \leq n-1$, see \cite{Bejenaru2017} and references therein.

\medskip
 
We note that in \cite{BennettCarberyTao2006} the multilinear estimates were shown as well for constant-coefficient phase functions as smooth perturbations thereof. Bourgain--Guth \cite{BourgainGuth2011} devised an iteration to deduce linear estimates from multilinear estimates. Guth \cite{Guth2016} observed that the full strength of $k$-multilinear estimates is not required, but a slightly weaker variant given by $k$-broad norms suffices to run the iteration. He used polynomial partitioning to improve on the previous results in \cite{Guth2016,Guth2018}. The idea is to equipartition the broad norm with polynomials of controlled degree: After wave packet decomposition, one finds that either the broad norm is concentrated on ``cells" or on the ``wall", which is a neighbourhood of a variety. To oversimplify matters for a moment, if the broad norm is concentrated on the cells, then sharp bounds follow from induction on scales. If the broad norm is concentrated along the wall, then we are morally dealing with a restriction problem in lower dimensions, which is amenable to another induction hypothesis.

\medskip

We introduce the $k$-broad norms in the present context: For the definition decompose $A^{n-1}$ into finitely overlapping sectors $\tau$ of aperture $\sim K^{-1}$ and length $\sim 1$, where $K$ is a large constant. Given $f: A^{n-1} \to \C$, write $f = \sum f_\tau$, where $f_\tau$ is supported in $\tau$. In view of the rescaling $\phi^\lambda$ of the phase, we define the rescaled Gauss map
\begin{equation*}
G^\lambda(x;\omega) = G(\frac{x}{\lambda};\omega) \text{ for } (x;\omega) \in \text{supp} (a^\lambda).
\end{equation*}
For each $x \in B(0,\lambda)$ let 
\begin{equation*}
G^\lambda(x;\tau) = \{ G^\lambda(x;\omega) \; : \; \omega \in \tau \text{ and } (x;\omega) \in \text{supp} (a^\lambda) \}.
\end{equation*}
For $V \subseteq \R^n$ a linear subspace, let $\angle(G^\lambda(x;\tau),V)$ denote the smallest angle between any non-zero vector $v \in V$ and $G^\lambda(x;\tau)$.

The spatial ball $B(0,\lambda)$ is decomposed into relatively small balls $B_{K^2}$ of radius $K^2$. We fix $\mathcal{B}_{K^2}$ a collection of finitely-overlapping $K^2$-balls, which are centred in and cover $B(0,\lambda)$. For $B_{K^2} \in \mathcal{B}_{K^2}$ centred at $\bar{x} \in B(0,\lambda)$, define
\begin{equation}
\label{eq:BroadNorm}
\mu_{T^\lambda f}(B_{K^2}) = \min_{V_1,\ldots,V_A \in Gr(k-1,n)} \big( \max_{\tau: \angle(G^\lambda(\bar{x};\tau),V_a) > K^{-2} \; \forall a} \| T^\lambda f_{\tau} \|^p_{L^p(B_{K^2})} \big),
\end{equation}
where $Gr(k-1,n)$ denotes the Grassmannian manifold of $(k-1)$-dimensional subspaces in $\R^n$. We stress the deviation from \cite{GuthHickmanIliopoulou2019}, in which the angle threshold $K^{-1}$ was considered. In case of the Fourier extension operator for the cone, the angle condition was narrowed to $K^{-2}$ to further confine the narrow part in \cite{OuWang2022}. Hence, we presently use the same threshold for the variable-coefficient operator.

We write $\tau \not \in V_a$ as shorthand for $\angle(G^\lambda(\bar{x};\tau),V_a) > K^{-2}$ provided that $\bar{x}$ is clear from context. Thus, we can write as well
\begin{equation*}
\mu_{T^\lambda f}(B_{K^2}) = \min_{V_1,\ldots,V_A \in Gr(k-1,n)} \big( \max_{\substack{\tau: \tau \not \in V_a, \\ \text{ for } 1 \leq a \leq A}} \| T^\lambda f_\tau \|^p_{L^p(B_{K^2})} \big).
\end{equation*}
For $U \subseteq \R^n$ the $k$-broad norm is defined as
\begin{equation*}
\| T^\lambda f \|_{BL^p_{k,A}(U)} = \big( \sum_{\substack{B_{K^2} \in \mathcal{B}_{K^2}, \\ B_{K^2} \cap U \neq \emptyset}} \mu_{T^\lambda f}(B_{K^2}) \big)^{1/p}.
\end{equation*}
A key step in the proof of the $L^p$-$L^p$-estimate is to show $k$-broad estimates, which are a substitute for the $k$-restriction conjecture.
\begin{theorem}
\label{thm:kBroadEstimate}
For $2 \leq k \leq n$ and all $\varepsilon > 0$, there exists a constant $C_\varepsilon > 1$ and an integer $A$ such that, whenever $T^\lambda$ is an oscillatory integral operator with reduced $(\phi,a)$ and $\phi$ a $1$-homogeneous phase  satisfying $C1)$ and $C2^+)$, the estimate
\begin{equation}
\| T^\lambda f \|_{BL^p_{k,A}(\R^n)} \lesssim_\varepsilon K^{C_\varepsilon} \lambda^\varepsilon \| f \|_{L^2(A^{n-1})}
\end{equation}
holds for all $\lambda \geq 1$ and $K \geq 1$ whenever
\begin{equation}
p \geq \bar{p}(k,n) = \frac{2(n+k)}{n+k-2}.
\end{equation}
\end{theorem}
Reduced phase functions are introduced in Section \ref{subsection:BasicReductions}. These phases are basically small $C^N$ perturbations of $1$-homogeneous phases with constant coefficients satisfying the convexity assumption. These reduced phase functions were previously used by Beltran--Hickman--Sogge \cite{BeltranHickmanSogge2020} to derive decoupling estimates. As in \cite{BeltranHickmanSogge2020}, general phase functions satisfying $C1)$ and $C2^+)$ are transformed by partitioning the support of the amplitude and parabolic rescaling to reduced phases. Additionally, we impose size conditions on the amplitude and a margin condition on its support.

 The scheme to deduce Theorem \ref{thm:LpLpEstimatesVariableCoefficients} from Theorem \ref{thm:kBroadEstimate} is essentially due to Bourgain--Guth \cite{BourgainGuth2011} and in the variable-coefficient context, see also Guth--Hickman--Iliopoulou \cite{GuthHickmanIliopoulou2019}. We find like in \cite[p.~263]{GuthHickmanIliopoulou2019}
\begin{equation}
\label{eq:BroadNarrowEstimate}
\| T^\lambda f\|^p_{L^p(B_{K^2})} \lesssim_A K^{O(1)} \mu_{T^\lambda f}(B_{K^2}) + \sum_{a=1}^A \| \sum_{\tau \in V_a} T^\lambda f_\tau \|^p_{L^p(B_{K^2})}.
\end{equation}
The first term is captured by the broad estimate; the second term is estimated by $\ell^p$-decoupling (cf. \cite{BourgainDemeter2015,BeltranHickmanSogge2020}) and induction on scales \cite{OuWang2022,GaoLiuMiaoXi2023}. A suitable narrow decoupling, taking into account a reduced number of sectors $\tau \in V_a$, is discussed in detail in the variable coefficient case. For the constant-coefficient case we refer to \cite{OuWang2022} and \cite{GaoLiuMiaoXi2023}.

\medskip

Very recently, Gao--Liu--Miao--Xi \cite{GaoLiuMiaoXi2023} proved an extension of Ou--Wang's result for the circular cone $\phi(x,\omega) = x' \cdot \omega+x_n|\omega|$ for more general conic surfaces, but still with constant coefficients. For these constant coefficient phase functions, the Kakeya compression described in the present work cannot happen. Gao \emph{et al.} \cite{GaoLiuMiaoXi2023} used $k$-broad estimates to derive new local smoothing estimates for the wave equation in Euclidean space. At small spatial scales, the variable coefficient phases are approximated with extension operators for conic surfaces. Then we can use arguments from \cite{GaoLiuMiaoXi2023}. Furthermore, Hickman and Iliopoulou showed sharp $L^p$-estimates for non-homogeneous phases with indefinite signature in \cite{HickmanIliopoulou2022}. This suggests to study also homogeneous phase functions with indefinite signature by the present methods.

\medskip

Notably, we do not use the usual wave packet decomposition for the cone as e.g. in \cite{OuWang2022} or \cite{GaoLiuMiaoXi2023} to prove the broad estimate. Instead, we stick to the wave packet decomposition commonly used for the Fourier extension operator of the paraboloid or its variable coefficient counterpart \cite{GuthHickmanIliopoulou2019}. This allows us to use many arguments from \cite{GuthHickmanIliopoulou2019} without change and hints at the possibility of a unified approach. A major change happens for the transverse equidistribution estimates, to be analyzed in Section \ref{section:TransverseEquidistribution}. Secondly, the narrow decoupling requires additional considerations, see Section \ref{section:LinearEstimates}. 

\medskip

We remark that the idea to use the same wave packet decomposition for homogeneous and inhomogeneous phase functions in the variable coefficient context is not new: In \cite{Lee2006} S. Lee considered linear and bilinear estimates for oscillatory integral operators and could treat variable coefficient versions of the Fourier extension operator of the paraboloid and the cone with the same wave packet decomposition. He generalized bilinear estimates due to Tao \cite{Tao2003} and Wolff \cite{Wolff2001} to variable coefficient phases. Lee \cite{Lee2006} pointed out for the first time that a convexity condition as $H2^+)$ or $C2^+)$ allows to go beyond Tomas--Stein $L^2-L^p$-estimates, which are sharp for phases without convexity condition. Bourgain \cite{Bourgain1991} showed in the context of non-homogeneous phases without convexity conditions that the Tomas-Stein range is sharp (see also \cite{BourgainGuth2011}). In the present work, the $L^p$-$L^p$-estimates for general oscillatory integral operators with phase satisfying $C1)$ and $C2^+)$ due to Lee \cite{Lee2006} are improved to the sharp range up to the endpoint for $n \geq 5$, and the previously known sharp $L^p$-$L^p$-estimates in lower dimensions $3 \leq n \leq 4$ are recovered.

\medskip

Furthermore, Ou--Wang \cite{OuWang2022} proved $L^q$-$L^p$-estimates for the cone extension operator $\mathcal{E}$; see \cite[Theorem~1.2]{OuWang2022}. The conjectured boundedness range of $\mathcal{E}:L^q \to L^p$ is given by $p>\frac{2(n-1)}{n-2}$ and $q' \leq \frac{n-2}{n} p$. For $n=5$ Ou--Wang proved this more general Fourier restriction conjecture for the cone and made progress for $n \geq 6$. The proof of \cite[Theorem~1.2]{OuWang2022} is a minor variant of the proof of $L^p$-$L^p$-estimates. It is based on the $k$-broad estimate and narrow decoupling. At one point, also interpolation with the bilinear estimates due to Wolff \cite{Wolff2000} is invoked. In the present work the $k$-broad estimate and narrow decoupling are extended to the variable coefficient setting, and Wolff's bilinear estimate was already generalized by Lee \cite{Lee2006}. For this reason, the arguments to show $L^q$-$L^p$-estimates in \cite{OuWang2022} can be extended to the variable coefficient case and yield the same range as stated in \cite[Theorem~1.2]{OuWang2022} for the estimate
\begin{equation*}
\| T^\lambda f \|_{L^p} \lesssim_{\varepsilon,\phi,a} \lambda^\varepsilon \| f \|_{L^q}.
\end{equation*}

\medskip

In Section \ref{section:LocalSmoothing} we apply the new estimates for oscillatory integral operators to prove new local smoothing estimates for solutions to wave equations on compact Riemannian manifolds $(M,g)$ with $\dim(M) \geq 3$. To avoid confusion with the number of space-time dimensions, which was previously denoted by $n$, we denote the space dimension in the following by $d$.
We consider
\begin{equation}
\label{eq:WaveEquationCompactManifoldIntroduction}
\left\{ \begin{array}{clcll}
\partial_t^2 u - \Delta_g u &= 0, &\quad (x,t) &\in& M \times \R, \\
u(\cdot,0) &= f_0, &\quad \dot{u}(\cdot,0) &=& f_1.
\end{array} \right.
\end{equation}
$\Delta_g$ denotes the Laplace--Beltrami operator, and the solution $u$ to \eqref{eq:WaveEquationCompactManifoldIntroduction} is given by
\begin{equation*}
u(t) = \cos(t \sqrt{- \Delta_g}) f_0 + \frac{\sin(t \sqrt{- \Delta_g})}{\sqrt{- \Delta_g}} f_1.
\end{equation*}
By results due to Seeger--Sogge--Stein \cite{SeegerSoggeStein1991} relying on the parametrix representation (see also \cite{Peral1980,Miyachi1980} in the Euclidean case), it is known that the fixed-time estimate
\begin{equation*}
\| u(\cdot,t) \|_{L^p(\R^d)} \lesssim \| f_0 \|_{L^p_{\bar{s}_p}(\R^d)} + \| f_1 \|_{L^p_{\bar{s}_p-1}(\R^d)}
\end{equation*}
with 
\begin{equation}
\label{eq:SharpFixedTimeRegularity}
\bar{s}_p = (d-1) \big| \frac{1}{2} - \frac{1}{p} \big|
\end{equation}
is sharp for all $1<p<\infty$ provided that $t$ avoids a discrete set. The local smoothing conjecture due to Sogge \cite{Sogge1991} for the Euclidean wave equation, i.e., $(M,g) = (\R^d,(\delta^{ij}))$ in \eqref{eq:WaveEquationCompactManifoldIntroduction}, states that
\begin{equation}
\label{eq:LocalSmoothingEstimate}
\big( \int_1^2 \| u(\cdot,t) \|_{L^p(\R^d)}^p \big)^{\frac{1}{p}} \lesssim \| f_0 \|_{L^p_{\bar{s}_p-\sigma}(\R^d)} + \| f_1 \|_{L^p_{\bar{s}_p-1-\sigma}(\R^d)}
\end{equation}
for $\sigma < \frac{1}{p}$ and $\frac{2d}{d-1} \leq p < \infty$. (Note that $\bar{s}_p - \frac{1}{p} = 0$ for $p = \frac{2d}{d-1}$.) This conjecture stands on top of prominent open problems in Harmonic Analysis as it implies as well the restriction conjecture for the paraboloid as the Bochner--Riesz conjecture. Initial progress was due to Sogge \cite{Sogge1991} and Mockenhaupt--Seeger--Sogge \cite{MockenhauptSeegerSogge1993}. Wolff identified decoupling inequalities \cite{Wolff2000} to yield sharp local smoothing estimates. Further progress in this direction was made in \cite{GarrigosSeeger2009,Lee2020,LeeVargas2012}. Bourgain--Demeter \cite{BourgainDemeter2015} covered the sharp range for decoupling inequalities, which implies sharp local smoothing estimates for $p \geq \frac{2(d+1)}{d-1}$. We refer to the survey by Beltran--Hickman--Sogge \cite{BeltranHickmanSogge2018} for local smoothing estimates for FIOs. Guth--Wang--Zhang \cite{GuthWangZhang2020} verified the Euclidean local smoothing conjecture for $d=2$ by a sharp $L^4$-square function estimate. Gao \emph{et al.} \cite{GaoLiuMiaoXi2020} extended this to compact Riemannian surfaces. We remark that for $d \geq 3$, counterexamples due to Minicozzi--Sogge \cite{MinicozziSogge1997} show that \eqref{eq:LocalSmoothingEstimate} fails if one replaces $\R^d$ with general compact Riemannian manifolds for $\sigma < 1/p$, if $p < p_{d,+}$ with
\begin{equation}
\label{eq:LpCompactManifold}
p_{d,+} =
\begin{cases}
\frac{2 \cdot (3d+1)}{3d-3}, \text{ if } d \text{ is odd}, \\
\frac{2 \cdot (3d+2)}{3d-2}, \text{ if } d \text{ is even}.
\end{cases}
\end{equation}
Hence, local smoothing estimates for solutions to wave equations on compact Riemannian manifolds are only conjectured for $p \geq p_{d,+}$ with $\sigma < 1/p$.

Local parametrices for $e^{it \sqrt{- \Delta_g}}$ are given by
\begin{equation*}
\mathcal{F} f(x',x_n) = \int_{\R^{n-1}} e^{i \phi(x',x_n;\omega)} a(x;\omega) \hat{f}(\omega) d\omega
\end{equation*}
with $\phi \in C^\infty(\R^n \times \R^{n-1} \backslash \{ 0 \})$ a phase function satisfying $C1)$ and $C2^+)$ and $a \in S^0(\R^{n})$ with compact support in $x$. Hence, it suffices to prove local smoothing estimates of rescaled Fourier integral operators $\mathcal{F}^\lambda$. In Theorem \ref{thm:ImprovedLocalSmoothing} we extend the recent results due to Gao \emph{et al.} \cite{GaoLiuMiaoXi2023} for the Euclidean wave equation to wave equations on compact Riemannian manifolds. This improves on the previously best local smoothing estimates due to Beltran--Hickman--Sogge \cite{BeltranHickmanSogge2020} in some range $p \leq \frac{2(d+1)}{d-1}$ for wave equations on compact manifolds. 

\medskip

\emph{Outline of the paper:} In Section \ref{section:NecessaryConditions} we show the necessary conditions for $L^p$-estimates for variable-coefficient $1$-homogeneous phases. Preliminaries for the polynomial partitioning argument to show Theorem \ref{thm:kBroadEstimate} are given in Section \ref{section:Preliminaries}. In this section we introduce the notion of a reduced homogeneous phase function and collect geometric consequences. This will simplify the proof of Theorem \ref{thm:kBroadEstimate}. We recall the wave packet analysis in the context of variable coefficients \cite{Lee2006,GuthHickmanIliopoulou2019} and collect facts on the $k$-broad norms. In Section \ref{section:PolynomialPartitioning} we recall the polynomial partitioning tools. In Section \ref{section:TransverseEquidistribution} transverse equidistribution estimates are proved. These differ from the transverse equidistribution estimates shown in \cite{GuthHickmanIliopoulou2019} for Carleson--Sj\"olin phase functions. 
In Section \ref{section:MainInductiveArgument} we deduce Theorem \ref{thm:kBroadEstimate} from Theorem \ref{thm:MainInductionTheorem}, which is suitable for induction on dimension and radius. The proof relies on polynomial partitioning and transverse equidistribution estimates play a key role. In Section \ref{section:LinearEstimates} we show how Theorem \ref{thm:kBroadEstimate} implies Theorem \ref{thm:LpLpEstimatesVariableCoefficients}. In Section \ref{section:epsRemoval} we show how the $\lambda^\varepsilon$-factor can be removed away from the endpoint. In Section \ref{section:LocalSmoothing} we apply the oscillatory integral estimates and narrow decoupling to show new local smoothing estimates for solutions to wave equations on compact Riemannian manifolds.

\medskip

\textbf{Basic Notations:}
\begin{itemize}
\item For $x,y \in \R^n$ we denote the Euclidean inner product by
\begin{equation*}
\langle x, y \rangle = x \cdot y = \sum_{i=1}^n x_i y_i.
\end{equation*}
\item $|x| = \sqrt{\langle x, x \rangle}$ denotes the Euclidean norm.
\item For a Lebesgue-measurable set $A \subseteq \R^d$, we denote by $|A|$ the $d$-dimensional Lebesgue measure.
\item For $x \in \R$, we denote Gauss brackets by
\begin{equation*}
 \lfloor x \rfloor = \max \{ k \in \Z : k \leq x \}.
 \end{equation*}
 \item For $x,y \in \R^n \backslash \{ 0 \}$ we denote the angle between $x$ and $y$ by
 \begin{equation*}
 \angle (x,y) = \arccos \frac{x \cdot y}{|x| \, |y|}.
 \end{equation*}
 \item For $x \in \R^n \backslash \{ 0 \}$ and $A \subseteq \R^n$ with $A \backslash \{0 \} \neq \emptyset$, we let
 \begin{equation*}
 \angle (x,A) = \inf_{y \in A \backslash \{0\}} \angle (x,y).
 \end{equation*}
 Similarly, we define $\angle (A_1,A_2)$ for $A_1, A_2 \subseteq \R^n$ with $A_i \backslash \{0 \} \neq \emptyset$:
 \begin{equation*}
 \angle (A_1,A_2) = \inf_{\substack{x_1 \in A_1 \backslash \{0\}, \\ x_2 \in A_2 \backslash \{0\}}} \angle (x_1,x_2).
 \end{equation*}
\end{itemize}

\section{Kakeya compression}
\label{section:NecessaryConditions}
In the following we modify the example due to Guth--Hickman--Iliopoulou \cite[Section~2]{GuthHickmanIliopoulou2019} (see also \cite{BourgainGuth2011}) for homogeneous phase functions. This yields the necessary conditions:
\begin{proposition}
\label{prop:Necessary}
Necessary for the estimate \eqref{eq:LpLpBoundVariableCoefficient} to hold for $n \geq 4$ is
\begin{equation*}
p \geq p_n.
\end{equation*}
\end{proposition}

\medskip

Recall that $p_n$ was defined in \eqref{eq:PolynomialPartitioningRange} for $n \geq 4$ by
\begin{equation*}
p_n =
\begin{cases}
&2 \cdot \frac{3n+1}{3n-3}, \quad n \text{ odd}, \\
&2 \cdot \frac{3n}{3n-4}, \quad n \text{ even}.
\end{cases}
\end{equation*}
This shows that the $L^p$-$L^p$-estimates proved in Theorem \ref{thm:LpLpEstimatesVariableCoefficients} are sharp up to endpoints.

\medskip

We only consider $n \geq 4$ because the construction in \eqref{eq:PhaseFunction} degenerates for $n=3$ to the translation-invariant case. Moreover, the range, for which the variable coefficient estimate holds, does not differ for $3 \leq n \leq 5$ from the translation-invariant case.

Let 
\begin{equation*}
x=(\underbrace{x^{\prime \prime},x_{n-1}}_{x^\prime},x_n) \in \R^{n-2} \times \R \times \R \text{ and } \omega = (\omega^\prime,\omega_{n-1}) \in \R^{n-2} \times \R.
\end{equation*}
 We consider the phase functions
\begin{equation}
\label{eq:PhaseFunction}
\phi(x;\omega) = x^\prime \cdot \omega + \frac{\langle A(x_n) \omega^\prime, \omega^\prime \rangle}{2 \omega_{n-1}}, \quad \omega_{n-1} \in (1/2,1).
\end{equation}
$A(x_n)$ denotes the $(n-2) \times (n-2)$--positive definite matrix
\begin{equation}
\label{eq:PhaseFunctionConstruction}
A(x_n) = 
\begin{cases}
\bigoplus_{i=1}^{\frac{n-2}{2}}
\begin{pmatrix}
x_n & x_n^2 \\
x_n^2 & x_n +x_n^3
\end{pmatrix}, &\quad n-2 \text{ even}, \\
\bigoplus_{i=1}^{\frac{n-3}{2}}
\begin{pmatrix}
x_n & x_n^2 \\
x_n^2 & x_n +x_n^3
\end{pmatrix}
\oplus (x_n), &\quad n-2 \text{ odd}.
\end{cases}
\end{equation}
The main idea is to construct many wave packets which are concentrated in the neighbourhood of a lower dimensional algebraic variety. Whereas the direction governed by the frequency $\omega_\theta$ below varies, for fixed $\omega_\theta$ we consider precisely one starting position $v_\theta$. This concentration in a low dimensional algebraic variety does not happen in the linear case \eqref{eq:ConeExtensionOperator}.

We consider wave packets adapted to $\phi$ as follows: $\Xi =  B^{n-2}(0,c_1) \times (1/2,1)$ is covered by essentially disjoint elongated caps 
\begin{equation*}
\Xi_{\theta} = \{ (\omega',\omega_{n-1}) \in \Xi : | \omega'/\omega_{n-1} - \omega_{\theta}| \lesssim \lambda^{-\frac{1}{2}} \}
\end{equation*}
with $\omega_{\theta} \in B^{n-2}(0,c_1)$ for $|c_1| \ll 1$. These regions will be referred to as `sectors' in the following.
 Apparently, $\Xi$ can be covered by $\sim \lambda^{\frac{ n-2}{2}}$ finitely overlapping sets $\Xi_{\theta}$. We consider a corresponding smooth partition of unity $(\psi_\theta)_{\omega_\theta \in \Xi}$ and wave packets
\begin{equation*}
f_{\theta,v}(\omega) = e^{-i \lambda \langle v, \omega^\prime \rangle} \psi_\theta(\omega), \quad v = (v_1,\ldots,v_{n-2}) \in \R^{n-2}.
\end{equation*}
We have by non-oscillation of the phase
\begin{equation*}
| T^\lambda f_{\theta,v}(x^{\prime \prime},x_{n-1},x_n) | \gtrsim \lambda^{- \frac{n-2}{2}} \chi_{T_{\theta,v}}(x).
\end{equation*}
$\chi_{T_{\theta,v}}$ denotes the characteristic function of $T_{\theta,v}$. The $T_{\theta,v}$ are curved `slabs' of size $( 1 \times  \underbrace{\lambda^{1/2} \times \ldots \times \lambda^{1/2}}_{ n-2 \text{ times}}  \times \lambda )$ with
\begin{equation*}
T_{\theta,v} \subseteq \{ x \in B(0,\lambda) \, : \, |x^{\prime \prime} - \lambda \gamma_{\theta,v} \left( \frac{x_n}{\lambda} \right) | < c_2 \lambda^{\frac{1}{2}+\varepsilon} \text{ and } |x_{n-1} - \lambda \gamma^\prime_{\theta}(x_n / \lambda) | < c_2 \lambda^{\frac{1}{2}+\varepsilon} \},
\end{equation*}
for any $\varepsilon >0$, which follows from non-stationary phase; $c_2$ denotes a small constant and $\gamma_{\theta,v}$, $\gamma_{\theta}'$ denote curves:
\begin{equation*}
\gamma_{\theta,v}(x_n) = v - A(x_n) \omega_\theta, \quad \gamma^\prime_\theta (x_n) = \frac{1}{2} \langle A(x_n) \omega_\theta, \omega_\theta \rangle.
\end{equation*}
Furthermore, note that the conditions on $\omega'$ and $\omega_{n-1}$
\begin{equation*}
\big| \frac{\omega'}{\omega_{n-1}} - \omega_\theta \big| \lesssim \lambda^{-\frac{1}{2}}, \; \omega_{n-1} \in (1/2,1), \; \omega' \in B_{n-2}(0,c_1 )
\end{equation*}
correspond to considering $\lambda^{-\frac{1}{2}}$-sectors in direction $(\omega_\theta,1)$. The degeneracy of $\partial_{\omega \omega}^2 \phi$ in the radial direction, which is immediate from the $1$-homogeneity of $\phi$, gives the localization of slabs to size $\lambda^\varepsilon$ in this direction: We have
\begin{equation*}
\partial_\omega \phi(x;\omega) \cdot \frac{(\omega_\theta,1)}{|(\omega_\theta,1)|} = \partial_\omega \phi(x;(\omega_\theta,1)) \cdot \frac{(\omega_\theta,1)}{|(\omega_\theta,1)|} + O(\lambda^{-1}) \text{ for } | \frac{\omega}{|\omega|} - \frac{(\omega_\theta,1)}{|(\omega_\theta,1)|} | \lesssim \lambda^{-\frac{1}{2}}.
\end{equation*}
The non-degeneracy of $\partial_{\omega \omega}^2 \phi$ gives localization to size less than $\lambda^{\frac{1}{2}+\varepsilon}$ in the remaining directions. We argue in the following why the curved tubes $\chi_{T_{\theta,v}}$ are in fact of size $1 \times \lambda^{\frac{1}{2}} \times \ldots \times \lambda^{\frac{1}{2}} \times \lambda$ (and not significantly less): Consider the oscillatory integral
\begin{equation*}
F(x) = \int e^{i(x' \cdot \omega + \lambda \tilde{\phi}(x_n/\lambda,\omega))} \psi_{\theta}(\omega) d\omega
\end{equation*}
with $\psi_{\theta} \in C^\infty_c(A^{n-1})$ localizing to a slab into direction $\theta \in \mathbb{S}^{n-2}$ and
\begin{equation*}
\tilde{\phi}(x_n,\mu \omega)= \mu \tilde{\phi}(x_n,\omega) \text{ for } \mu > 0.
\end{equation*}
We use Taylor expansion in $\omega$ to write
\begin{equation*}
\begin{split}
\lambda \tilde{\phi}(x_n/\lambda,\omega) &= |\omega| (\lambda \tilde{\phi}(x_n/\lambda,\omega/|\omega|)) \\
&= |\omega| (\lambda \tilde{\phi}(x_n/\lambda,\theta) + \lambda \nabla_{\omega} \tilde{\phi}(x_n/\lambda,\theta) ( \frac{\omega}{|\omega|} - \theta)) + O(c).
\end{split}
\end{equation*}
For $\omega \in \text{supp}( \psi_{\theta})$ we have
\begin{equation*}
|\omega| = \omega \cdot \theta + O(c \lambda^{-1}).
\end{equation*}
Hence, we can write
\begin{equation*}
\lambda \tilde{\phi}(x_n/\lambda,\omega) = \lambda \tilde{\phi}(x_n/\lambda,\theta)(\omega \cdot \theta) + \lambda \nabla_{\omega} \tilde{\phi}(x_n/\lambda,\theta)(\omega - (\omega \cdot \theta) \theta) + O(c).
\end{equation*}
Let $\{\theta^1_{\perp},\ldots,\theta_{\perp}^{n-2},\theta \}$ be an orthonormal basis of $\R^{n-1}$. Then,
\begin{equation*}
\lambda \tilde{\phi}(x_n/\lambda,\omega) = \lambda \tilde{\phi}(x_n/\lambda,\theta)(\omega \cdot \theta) + \lambda \sum_{i=1}^{n-2} (\nabla_\omega \tilde{\phi}(x_n/\lambda,\theta) \cdot \theta_{\perp}^i ) (\omega \cdot \theta_{\perp}^i) + O(c).
\end{equation*}
Consequently,
\begin{equation*}
\begin{split}
x' \cdot \omega + \lambda \tilde{\phi}(x_n/\lambda,\omega) &= (x' \cdot \theta + \lambda \tilde{\phi}(x_n/\lambda,\theta))(\omega \cdot \theta) \\
&\quad + \sum_{i=1}^{n-2} (x' \cdot \theta_{\perp}^i + \lambda \nabla_{\omega} \tilde{\phi}(x_n/\lambda,\theta) \cdot \theta_{\perp}^i )(\omega \cdot \theta_{\perp}^i) + O(c).
\end{split}
\end{equation*}
And for $|x' \cdot \theta + \lambda \tilde{\phi}(x_n/\lambda,\theta)| \ll c$ and $|x' \cdot \theta_{\perp}^i + \lambda \nabla_{\omega} \tilde{\phi}(x_n/\lambda,\theta) \cdot \theta_{\perp}^i| \ll c \lambda^{1/2}$, we see that the modulus of the whole phase is $O(c)$. Hence, there is no oscillation within $\text{supp}(\psi_{\theta})$ and for fixed $x_n$ this defines a region $A_{x_n}$ for $x'$ of size $1 \times \lambda^{1/2} \times \ldots \times \lambda^{1/2}$, for which $|T^\lambda f(x',x_n)|$ is roughly constant. Taking $T_{\theta} = \bigcup_{x_n} A_{x_n}$ yields the $1 \times \lambda^{1/2} \times \ldots \times \lambda^{1/2} \times \lambda$-tube. Note that the factor $e^{-i \lambda \langle v,\omega' \rangle}$ amounts to a shift in $x'$ by $\lambda v$, but does not change the size of the tube.

\medskip

\noindent We prepare the initial data with randomized signs:
\begin{equation*}
f = \sum_{\theta} \varepsilon_\theta f_{\theta,v}.
\end{equation*}
By Khintchine's theorem, the expected value of $|T^\lambda f(x)|$ is given by the square sum:
\begin{equation*}
\mathbf{E}[|T^\lambda f(x)|] \sim \big( \sum_\theta |T^\lambda f_{\theta,v_\theta} |^2 \big)^{1/2} \gtrsim \lambda^{-\frac{n-2}{2}} \big( \sum_\theta \chi_{T_{\theta,v_\theta}}(x) \big)^{1/2}.
\end{equation*}
Taking $L^p$-norms yields by Minkowski's inequality
\begin{equation*}
\lambda^{- \frac{n-2}{2}} \big( \int \big( \sum_{\theta} \chi_{T_{\theta,v_{\theta}}} \big)^{\frac{p}{2}} \big)^{\frac{1}{p}} \lesssim \mathbf{E}[ \| T^\lambda f \|_{L^p} ].
\end{equation*}
Next, we find by applying H\"older's inequality
\begin{align*}
\lambda^{-\frac{n-2}{2}} \big( \int \sum_\theta \chi_{T_{\theta,v_\theta}} \big)^{1/2} &\lesssim \left| \bigcup_{\theta} T_{\theta,v_\theta} \right|^{1/2 - 1/p} \mathbf{E}[ \| T^\lambda f \|_{L^p} ] \\
 &\lesssim \left| \bigcup_{\theta} T_{\theta,v_\theta} \right|^{1/2-1/p} \| f \|_p \lesssim \left| \bigcup_{\theta} T_{\theta,v_\theta} \right|^{1/2 - 1/p}.
\end{align*}
The penultimate estimate is by hypothesis, and the final estimate follows from $|f| = 1$ and $|\text{supp } f| \sim 1$. Since the tubes $T_{\theta,v_\theta}$ are $( 1 \times \lambda^{1/2} \times \ldots \times \lambda^{1/2} \times \lambda)$-slabs, $\int \chi_{T_{\theta, v_\theta}} \sim \lambda^{\frac{n}{2}}$. Moreover, there are about $\lambda^{\frac{n-2}{2}}$ slabs. Hence,
\begin{equation*}
\lambda^{-\frac{n-2}{2}} \big( \int \sum \chi_{T_{\theta, v_\theta}} \big)^{1/2} \sim \lambda^{\frac{1}{2}}.
\end{equation*}
Thus, we arrive at
\begin{equation}
\label{eq:EstimateLp}
1 \lesssim \left| \bigcup_{\theta} T_{\theta,v_{\theta}} \right|^{1/2 - 1/p} \lambda^{-\frac{1}{2} }.
\end{equation}
Next, we shall see how to choose $v_\theta$ such that the curved slabs are concentrated in a neighbourhood of a low-dimensional algebraic variety inspired by \cite{GuthHickmanIliopoulou2019}.\\
 For $\Xi_{\theta}$, we set
\begin{equation}
\label{eq:StartingPosition}
v_{\theta,2j-1} = - (\omega_{\theta})_{2j} \text{ and } v_{\theta,2j} = v_{\theta,n-2} = 0 \text{ for } 1 \leq j \leq \lfloor \frac{n-2}{2} \rfloor.
\end{equation}

Let $d = n-1 - \lfloor \frac{n-2}{2} \rfloor$ and $Z = Z(P_1,\ldots,P_{n-1-d})$ be the common zero set of the polynomials
\begin{equation}
\label{eq:Polynomials}
P_j(x_1,\ldots,x_{n-2},x_n) = \lambda x_{2j} - x_{2j-1} x_n \text{ for } 1 \leq j \leq \lfloor \frac{n-2}{2} \rfloor.
\end{equation}
It is straight-forward to show that the image of $x_n \mapsto (\lambda \gamma_{\theta,v_\theta}(x_n/\lambda),x_n)$ is contained in $Z(P_1,\ldots,P_{n-1-d})$. $Z$ is an algebraic variety of dimension 
\begin{equation}
\label{eq:DimensionVariety}
d=(n-1) - \lfloor \frac{n-2}{2} \rfloor
\end{equation}
in $\R^{n-1}$ and of degree $O_n(1)$. Thus, Wongkew's theorem (cf. \cite{Wongkew1993}) on the size of neighbourhoods of algebraic varieties applies, and we find
\begin{equation}
\label{eq:SizeEstimateNeighbourhood}
\tilde{N}_{\lambda^{\frac{1}{2}}} = |N_{\lambda^{\frac{1}{2}}}(Z) \cap B_{n-1}(0,\lambda) | \lesssim \lambda^{d+\frac{n-1-d}{2}}.
\end{equation}
We find by \eqref{eq:DimensionVariety} and \eqref{eq:SizeEstimateNeighbourhood}
\begin{equation}
\label{eq:EstimateVarietyII}
\tilde{N}_{\lambda^{\frac{1}{2}}} \lesssim 
\begin{cases}
\lambda^{\frac{3n-2}{4}}, \quad n \text{ even}, \\
\lambda^{\frac{3n-1}{4}}, \quad n \text{ odd.}
\end{cases}
\end{equation}
Moreover, for $(x_1,\ldots,x_n) \in T_{\theta,v_\theta}$ we have $x_{n-1} \in B(\lambda \gamma'_\theta(x_n/\lambda),\lambda^{\frac{1}{2}+\varepsilon})$.

This yields
\begin{equation}
\label{eq:EstimateTubeUnion}
\left| \bigcup_{\theta} T_{\theta,v_\theta} \right|^{1/2 - 1/p} \lesssim (\tilde{N}_{\lambda^{\frac{1}{2}}} \cdot \lambda^{\frac{1}{2}})^{\frac{1}{2} - \frac{1}{p}}.
\end{equation}

Plugging \eqref{eq:EstimateTubeUnion} into \eqref{eq:EstimateLp} with the estimate from \eqref{eq:EstimateVarietyII}, we find
\begin{equation*}
p \geq
\begin{cases}
&2 \cdot \frac{3n}{3n-4}, \quad n \text{ even}, \\
&2 \cdot \frac{3n+1}{3n-3}, \quad n \text{ odd}.
\end{cases}
\end{equation*}
This finishes the proof of Proposition \ref{prop:Necessary}. $\hfill \Box$

\section{Preliminaries}
\label{section:Preliminaries}

\subsection{Basic reductions of the phase and amplitude function}
\label{subsection:BasicReductions}
In this paragraph we shall normalize $1$-homogeneous phase functions, which satisfy $C1)$ and $C2)$, and the amplitude. This will highlight that the class of considered phase functions are indeed $C^N$-perturbations of the translation-invariant case
\begin{equation*}
\phi_*(x;\omega) = \langle x', \omega \rangle + \frac{x_n (\omega')^2}{2 \omega_{n-1}}, \quad \omega' \in B(0,c), \; \omega_{n-1} \in (1,2).
\end{equation*}
Constant-coefficient perturbations were analyzed in \cite{GaoLiuMiaoXi2023}.

\bigskip

The arguments were provided in \cite[Section~2]{BeltranHickmanSogge2020}, and details are omitted here (see also \cite{Lee2006}). It is important to work with a class of phase functions, which is stable under rescaling. After localisation and translation, we may assume that $a$ is supported inside $X \times \Xi$, where $X = X' \times T$ for $X' \subseteq B(0,1) \subseteq \R^{n-1}$ and $T \subseteq (-1,1) \subseteq \R$ are small open neighbourhoods of the origin and $\Xi \subseteq A^{n-1}$ is a small sector of dimensions $1 \times c \times \ldots \times c$ with $0<c\ll 1$ centred at $e_{n-1} = (0,\ldots,0,1) \in \R^{n-1}$.

Firstly, we can suppose that
\begin{align*}
&C1^\prime) \quad \det \partial^2_{\omega x'} \phi(x;\omega) \neq 0 \text{ for all } (x,\omega) \in X \times \Xi, \\
&C2^\prime) \quad \partial^2_{\omega^\prime \omega^\prime} \partial_{x_n} \phi(x,\omega) \text{ has eigenvalues of the same sign for all } (x,\omega) \in X \times \Xi.
\end{align*}
This follows as in \cite{BeltranHickmanSogge2020}. By rotation in the $x$-variables, we can also suppose that
\begin{equation*}
G(0;e_{n-1}) = e_n \text{ and } \partial^2_{x_n \omega} \phi(0;e_{n-1}) = 0.
\end{equation*}
Hence, by making $\Xi$ small enough, we find that
\begin{equation}
\label{eq:XiXnDerivativeBound}
|\partial^2_{x_n \omega} \phi(x;\omega)| \leq c_{cone} \text{ for } (x,\omega) \in X \times \Xi.
\end{equation}

 By non-degeneracy $C1')$ and the implicit function theorem, we find a smooth mapping $\Phi_{x_n,\omega} : X' \to \R^{n-1}$ such that
\begin{equation*}
\partial_{\omega} \phi(\Phi_{x_n,\omega}(x'),x_n;\omega) = x'.
\end{equation*}
We shall also write $\Phi_{x_n,\omega}(x') = \Phi(x',x_n;\omega)$. There is also a smooth mapping $\Psi(x,\cdot)$ with
\begin{equation*}
\partial_{x^\prime} \phi(x;\Psi(x;\omega)) = \omega.
\end{equation*}
For $\lambda \geq 1$, we consider the rescaled versions $\Phi^\lambda(x;\omega) = \lambda \Phi(x/\lambda;\omega)$ and $\Psi^\lambda(x;\omega) = \Psi(x/\lambda;\omega)$. We assume that $X$ and $\Xi$ are such that the above mappings are defined on the whole support of $a$.

In the following we shall quantify the deviation from $\phi_*$ by restricting the values of second and third derivatives and bounding higher order derivatives:
Let $c_{cone} > 0$ denote a small constant and $I_m \in \R^{m \times m}$, $(I_m)_{ij} = \delta_{ij}$ denote the unit matrix. Firstly note that there are (possibly large) constants $A_1,A_2,A_3 \geq 1$ such that
\begin{align*}
&C1^{\prime \prime}) \quad |\partial^2_{\omega x^\prime} \phi(x;\omega) - I_{n-1}| \leq c_{cone} A_1 \text{ for } (x;\omega) \in X \times \Xi,\\
&C2^{\prime \prime}) \quad |\partial^2_{\omega^\prime \omega^\prime} \partial_{x_n} \phi(x;\omega) - \frac{I_{n-2}}{\omega_{n-1}} | \leq c_{cone} A_2 \text{ for } (x;\omega) \in X \times \Xi.
\end{align*}
In the above display and in the following we abuse notation and denote the Euclidean norm of a vector $v \in \R^m$ or the Hilbert-Schmidt norm of a matrix $A \in \R^{m \times m}$ by $|v|$ or $|A|$.

For technical reasons, we shall impose in the following for the higher order derivatives:
\begin{align*}
&D1) \; \| \partial^{\beta}_\omega \partial_{x_k} \phi \|_{L^\infty(X \times \Xi)} \leq c_{cone} A_1 \text{ for } 1 \leq k \leq n-1 \text{ and } \beta = (\beta',\beta_{n-1}) \in \mathbb{N}_0^{n-1}, \\
&\qquad \text{ which satisfies } 2 \leq |\beta| \leq 3 \text{ and } |\beta^\prime| \geq 2; \\
&\qquad \| \partial^{\beta}_{\omega} \partial_{x_n} \phi \|_{L^\infty(X \times \Xi)} \leq \frac{c_{cone}}{2n} A_1 \text{ for all } \beta \in \mathbb{N}_0^{n-1} \text{ with } |\beta^\prime| \geq 3 \text{ and } |\beta| \leq N. \\
&D2) \; \text{For some large integer } N \in \mathbb{N}, \text{ one has} \\
&\qquad \| \partial_\omega^\beta \partial_x^\alpha \phi \|_{L^\infty(X \times \Xi)} \leq \frac{c_{cone}}{2n} A_3 \text{ for all } (\alpha,\beta) \in \mathbb{N}_0^{n} \times \mathbb{N}_0^{n-1} \text{ with } 2 \leq |\alpha| \leq 4N \\
&\qquad \text{and } 1 \leq |\beta| \leq 4N+2 \text{ satisfying } 1 \leq |\beta| \leq 4N \text{ or } |\beta^\prime| \geq 2.
\end{align*}

By parabolic rescaling (cf. Lemma \ref{lem:ParabolicRescaling}), we see that we can reduce to phases with $A_i = 1$; these phases are said to be reduced.

\medskip

Moreover, we can suppose that uniform bounds hold for the amplitude $a \in C^\infty_c(X \times \Xi)$:
\begin{equation}
\label{eq:UniformBoundsAmplitude}
\| \partial_\omega^\alpha a \|_{L^\infty(Z \times \Xi)} \leq C_{\text{amp}} \text{ for } 0 \leq |\alpha| \leq N.
\end{equation}
$C_{\text{amp}}$ denotes a universal constant. \eqref{eq:UniformBoundsAmplitude} is accomplished by dividing $a$ through a large constant depending on $a$ and its derivatives.
Furthermore, we can suppose by Fourier series expansion (see \cite{BeltranHickmanSogge2020}) that $a$ is of product form 
\begin{equation*}
a(x,\omega) = a_1(x) a_2(\omega).
\end{equation*}
For the spatial part, we require a \emph{margin condition}:
\begin{equation}
\label{eq:MarginCondition}
\text{dist} (\text{supp } a_1, \R^{n+1} \backslash X) \geq \frac{1}{4 A_3}.
\end{equation}
This will become useful for variable-coefficient decoupling in Section \ref{section:LinearEstimates} and can always be achieved by finite decomposition and re-centering (cf. \cite{BeltranHickmanSogge2020}). An amplitude $a$ is said to be reduced, if it is of product form, satisfies \eqref{eq:MarginCondition} with $A_3 = 1$ and uniform bounds on the derivatives \eqref{eq:UniformBoundsAmplitude}.

\begin{definition}
\label{def:TypeData}
We say that $(\phi,a)$ are of type $(A_1,A_2,A_3)$, if
\begin{itemize}
\item $\phi$ satisfies $C1'')$, $C2'')$, $D1)$, $D2)$, and $a$ satisfies \eqref{eq:MarginCondition},
\item $a$ is of product form and satisfies \eqref{eq:UniformBoundsAmplitude}.
\end{itemize}

Moreover, $(\phi,a)$ is said to be reduced if $\phi$ is a reduced phase function and $a$ is a reduced amplitude.
\end{definition}

\medskip

The following observation will be useful at a later point:
\begin{lemma}
\label{lem:ReducedPhaseTimeDerivatives}
Suppose that $\phi$ is a reduced phase function. Then we have
\begin{equation}
\label{eq:ReducedPhaseTimeDerivatives}
\| \partial_{x_n} \partial^\beta_\omega \phi \|_{L^\infty(X \times \Xi)} \lesssim_{N,c_{\text{cone}}} 1 \text{ for } 1 \leq | \beta | \leq N. 
\end{equation}
\end{lemma}

\begin{proof}
For $\beta = (\beta',\beta_{n-1}) \in \N_0^{n-1}$ with $|\beta'| \geq 3$ this is covered by $D1)$ and for $|\beta| = 1$ this is $C2'')$. So it remains to show \eqref{eq:ReducedPhaseTimeDerivatives} for $|\beta|=2$, and $|\beta| \geq 3$ with $|\beta'| \leq 2$:
We obtain by homogeneity
\begin{equation*}
\partial_{x_n} \phi(x;\omega',\omega_{n-1}) = \omega_{n-1} \partial_{x_n} \phi(x;\omega'/\omega_{n-1},1).
\end{equation*}
From this follows
\begin{equation}
\label{eq:TimeDerivativesI}
\partial^2_{x_n \omega_{n-1}} \phi(x;\omega',\omega_{n-1}) = \partial_{x_n} \phi(x;\omega' / \omega_{n-1},1) - \partial^2_{x_n \omega'} \phi(x;\omega'/\omega_{n-1},1) \cdot \frac{\omega'}{\omega_{n-1}}.
\end{equation}
And moreover,
\begin{equation}
\label{eq:TimeDerivativesII}
\partial^3_{x_n, \, \omega_{n-1}, \, \omega_{n-1}} \phi = \frac{1}{\omega_{n-1}} \langle \partial^3_{x_n, \, \omega', \, \omega'} \phi(x;\omega'/\omega_{n-1},1) \frac{\omega'}{\omega_{n-1}}, \frac{\omega'}{\omega_{n-1}} \rangle .
\end{equation}
Hence, by $C2'')$ this is uniformly bounded.

For $|\beta| = 2$ and $|\beta'| = 1$, the claim follows from taking an additional derivative in \eqref{eq:TimeDerivativesI} and estimates \eqref{eq:XiXnDerivativeBound} and $C1'')$.

For $3 \leq |\beta| \leq N$ and $|\beta'| \leq 1$, the claim is evident from taking derivatives in \eqref{eq:TimeDerivativesII}, and the estimates $C2'')$ and $D1)$.
\end{proof}

\subsection{Geometric consequences}
Let $\phi$ be a reduced phase function in the above sense. We shall see how the corresponding hypersurfaces $\Sigma_x$ parametrized by $\omega \mapsto \partial_x \phi(x;\omega)$ resemble the ones from $\phi_*$. To see this, recall that $\Psi: U \rightarrow \Xi$ satisfies $\partial_{x^\prime} \phi(x;\Psi(x;\omega)) = \omega$. Hence, $\Sigma_x$ is the graph of the function $h_x(\omega) = \partial_{x_n} \phi(x;\Psi(x;\omega))$ over the fibre $U_x$.\\
Each $h_x$ is a perturbation of the translation invariant case in the following sense:
\begin{lemma}
The following estimate holds for all $\omega \in U_x$:
\begin{equation}
\label{eq:HPerturbation}
\| \partial^2_{\omega^\prime \omega^\prime} h_x(\omega) - I_{n-1}/ \omega_{n-1} \|_{L^\infty} \lesssim c_{cone}.
\end{equation}
Here $c_{cone}>0$ denotes the constant from the definition of a reduced phase function.
\end{lemma}
\begin{proof}
This is a consequence of properties of $\Psi$. Firstly, we record that $\Psi(x;e_{n-1}) = 1$. By the implicit function theorem and non-degeneracy of $\phi$, we find
\begin{equation*}
\partial_\omega \Psi(x;\omega) = \partial^2_{x^\prime \omega} \phi(x;\Psi(x;\omega))^{-1}.
\end{equation*}
Hence,
\begin{equation*}
\| \partial_\omega \Psi(x;\omega) - I_{n-1} \|_{L^\infty} = O(c_{cone}).
\end{equation*}
As a consequence of this identity (and choosing $c_{cone}$ to be sufficiently small),
\begin{equation*}
|\Psi(x;\omega) - \Psi(x;\omega^\prime)| \sim |\omega - \omega^\prime| \text{ for all } \omega, \omega^\prime \in U_x
\end{equation*}
with implicit constant only depending on $n$.\\
Additionally, if $1 \leq k \leq n-1$, then by twice differentiating the identity
\begin{equation*}
\partial_{x_k} \phi(x;\Psi(x;\omega)) = \omega_k
\end{equation*}
in the $\omega$-variables, it follows that
\begin{equation*}
\| \partial^2_{\omega \omega} \Psi_k(x;\omega) \|_{L^\infty} = O(c_{cone}).
\end{equation*}
By the previous estimate, \eqref{eq:HPerturbation} follows from $C2'')$.
\end{proof}
By similar means, we infer estimates for the generalized Gauss map associated with $T^\lambda$. To give the results, let
\begin{equation*}
X^\lambda = \{ x \in \R^n \; | \; \frac{x}{\lambda} \in X \}
\end{equation*}
denote the $\lambda$-dilate of $X$, so that $a^\lambda$ is supported in $X^\lambda \times \Xi$.
\begin{lemma}
\label{lem:ConsequencesReductionGaussMap}
For all $x, \bar{x} \in X^\lambda$ and $\omega, \bar{\omega} \in \Xi$, the estimates
\begin{equation}
\label{eq:EstimatesReducedGaussMap}
\begin{split}
\angle(G^\lambda(x;\omega),G^\lambda(x;\bar{\omega})) &\sim |\frac{\omega}{|\omega|} - \frac{\bar{\omega}}{|\bar{\omega}|}| \sim \angle (\omega, \bar{\omega}), \\
\angle(G^\lambda(x;\omega),G^\lambda(\bar{x};\omega)) &\lesssim \lambda^{-1} |x-\bar{x}|
\end{split}
\end{equation}
hold true.
\end{lemma}
This will be helpful to understand the wave packet analysis in the following sections.
\subsection{Wave packet decomposition}
\label{subsection:WavePacket}
We carry out the wave packet decomposition with respect to some spatial parameter $1 \ll R \ll \lambda$. For this purpose, we follow \cite{GuthHickmanIliopoulou2019} and use that the construction only depends on the non-degeneracy condition $C1)$. We do not use the usual wave packet decomposition for the cone as e.g. in \cite{OuWang2022}, but adapt the parabolic case, as previously done by Lee \cite{Lee2006}. The reason is that in Section \ref{section:MainInductiveArgument} we would sort the smaller cone tubes into larger tubes anyway. It appears that the present choice of wave packet decomposition allows us to transfer arguments from \cite{GuthHickmanIliopoulou2019} to the homogeneous setting more directly. In the following we introduce notations from \cite{GuthHickmanIliopoulou2019}.

Cover $A^{n-1}$ by finitely overlapping balls of radius $R^{-1/2}$, and let $\psi_\theta$ be a smooth partition of unity adapted to this cover. These $\theta$ will frequently be referred to as $R^{-1/2}$-balls. For a ball $\theta$, cover $\R^{n-1}$ with finitely overlapping balls of size $R^{\frac{1+\delta}{2}} \times \ldots \times R^{\frac{1+\delta}{2}}$ with center $v \in R^{\frac{1+\delta}{2}} \Z^{n-1}$. Let $\eta_v=\eta(\cdot - v)$ denote a bump function adapted to $B(v,R^{\frac{1+\delta}{2}})$ such that
\begin{equation*}
\sum_{v \in \Z^{n-1}} \eta_v = 1
\end{equation*}
with $\hat{\eta}_v$ essentially supported in $B(0,C R^{-\frac{1+\delta}{2}})$. This is possible by the Poisson summation formula.

 Let $\mathbb{T}$ denote the collection of all pairs $(\theta,v)$. Then, for $f:\R^{n-1} \to \C$ with support in $A^{n-1}$ and sufficiently regular, we find
\begin{equation*}
f = \sum_{(\theta,v) \in \mathbb{T}} (\eta_v(\psi_\theta f)\check{\;})\hat{\;} = \sum_{(\theta,v) \in \mathbb{T}} \hat{\eta}_v * (\psi_\theta f).
\end{equation*}
For each $R^{-1/2}$-ball $\theta$, let $\omega_\theta$ denote its centre. Choose a real-valued smooth function $\tilde{\psi}$ so that $\tilde{\psi}_\theta$ is supported in $\theta$, and $\tilde{\psi}_\theta(\omega) = 1$ whenever $\omega$ belongs to a $cR^{-1/2}$-neighbourhood of the support of $\psi_\theta$ for some small $c>0$. Finally, define
\begin{equation*}
f_{\theta,v} = \tilde{\psi}_\theta \cdot [ \hat{\eta}_v * (\psi_\theta f)].
\end{equation*}
The function $\hat{\eta}_v$ is rapidly decaying outside $B(0,CR^{-\frac{1+\delta}{2}})$ and, consequently,
\begin{equation*}
\| f_{\theta,v} - (\hat{\eta}_v * (\psi_\theta f)) \|_{L^\infty(\R^{n-1})} \leq \text{RapDec}(R) \| f \|_{L^2(A^{n-1})}.
\end{equation*}
The functions $f_{\theta,v}$ are almost orthogonal: if $\mathbb{S} \subseteq \mathbb{T}$, then
\begin{equation*}
\Vert \sum_{(\theta,v) \in \mathbb{S}} f_{\theta,v} \|^2_{L^2(\R^{n-1})} \sim \sum_{(\theta,v) \in \mathbb{S}} \| f_{\theta,v} \|^2_{L^2(\R^{n-1})}.
\end{equation*}
Let $T^\lambda$ be an oscillatory integral operator with reduced phase $\phi$ satisfying $C1')$ and amplitude $a$ supported in $X \times \Xi$. For $(\theta,v) \in \mathbb{T}$ define the curve $\gamma^1_{\theta,v}: I^1_{\theta,v} \to \R^{n-1}$ by setting $\gamma^1_{\theta,v}(t)=\Phi(v,t;\omega_\theta)$, where $\Phi$ is the function introduced in Subsection \ref{subsection:BasicReductions} and
\begin{equation*}
I^1_{\theta,v} = \{ x_n \in T \, | \, \partial_\omega \phi(x^\prime,x_n;\omega_\theta) = v \text{ for some } x^\prime \in X^\prime \}.
\end{equation*}
Hence, $\partial_\omega \phi(\gamma^1_{\theta,v}(x_n),x_n;\omega_\theta) = v$ for all $x_n \in I^1_{\theta,v}$. For the rescaled curve
\begin{equation*}
\gamma^\lambda_{\theta,v}(t) = \lambda \gamma^1_{\theta,v/\lambda}(t/\lambda),
\end{equation*}
we find
\begin{equation*}
\partial_\omega \phi^\lambda(\gamma^\lambda_{\theta,v}(x_n),x_n;\omega_\theta) = v \text{ for all } t \in I^\lambda_{\theta,v} = \{t \in \R \, : \, \frac{t}{\lambda} \in I^1_{\theta,v} \}.
\end{equation*}
Let $\Gamma^\lambda_{\theta,v}: I^\lambda_{\theta,v} \to \R^n$ denote the graph mapping $\Gamma^\lambda_{\theta,v}(x_n) = (\gamma^\lambda_{\theta,v}(x_n),x_n)$; for the sake of brevity, the image of this mapping is denoted by $\Gamma^\lambda_{\theta,v}$, too.
\begin{lemma}[{\cite[Lemma~5.2]{GuthHickmanIliopoulou2019}}]
The tangent space $T_{\Gamma^\lambda_{\theta,v}(x_n)} \Gamma^\lambda_{\theta,v}$ lies in the direction of the unit vector $G^\lambda(\Gamma^\lambda_{\theta,v}(x_n);\omega_\theta)$ for all $x_n \in I^\lambda_{\theta,v}$.
\end{lemma}
We consider curved tubes
\begin{equation*}
T_{\theta,v} = \{(x^\prime,x_n) \in B(0,R) \, : \, x_n \in I^\lambda_{\theta,v} \text{ and } x^\prime \in B(\gamma^\lambda_{\theta,v}(x_n),R^{\frac{1}{2}+\delta}) \}.
\end{equation*}
We refer to the curve $\Gamma^\lambda_{\theta,v}$ as the core of $T_{\theta,v}$. Since $\phi$ is of reduced form, we find by the diffeomorphism property of $\Phi$ (writing $x' = \Phi^{-1}_{x_n,\omega_\theta} \circ \Phi_{x_n,\omega_\theta} (x')$)
\begin{equation*}
|x^\prime - \gamma^\lambda_{\theta,v}| \sim |\partial_\omega \phi^\lambda(x;\omega_\theta) - v|,
\end{equation*}
for all $x = (x^\prime,x_n) \in X_\lambda$ with $x_n \in I^\lambda_{\theta,v}$ uniformly in $\lambda$. This has the following consequence:

\begin{lemma}[{\cite[Lemma~5.4]{GuthHickmanIliopoulou2019}}]
If $1 \ll R \ll \lambda$ and $x \in B(0,R) \backslash T_{\theta,v}$, then
\begin{equation*}
|T^\lambda f_{\theta,v}(x)| \leq (1+ R^{-1/2} |\partial_\omega \phi^\lambda(x;\omega_\theta) -v |)^{-(n+1)} \text{RapDec}(R) \| f \|_{L^2(A^{n-1})}.
\end{equation*}
\end{lemma}

\subsection{$L^2$-$L^2$-estimate}

We recall the following generalization of Parseval's theorem, only depending on non-degeneracy $C1')$ of the phase function (cf. \cite[Section~2.1]{Sogge2017}):
\begin{lemma}[{\cite[Lemma~5.5]{GuthHickmanIliopoulou2019}}]
\label{lem:Hoermander}
If $1 \leq R \leq \lambda$ and $B_R$ is any ball of radius $R$, then
\begin{equation}
\label{eq:L2Estimate}
\| T^\lambda f \|_{L^2(B_R)} \lesssim R^{1/2} \| f \|_{L^2(A^{n-1})}.
\end{equation}
\end{lemma}
This is based on the following estimate:
\begin{lemma}[{\cite[Lemma~5.6]{GuthHickmanIliopoulou2019}}]
For any fixed $x_n \in \R$, we find the estimate
\begin{equation}
\label{eq:FixedTimeL2Estimate}
\| T^\lambda f \|_{L^2(\R^{n-1} \times \{ x_n \})} \lesssim \| f \|_{L^2(A^{n-1})} .
\end{equation}
\end{lemma}

\subsection{$k$-broad norms}
\label{subsection:kBroadNorms}
Here we recall basic properties of the $k$-broad norms. Although the naming is misleading as $k$-broad norms are, strictly speaking, no norms, the properties are similar enough to make the following arguments work. We shall also see that $U \mapsto \| T^\lambda f \|^p_{BL^p_{k,A}(U)}$ behaves as a measure.

\begin{lemma}[Finite (sub-)additivity, {\cite[Lemma~6.1]{GuthHickmanIliopoulou2019}}]
Let $U_1,U_2 \subseteq \R^n$ and $U = U_1 \cup U_2$. If $1 \leq p < \infty$ and $A$ is a non-negative integer, then
\begin{equation}
\label{eq:SubadditivityBroadNorm}
\| T^\lambda f \|^p_{BL^p_{k,A}(U)} \leq \| T^\lambda f \|^p_{BL^p_{k,A}(U_1)} + \| T^\lambda f \|^p_{BL^p_{k,A}(U_2)}
\end{equation}
holds for all integrable $f: A^{n-1} \to \C$.
\end{lemma}
Secondly, we have the following variant of the triangle inequality:
\begin{lemma}[Triangle inequality, {\cite[Lemma~6.2]{GuthHickmanIliopoulou2019}}]
If $U \subseteq \R^n$, $1 \leq p < \infty$ and $A = A_1 + A_2$ for $A_1$ and $A_2$ non-negative integers, then
\begin{equation}
\| T^\lambda(f_1 + f_2) \|_{BL^p_{k,A}(U)} \lesssim \| T^\lambda f_1 \|_{BL^p_{k,A_1}(U)} + \| T^\lambda f_2 \|_{BL^p_{k,A_2}(U)}
\end{equation}
holds for all integrable $f_1,f_2: A^{n-1} \to \C$.
\end{lemma}
We further have the following variant of H\"older's inequality:
\begin{lemma}[Logarithmic convexity, {\cite[Lemma~6.3]{GuthHickmanIliopoulou2019}}]
Suppose that $U \subseteq \R^n$, $1 \leq p,p_1,p_2 < \infty$ and $0 \leq \alpha_1,\alpha_2 \leq 1$ satisfy $\alpha_1 + \alpha_2 = 1$ and
\begin{equation*}
\frac{1}{p} = \frac{\alpha_1}{p_1} + \frac{\alpha_2}{p_2}.
\end{equation*}
If $A = A_1 + A_2$ for $A_1$, $A_2$ non-negative integers, then
\begin{equation*}
\| T^\lambda f \|_{BL^p_{k,A}(U)} \lesssim \| T^\lambda f \|^{\alpha_1}_{BL^{p_1}_{k,A_1}(U)} \| T^\lambda f \|^{\alpha_2}_{BL^{p_2}_{k,A_2}(U)}.
\end{equation*}
\end{lemma}
Later on, we shall only consider $A \gg 1$, which allows us to use H\"older's and Minkowski's inequality for $k$-broad norms.

\subsection{Overview of parameters}

For the reader's convenience, we provide an overview of the parameters to be used in Sections \ref{section:Preliminaries}-\ref{section:LinearEstimates}:
\begin{itemize}
\item $0<\varepsilon\ll 1$ denotes the parameter, for which we aim to prove the estimates in Theorem  \ref{thm:LpLpEstimatesVariableCoefficients} and \ref{thm:kBroadEstimate},
\item $m$ will denote the dimension of the algebraic variety, which is used for the polynomial partitioning (see Section \ref{section:PolynomialPartitioning}),
\item we have the hierarchy of parameters:
\begin{equation*}
\delta \ll \delta_n \ll \ldots \delta_1 \ll \varepsilon
\end{equation*}
with $\delta_i = \delta_i(\varepsilon)$ and $\delta = \delta(\varepsilon)$. The parameters $\delta_i$ will quantify tangentiality of wave packets with respect to a variety of dimension $i$ and will be specified in the iteration in Section \ref{section:MainInductiveArgument}. 
\item $\delta$ is the parameter from wave packet decomposition and eventually, $\delta = \delta(\varepsilon)$, and $N=N(\delta)=N(\varepsilon)$ the parameter $N$ for reduced data will be chosen only depending on $\varepsilon$.
\item In Section \ref{section:LinearEstimates} we need the derivative bounds for reduced phase functions up to $N$ for a possibly larger value of $N$ depending on a parameter $\delta_1=\delta_1(\varepsilon)$ used in a narrow decoupling estimate.
\end{itemize}

\section{Polynomial partitioning}
\label{section:PolynomialPartitioning}
A key tool in the proof will be polynomial partitioning following previous work by Guth \cite{Guth2016,Guth2018} (see also Guth--Katz \cite{GuthKatz2015}) and in the variable coefficient case Guth--Hickman--Iliopoulou \cite{GuthHickmanIliopoulou2019}. The idea is to divide the ball $B_R$ by the zero set of a polynomial into cells, which equidistribute the broad norm. Either $\mu_{T^\lambda f}$ will be concentrated in the cells or at the wall, i.e., an appropriate neighbourhood of the zero locus of the polynomial. Both cases will be handled by induction. We recall some facts from \cite{GuthHickmanIliopoulou2019}, which we will use in the following.

\subsection{Tools from algebraic geometry}
Given a polynomial $P$ in $\R^n$, its zero set is denoted by $Z(P)$. To make the varieties $Z(P_1,\ldots,P_{n-m})$ smooth $m$-dimensional manifolds, we consider transverse complete intersections:
\begin{definition}
Let $m \in \N$, $ m \leq n$, and let $P_1,\ldots,P_{n-m}$ be polynomials on $\R^n$ whose common zero set is denoted by $Z(P_1,\ldots,P_{n-m})$. The variety $Z(P_1,\ldots,P_{n-m})$ is called a transverse complete intersection if
\begin{equation*}
\nabla P_1(x) \wedge \ldots \wedge \nabla P_{n-m}(x) \neq 0 \qquad \forall x \in Z(P_1,\ldots,P_{n-m}).
\end{equation*}
The degree of the transverse complete intersection $\overline{\deg} Z$ is defined as \\
 $\max_{j=1,\ldots,n-m} \deg P_j$.
\end{definition}
 We have the following partitioning argument:
\begin{theorem}[{\cite[Theorem~7.3]{GuthHickmanIliopoulou2019}}]
\label{thm:EquidistributionPolynomialPartitioning}
Suppose that $W \geq 0$ is a non-zero $L^1$-function on $\R^n$. Then, for any degree $D \in \N$, there exists a non-zero polynomial $P$ of degree $\deg P \lesssim D$ such that the following holds:
\begin{itemize}
\item The set $Z(P)$ is a finite union of $\lesssim \log D$ transverse complete intersections.
\item If $(O_i)_{i \in \mathcal{I}}$ denotes the set of connected components of $\R^n \backslash Z(P)$, then $\# \mathcal{I} \lesssim D^n$ and
\begin{equation}
\label{eq:EquidistributionPartitioning}
\int_{O_i} W \sim D^{-n} \int_{\R^n} W \text{ for all } i \in \mathcal{I}.
\end{equation} 
\end{itemize}
\end{theorem}
The connected components are called \emph{cells}.

\medskip

We further need the following lemma on transverse intersections of tubes with varieties:
\begin{lemma}[{\cite[Lemma~5.7]{Guth2018}}]
\label{lem:TransverseIntersectionsStraightLines}
Let $T$ be a cylinder of radius $r$ with central line $\ell$ and suppose that $Z=Z(P_1,\ldots,P_{n-m}) \subseteq \R^n$ is a transverse complete intersection, where the polynomials $P_j$ have degree at most $D$. For $\alpha > 0$, let
\begin{equation*}
Z_{> \alpha} = \{ z \in Z: \angle(T_z Z, \ell) > \alpha \}.
\end{equation*}
Then $Z_{> \alpha} \cap T$ is contained in a union of $\lesssim D^n$ balls of radius $\lesssim r \alpha^{-1}$.
\end{lemma}

For the application, we are interested in $r = R^{(1+\delta)/2}$, as this will be the radius of the (thin) tubes and $\alpha = R^{-\frac{1}{2}+\delta}$. 
\subsection{Polynomial approximation}
\label{subsection:PolynomialApproximation}
However, with smooth core curves, Lemma \ref{lem:TransverseIntersectionsStraightLines} is not applicable directly. We approximate the core curves by polynomials such that algebraic methods can still be applied to the curved tubes. We follow \cite[Section~7.2]{GuthHickmanIliopoulou2019}. Let $\varepsilon > 0$ be a small parameter and let $N=N_\varepsilon:= \lceil 1/(2 \varepsilon) \rceil \in \N$. Suppose that $\Gamma:(-1,1) \to \R^n$ is a smooth curve with
\begin{equation*}
\| \Gamma \|_{C^{N+1}(-1,1)} = \max_{0 \leq k \leq N+1} \sup_{|t| < 1} | \Gamma^{(k)}(t)| \lesssim 1.
\end{equation*}
After the reductions of Section \ref{subsection:BasicReductions}, we find the following estimates:
\begin{lemma}
\label{lem:PolynomialApproximation}
The curves $\Gamma^1_{\theta,v}$ satisfy
\begin{equation*}
|(\Gamma^1_{\theta,v})'(t)| \sim 1 \text{ for all } t \in I^1_{\theta,v},
\end{equation*}
and
\begin{equation*}
\sup_{t \in I^1_{\theta,v}} | ( \Gamma^1_{\theta,v})^{(k)}(t)| \lesssim c_{par} \text{ for } 2 \leq k \leq N.
\end{equation*}
\end{lemma}
The proof from \cite[Lemma~7.4]{GuthHickmanIliopoulou2019} applies verbatim because of bounds \eqref{eq:XiXnDerivativeBound} and $D2)$ from Subsection \ref{subsection:BasicReductions} although the phase functions are from different classes.

We denote by $[\Gamma]_\varepsilon: \R \to \R^n$ the polynomial curve given by the degree-$N$ Taylor approximation of $\Gamma$ around zero. Observe that
\begin{equation*}
\| [ \Gamma ]_\varepsilon \|_{C^\infty(-2,2)} \leq e^2 \| \Gamma \|_{C^N(-1,1)} \lesssim 1.
\end{equation*}
Furthermore, for $\lambda \gg 1$, noting that $\lambda^{-\varepsilon N} \leq \lambda^{-1/2}$, Taylor's theorem yields
\begin{equation*}
| \Gamma^{(i)}(t) - [ \Gamma ]^{(i)}_\varepsilon(t) | \lesssim_\varepsilon \lambda^{-\frac{1}{2}} |t|^{1-i} \text{ for all } |t| \lesssim_\varepsilon \lambda^{-\varepsilon} \text{ and } i = 0,1.
\end{equation*}
Letting $\Gamma^\lambda:(-\lambda,\lambda) \to \R^n$ denote the rescaled curve $\Gamma^\lambda(t) = \lambda \Gamma(t/\lambda)$, the above inequalities imply that
\begin{equation}
\label{eq:RescaledBoundsPolynomialApproximation}
\| [ \Gamma^\lambda]'_\varepsilon \|_{C^\infty(-2\lambda,2\lambda)} \lesssim 1 \text{ and } \| [ \Gamma^\lambda]''_\varepsilon \|_{C^\infty(-2\lambda,2\lambda)} \lesssim \lambda^{-1},
\end{equation}
and
\begin{equation*}
| (\Gamma^\lambda)^{(i)}(t) - ([ \Gamma^\lambda]_\varepsilon )^{(i)}(t) | \lesssim_\varepsilon \lambda^{-\frac{1}{2}} |t|^{1-i} \text{ for all } |t| \lesssim_\varepsilon \lambda^{1-\varepsilon} \text{ and } i=0,1.
\end{equation*}
As a consequence of $|(\Gamma^\lambda)'(t)| \sim |[\Gamma^\lambda]'_\varepsilon(t) | \sim 1$, the tangent spaces to the curves $\Gamma^\lambda$ and $[\Gamma^\lambda]_\varepsilon$ have a small angular separation, i.e.,
\begin{equation}
\label{eq:TangentSpaceApproximant}
\angle (T_{\Gamma^\lambda(t)} \Gamma^\lambda, T_{[\Gamma^\lambda]_\varepsilon(t)} [ \Gamma^\lambda]_\varepsilon ) \lesssim_\varepsilon \lambda^{-\frac{1}{2}} \text{ for all } |t| \lesssim_\varepsilon \lambda^{1-\varepsilon}.
\end{equation}
\subsection{Transverse interactions between curved tubes and varieties}
We have the following generalization of Lemma \ref{lem:TransverseIntersectionsStraightLines}:
\begin{lemma}[{\cite[Lemma~7.5]{GuthHickmanIliopoulou2019}}]
\label{lem:TransverseIntersectionPolynomialApproximant}
Let $n \geq 2$, $1 \leq m \leq n$ and $Z=Z(P_1,\ldots,P_{n-m}) \subseteq \R^n$ be a transverse complete intersection. Suppose that $\Gamma:\R \to \R^n$ is a polynomial graph satisfying
\begin{equation}
\label{eq:DerivativeConditionsPolynomialGraph}
\| \Gamma' \|_{L^\infty(-2\lambda,2\lambda)} \lesssim 1 \text{ and } \| \Gamma'' \|_{L^\infty(-2\lambda,2\lambda)} \leq \delta
\end{equation}
for some $\lambda,\delta>0$. There exists a dimensional constant $\bar{C} > 0$ such that, for all $\alpha > 0$ and $0<r<\lambda$ satisfying $\alpha \geq \bar{C} \delta r$, the set $Z_{>\alpha,r,\Gamma} \cap B(0,\lambda)$ is contained in a union of
\begin{equation*}
O((\overline{\deg} Z \cdot \deg \Gamma)^n)
\end{equation*}
balls of radius $r/\alpha$.
\end{lemma}
This will be used when proving the $k$-broad estimate via polynomial partitioning.

\section{Transverse equidistribution estimates}
\label{section:TransverseEquidistribution}

\subsection{Outline of the argument.}
\label{subsection:OutlineTransverseEquidistribution}
In this section transverse equidistribution estimates for wave packets tangential to varieties will be examined. Functions having wave packets tangential to a variety arise in the \emph{algebraic case} when applying polynomial partitioning in the main induction argument. Contrary to \cite{OuWang2022} or \cite{GaoLiuMiaoXi2023}, however, we stick to the wave packet decomposition used in \cite{GuthHickmanIliopoulou2019}. We make the following definition (see \cite{GuthHickmanIliopoulou2019,Guth2018}):
\begin{definition}
Let $Z = Z(P_1,\ldots,P_{n-m})$ be a transverse complete intersection. A wave packet $(\theta,v)$ is said to be $R^{-\frac{1}{2}+\delta_m}$-tangent to $Z$ in $B(0,R)$ if
\begin{equation}
\label{eq:TangentialConditionI}
T_{\theta,v} \cap B_R \subseteq N_{R^{\frac{1}{2} + \delta_m}}(Z)
\end{equation}
and
\begin{equation}
\label{eq:TangentialConditionII}
\angle(G^\lambda(x;\omega_\theta), T_z Z) \leq \bar{c}_{tang} R^{-\frac{1}{2}+\delta_m}
\end{equation}
for any $x \in T_{\theta,v}$ and $z \in Z \cap B(0,2R)$ with $|x-z| \leq \bar{C}_{tang} R^{\frac{1}{2}+\delta_m}$.
\end{definition}
We want to study functions concentrated on the collection of wave packets
\begin{equation*}
\mathbb{T}_Z = \{ (\theta,v) \in \mathbb{T} : T_{\theta,v} \text{ in } R^{-\frac{1}{2}+\delta_m}-\text{tangent to } Z \text{ in } B(0,R)\}.
\end{equation*}
Precisely, we make the following definition:
\begin{definition}
If $\mathbb{S} \subseteq \mathbb{T}$, then $f$ is said to be concentrated on wave packets from $\mathbb{S}$ if
\begin{equation*}
f = \sum_{(\theta,v) \in \mathbb{S}} f_{\theta,v} + \text{RapDec}(R) \| f \|_{L^2}.
\end{equation*}
\end{definition}
Let $B \subseteq \R^n$ be a ball of radius $R^{\frac{1}{2}+\delta_m}$ with centre $\bar{x} \in B(0,R)$. We study $\eta_B \cdot T^\lambda g$, where $\eta_B$ is a suitable choice of Schwartz function adapted to $B$. A stationary phase argument yields that the Fourier transform of $ \eta_B \cdot T^\lambda g_{\theta,v}$ is concentrated near the surface
$\Sigma = \{ \Sigma(\omega) \, : \omega \in A^{n-1} \}$, where $\Sigma(\omega) = \partial_x \phi^\lambda(\bar{x};\omega)$. This leads to the refined set of wave packets
\begin{equation*}
\mathbb{T}_{Z,B} = \{(\theta,v) \in \mathbb{T}_Z : T_{\theta,v} \cap B \neq \emptyset \}.
\end{equation*}
For $(\theta,v) \in \mathbb{T}_{Z,B}$, the direction $G^\lambda(\bar{x};\omega_\theta)$ of $T_{\theta,v}$ must make a small angle with each of the tangent spaces $T_z Z$ for all $z \in Z \cap B$. This constrains $\Sigma(\omega_\theta)$ to lie in a small neighbourhood of some typically $m$-dimensional manifold $S_\xi$. But in the homogeneous case, $S_\xi$ might only be one-dimensional, or ``close" to a one-dimensional manifold. This will be quantified below. We refer to this as \emph{case of narrow space-time frequencies}. This case does not contribute in the broad norm, for which reason it is referred to as \emph{narrow}. The role of space-time frequencies is emphasized, as in the analysis we are rather considering sectors on the generalized cone (and not the projection to $\R^{n-1}$).

To linearize $S_\xi$, if it is not a ``narrow", essentially one-dimensional set, let $R^{\frac{1}{2}} < \rho \ll R$ and for the remainder of this section, let $\tau \subseteq A^{n-1}$ denote a sector of aperture $O(\rho^{-\frac{1}{2}+\delta_m})$. We define
\begin{equation*}
\mathbb{T}_{Z,B,\tau} = \{ (\theta,v) \in \mathbb{T}_Z : \theta \cap \tau \neq \emptyset \text{ and } T_{\theta,v} \cap B \neq \emptyset \}.
\end{equation*}

\medskip

We recall the constant-coefficient case. Suppose that $Z$ is an $m$-dimensional affine plane so that $T_z Z = V$ for all $z \in Z$, where $V \parallel Z$. 

The extension operator for the cone has the unnormalized Gauss map $G_0(\omega) = \big( - \frac{\omega}{|\omega|}, 1 \big)$. Let $V^+ = \{ \omega \in \R^{n-1} \backslash \{ 0 \} : G_0(\omega) \in V \}$. By a crucial observation due to Ou--Wang \cite{OuWang2022}, if $V^+$ is tangent to $\mathcal{C} = \{ (\omega, \frac{\omega}{|\omega|}) : \omega \in A^{n-1} \}$ up to an angle $R^{-\delta_m}$, then $N_{R^{-\frac{1}{2}+\delta_m}} V^+ \cap \mathcal{C}$ is a $O(R^{-\delta_m})$-neighbourhood of $O(1)$ radial lines. In the easier case of the extension operator of the paraboloid (see \cite{Guth2018}), this \emph{narrow} case does not happen. As an example in case of the cone, we can consider $V= \langle e_1 + e_n, e_{i_1}, \ldots, e_{i_{m-1}} \rangle \subseteq \R^n$ with $e_{i_j} \neq e_{i_k}$ for $j \neq k$ and $i_k \in \{2,\ldots,n-1\}$. Here $e_i$ denotes the $i$th unit vector. We have $\dim(V) = m$, but clearly
\begin{equation*}
V^+ = \{ \omega \in \R^{n-1} \backslash \{0 \} : ( \frac{- \omega}{|\omega|},1) \in V \} = \{ - \mu e_1 : \mu > 0 \} 
\end{equation*}
is always one-dimensional.

In the variable coefficient case, we see that if $V^+$ is tangent to $\mathcal{C}_x = \{ \partial_x \phi^\lambda(\bar{x};\omega) \}$ 
up to an angle $R^{-\delta_m}$, then $N_{R^{-\frac{1}{2}+\delta_m}} V^+ \cap \mathcal{C}_x$ is a $O(R^{-\delta_m})$-neighbourhood of $O(1)$ radial lines. Hence, the case of narrow space-time frequencies does not contribute to the $k$-broad norm. Otherwise, we refer to this as \emph{case of broad space-time frequencies}, and we shall see that we find quantitative transversality to hold. We can deduce transverse equidistribution estimates similar to the paraboloid case (or its variable coefficient counterpart). In the constant coefficient case, but for arbitrary cones, this was recently investigated by Gao \emph{et al.} in \cite{GaoLiuMiaoXi2023}. We shall see how the arguments extend to the variable coefficient case. 

In this section we aim to prove the following estimate for $g$ concentrated on wave packets tangential to $Z$ in the case of broad space-time frequencies:
\begin{equation*}
\int_{B \cap N_{\rho^{\frac{1}{2}+\delta_m}}(Z)} |T^\lambda g|^2 \lesssim R^{\frac{1}{2}+O(\delta_m)} \big( \frac{\rho}{R} \big)^{\frac{n-m}{2}} \| g \|^2_{L^2}.
\end{equation*}
The precise statement with additional assumptions on $g$ is provided in Lemma \ref{lem:TransverseEquidistributionEstimate}.

\subsection{Geometric preliminaries: narrow and broad space-time frequencies}
\label{subsection:GeometricPreliminaries}
In this subsection we quantify the above cases of narrow and broad space-time frequencies. There are two linearizations involved. The transverse complete intersection is linearized and below $V$ plays the role of the tangent space of $Z$. Moreover, the phase function is linearized such that we can apply arguments from the constant-coefficient case due to Ou--Wang \cite{OuWang2022} and Gao \emph{et al.} \cite{GaoLiuMiaoXi2023}.

We consider an $m$-dimensional linear subspace $V \subseteq \R^n$. Let $A = (a_{i,j})_{i,j} \in \R^{(n-m) \times n}$ be a matrix of maximal rank such that
\begin{equation*}
V = \{ x \in \R^n : Ax = 0 \}.
\end{equation*}
We consider the intersection of $V$ with the hyperplane $\{x_n = 0\}$ perceived as subspace of $\R^{n-1}$:
\begin{equation*}
V^{-} = \{ x' \in \R^{n-1}: A(x',0) = 0 \}.
\end{equation*}
With $V$ playing the role of the tangent space of $Z$, which is close to the values of the generalized Gauss map, which are not contained in the hyperplane $\{ x_n = 0 \}$, we suppose that $\dim V^{-} = \dim V - 1$.

\medskip

The linearization is easier to carry out after changing to graph parametrization in $u$-frequencies via $\Psi^\lambda$, which we recall is implicitly defined by
\begin{equation*}
\partial_{x'} \phi^\lambda(\bar{x};\Psi^\lambda(\bar{x};u)) = u.
\end{equation*}
We use short-hand notation $\Psi(u) = \Psi^\lambda(\bar{x};u)$. Note that $\Psi$ is $1$-homogeneous like $\partial_{x'} \phi^\lambda$ and the identity mapping because
\begin{equation*}
\partial_{x'} \phi^\lambda(\bar{x};\Psi(\mu u)) = \mu u = \mu \partial_{x'} \phi^\lambda(\bar{x};\Psi(u)) = \partial_{x'} \phi^\lambda(\bar{x};\mu \Psi(u)). 
\end{equation*}
By substituting $\tilde{\phi}(u) = h_{\bar{x}}(u) = \partial_{x_n} \phi^\lambda(\bar{x};\Psi(u))$ the arguments due to Gao \emph{et al.} \cite{GaoLiuMiaoXi2023} apply. We define a set
\begin{equation*}
L = \{ u \in B_{n-1}(0,2) \backslash B_{n-1}(0,1/2) : A (-\nabla_u \tilde{\phi}(u),1) = 0  \}.
\end{equation*}
The set $\{ (u,\tilde{\phi}(u)) : u \in L \}$ describes the points on the generalized cone, which have a normal in $V$. The \emph{case of narrow space-time frequencies} gives a negligible contribution to the broad norm:
\begin{lemma}[{\cite[Lemma~4.5]{GaoLiuMiaoXi2023}}]
\label{lem:NegligibleContributionBall}
Let $\eta \in \mathbb{S}^{n-2} \subseteq \R^{n-1}$. If $\eta \in L$ and $\angle(\eta,V^-) > \frac{\pi}{2} - K^{-2}$, then $L$ is contained in the set $\{ \xi \in \R^{n-1} \backslash \{ 0 \}: \angle(\xi,\eta) \lesssim K^{-2} \}$.
\end{lemma}

\begin{remark}
\label{rem:NarrowFrequencyLocalization}
It is important to note that, contrary to the case of broad space-time frequencies analyzed below, the lemma does not hinge on a stronger localization of $\eta$. For later purposes, note that balls $B(\bar{x};R^{\frac{1}{2}+\delta_m})$, for which Lemma \ref{lem:NegligibleContributionBall} applies, do not contribute in the $k$-broad norm.
\end{remark}

\medskip

We turn to the more involved case of \emph{broad space-time frequencies}: In general $\{(u,\tilde{\phi}(u)) : u \in L\}$ does not lie in an affine subspace because $L$ is not an affine subspace. We start by linearizing $L$ at $\eta \in \mathbb{S}^{n-2}$. By taking the orthogonal complement of a suitable extension of the tangent space of $L$ at $\eta$ we shall construct $W$, which is quantitatively transverse to $V$. 
Let $\eta \in L \cap \mathbb{S}^{n-2}$ with $\angle(\eta,V^-) < \frac{\pi}{2} - K^{-2}$. Let
\begin{equation*}
L_i = \{ u \in A^{n-1} \, : \, \sum_{j=1}^{n-1} a_{i, j} \partial_j \tilde{\phi}(\eta) - a_{i,n} = 0 \}, \quad i =1,\ldots,n-m
\end{equation*}
such that 
\begin{equation*}
L = \bigcap_{i=1}^{n-m} L_i.
\end{equation*}

We compute the normal $n^{(i)}$ of $L_i$ at $\eta$ to be
\begin{equation*}
n^{(i)}_k = \sum_{j=1}^{n-1} a_{i,j} \partial^2_{jk} \tilde{\phi}(\eta), \quad n^{(i)} = \partial^2_{\eta \eta} \tilde{\phi}(\eta) \alpha_i \text{ with } \alpha_i = (a_{i,1},\ldots,a_{i,n-1}), \; i =1,\ldots,n-m.
\end{equation*}
The tangent space of $L_i$ at $\eta$ is given by 
\begin{equation*}
T_\eta L_i = (n^{(i)})^{\perp} = \{ \xi' \in \R^{n-1} : (n^{(i)}, \xi') = 0 \} = \{ \xi' \in \R^{n-1} \,:\, (\partial^{2}_{\eta \eta} \tilde{\phi}(\eta) \alpha_i, \xi') = 0 \} .
\end{equation*}

We argue that the normals are linearly independent: Let $\tilde{V} = \langle n^{(1)}, \ldots, n^{(n-m)} \rangle$.
\begin{lemma}
$\tilde{V} \subseteq \R^{n-1}$ is a subspace of dimension $n-m$.
\end{lemma}
\begin{proof}
It is key to observe that $\partial^{2}_{\eta \eta} \tilde{\phi}(\eta) \alpha_i$, $i=1,\ldots,n-m$ are linearly independent due to the angle condition $\angle (\eta, V^-) < \frac{\pi}{2} - K^{-2}$. Indeed, the Hessian is degenerate only in the direction of $\eta$, but $\alpha_i$ is orthogonal to $V^-$, and $A$ has maximal rank.
\end{proof}
Therefore, the tangent space of $L$ is given by the intersection of $T_\eta L_i$:
\begin{equation*}
T_\eta L = \gamma_1^{\perp} \cap \ldots \cap \gamma_{n-m}^{\perp} = \langle \gamma_1, \ldots , \gamma_{n-m} \rangle^{\perp}.
\end{equation*}

We have $T_\eta L = \bar{V}^- = \tilde{V}^\perp$ is the orthogonal complement of $\tilde{V}$ in $\R^{n-1}$ such that
\begin{equation*}
\R^{n-1} = \tilde{V} \oplus \bar{V}^-.
\end{equation*}

Let $\bar{V} = \langle \bar{V}^-, e_n \rangle$ be the linear subspace spanned by $\bar{V}^-$ and $e_{n}$. We let $W= \bar{V}^\perp$ be the orthogonal complement of $\bar{V}$ in $\R^{n}$:
\begin{equation*}
\R^{n} = \bar{V} \oplus W.
\end{equation*}
As pointed out in \cite{GaoLiuMiaoXi2023}, all the linear subspaces depend on the choice of $\eta$. We have the following quantitative transversality:
\begin{lemma}[{\cite[Lemma~4.6]{GaoLiuMiaoXi2023}}]
\label{lem:QuantitativeTransverse}
Let $\eta \in \mathbb{S}^{n-2} \cap L$. If $\angle(\eta,V^-) \leq \frac{\pi}{2} - K^{-2}$, then $W$ and $V$ are transverse in the sense that $\angle(V,W) \gtrsim K^{-4}$.
\end{lemma}

\medskip

\subsection{Verifying the transverse equidistribution estimate: broad space-time frequencies}

\smallskip

We turn to the key equidistribution estimate in case of broad space-time frequencies. 

\begin{remark}
\label{remark:Flatness}
We can suppose that $Z$ in $2B$ is $K$-flat, i.e., for any $z,z' \in Z \cap 2B$ we have
\begin{equation*}
\angle(T_z Z, T_{z'} Z) \lesssim K^{-5}
\end{equation*}
with $K \lesssim R^\delta \ll \rho^{\delta_m}$. If there is a wave packet with $\angle (G^\lambda(x;\omega_\theta), T_z Z) \lesssim R^{-\frac{1}{2}+\delta_m}$ for all $z \in Z \cap 2 B$, it follows from the triangle inequality for small angles that
\begin{equation*}
\angle (T_z Z, T_{z'} Z) \lesssim R^{-\frac{1}{2}+\delta_m}.
\end{equation*}
If there is no wave packet such that the above holds, then there is nothing to show.
\end{remark}

\begin{lemma}
\label{lem:TransverseEquidistributionEstimate}
Let $Z$ be a transverse complete intersection with $\dim Z = m$, $\overline{\text{deg}} \; Z \lesssim_\varepsilon 1$, $B=B(\bar{x},R^{\frac{1}{2}+\delta_m})$ a ball of radius $R^{\frac{1}{2}+\delta_m}$, and let $g$ be concentrated on wave packets in $\mathbb{T}_{Z,B,\tau}$. Suppose that with the notations of Subsection \ref{subsection:GeometricPreliminaries}, with $\tilde{\phi} = h_{\bar{x}}$, and for some $\eta \in \Psi^{-1}(\tau) \cap \mathbb{S}^{n-2}$ we are in the situation of Lemma \ref{lem:QuantitativeTransverse}, that is in the case of broad space-time frequencies. Then, for any $R^{\frac{1}{2}} \ll \rho \leq R$,
\begin{equation}
\label{eq:TransverseEquidistributionEstimateI}
\int_{B \cap N_{\rho^{\frac{1}{2}+\delta_m}}(Z)} |T^\lambda g|^2 \lesssim R^{\frac{1}{2}+O(\delta_m)} \big( \frac{\rho}{R} \big)^{\frac{n-m}{2}} \| g \|^2_{L^2}.
\end{equation}
\end{lemma}

The proof relies on quantifying the uncertainty principle from \cite[Subsection~8.2]{GuthHickmanIliopoulou2019}.
Let $G: \R^n \to \C$  be such that $\text{supp} (\hat{G}) \subseteq B(c,r)$. Then we have essentially for $0<\rho < r^{-1}$:
\begin{equation}
\label{eq:QuantificationUncertaintyPrinciple}
\dashint_{B(x_0,\rho)} |G|^2 \lesssim \dashint_{B(x_0,r^{-1})} |G|^2.
\end{equation}

As a first step in the proof of Lemma \ref{lem:TransverseEquidistributionEstimate}, we consider wave packets tangential to linear subspaces: In the following transverse equidistribution estimates are considered with respect to some fixed linear subspace $V \subseteq \R^n$. Recall that $B$ is a ball of radius $R^{\frac{1}{2}+\delta_m}$ with centre $\bar{x} \in \R^n$, and define
\begin{equation*}
\mathbb{T}_{V,B} = \{(\theta,v) : \angle(G^\lambda(\bar{x};\omega_\theta),V) \lesssim R^{-\frac{1}{2}+\delta_m} \text{ and } T_{\theta,v} \cap B \neq \emptyset \}.
\end{equation*}
Recall that $R^{\frac{1}{2}} < \rho < R$ and, for $\tau \subseteq \R^{n-1}$ a sector of aperture $O(\rho^{-\frac{1}{2}+\delta_m})$ centred around a point in $A^{n-1}$, define
\begin{equation*}
\mathbb{T}_{V,B,\tau} =  \{ (\theta,v) \in \mathbb{T}_{V,B} : \theta \cap \big( \frac{\tau}{10} \big) \neq \emptyset \}.
\end{equation*}

\begin{lemma}
\label{lem:TransverseEquidistributionEstimateLinearSubspace}
If $V \subseteq \R^n$ is a linear subspace, then there exists a linear subspace $W$ with the following properties:
\begin{itemize}
\item[(1)] $\dim V + \dim W = n$;
\item[(2)] $V$ and $W$ are quantitatively transverse: we have $\angle(V,W) \gtrsim K^{-4}$;
\item[(3)] if $g$ is concentrated on wave packets from $\mathbb{T}_{V,B,\tau}$, and there is $\eta \in \Psi^{-1}(\tau)$ such that for $\tilde{\phi} = h_{\bar{x}}$ the assumptions of Lemma \ref{lem:QuantitativeTransverse} are valid, $\Pi$ is any plane parallel to $W$ and $x_0 \in \Pi \cap B$, then the inequality
\begin{equation*}
\int_{\Pi \cap B(x_0,\rho^{\frac{1}{2}+\delta_m})} |T^\lambda g|^2 \lesssim_\delta R^{O(\delta_m)} \big( \frac{\rho}{R} \big)^{\frac{\dim W}{2}} \| g \|^{2\delta / (1+\delta)}_{L^2} \big( \int_{\Pi \cap 2B} |T^\lambda g|^2 \big)^{\frac{1}{1+\delta}}
\end{equation*}
holds, up to inclusion of $\text{RapDec}(R) \| g \|_{L^2}$ on the right-hand side.

\end{itemize}
\end{lemma}
\begin{proof}
The proof largely follows \cite{GuthHickmanIliopoulou2019}. The difference is that we are dealing with the larger sectors $\tau$ here having an aperture $O(\rho^{-\frac{1}{2}+\delta_m})$ in contrast to balls of radius $O(\rho^{-\frac{1}{2}+\delta_m})$. This will be compensated by the null direction of the conic surface.

\medskip

\textbf{Constructing the subspace $W$:} Recall that $h_{\bar{x}}(u) = \partial_{x_n} \phi^\lambda(\bar{x};\Psi(u))$ with $\partial_{x'} \phi^\lambda(\bar{x};\Psi(u)) = u$ such that $(u,h_{\bar{x}}(u))$ is a graph parametrization of $\partial_{x} \phi^\lambda(\bar{x}; \cdot)$ in $u$-frequencies with $\omega = \Psi(u)$. If $L \cap \Psi^{-1}(\tau) = \emptyset$, then $\Psi(L) \cap \tau = \emptyset$, but then, by Lemma \ref{lem:ConsequencesReductionGaussMap} we had $\mathbb{T}_{V,B,\tau} = \emptyset$ and there is nothing to show. 

Suppose $L \cap \Psi^{-1}(\tau) \neq \emptyset$ in the following and fix $\eta \in L \cap \Psi^{-1}(\tau)$. Note that
\begin{equation*}
\partial^2_{x u_1} \phi^\lambda(\bar{x};\Psi(u)) \wedge \ldots \wedge \partial^2_{x u_{n-1}} \phi^\lambda(\bar{x};\Psi(u)) = G_0(\bar{x};\omega) \cdot \det J \Psi(u).
\end{equation*}
Hence, we can construct $W$ around $\eta \in L \cap \Psi^{-1}(\tau)$ as in Subsection \ref{subsection:GeometricPreliminaries}. $W$ and $V$ are quantitatively transverse as in (2) by Lemma \ref{lem:QuantitativeTransverse}.

\medskip

\textbf{Verifying the transverse equidistribution estimate:} Recall that $g$ is concentrated on wave packets $\mathbb{T}_{V,B,\tau}$, $B$ is a $R^{\frac{1}{2}+\delta_m}$-ball, and $\tau$ is a $O(\rho^{-\frac{1}{2}+\delta_m})$-sector. Let $\eta_B(x) = \eta((x-\bar{x})/R^{\frac{1}{2}+\delta_m})$ denote a Schwartz cutoff, which satisfies $\eta(x) = 1$ for $x \in B(0,2)$. 

Let $\Sigma(\omega) = \partial_x \phi^\lambda(\bar{x};\omega)$. If $\omega \in \text{supp} (g_{\theta,v})$, then $|\omega- \omega_\theta| < R^{-\frac{1}{2}}$, and so $|\Sigma(\omega) - \xi_\theta| \lesssim R^{-\frac{1}{2}}$, where $\xi_\theta = \Sigma(\omega_\theta)$. By non-stationary phase, we find like in \cite{GuthHickmanIliopoulou2019}
\begin{equation*}
|(\eta_B \cdot T^\lambda g_{\theta,v}) \vert_{\Pi} \widehat \; (\xi) | \lesssim_N R^{O(1)} w_{B(\text{proj}_{W} \xi_\theta, R^{-\frac{1}{2}})}(\xi) \| g_{\theta,v} \|_{L^2}.
\end{equation*}
Let
\begin{align*}
L = \{ u \in A^{n-1} : (-\nabla h_{\bar{x}}(u),1) \in V \}, \quad S_\omega = \{ \omega \in A^{n-1} : G^\lambda(\bar{x};\omega) \in V \}.
\end{align*}
Let $A_u = \eta + T_\eta L$ denote the affine variant of the linear subspace $T_\eta L$, and $A_\xi = A_u \times \R$. Let $V_\xi$ be the linear subspace corresponding to $A_\xi$. We have $V_\xi^\perp = W$.

Next, we shall show that
\begin{equation}
\label{eq:DistanceFrequencySubspace}
\text{dist}(\xi_\theta, A_\xi) \lesssim R^{-\frac{1}{2}+\delta_m}.
\end{equation}
Once the above is proved, the proof can then be concluded like in \cite{GuthHickmanIliopoulou2019}. Since the space-time frequencies of $\eta_B \cdot T^\lambda g \vert_\Pi$ are concentrated in a ball of radius $R^{-\frac{1}{2}+\delta_m}$, we can then use a rigorous version of \eqref{eq:QuantificationUncertaintyPrinciple} to conclude.

\smallskip

We turn to the proof of \eqref{eq:DistanceFrequencySubspace}: Fix $(\theta,v) \in \mathbb{T}_{V,B,\tau}$ and let
\begin{equation*}
u_\theta =  (\Sigma(\omega_\theta)_1,\ldots,\Sigma(\omega_\theta)_{n-1}).
\end{equation*}
The triangle inequality yields
\begin{equation}
\label{eq:TriangleInequality}
\text{dist}(\xi_\theta,A_\xi) = \text{dist}(u_\theta,A_u) \leq \text{dist}(u_\theta, L \cap \Psi^{-1}(\tau)) + \sup_{u_* \in L \cap \Psi^{-1}(\tau)} \text{dist}(u_*,A_u).
\end{equation}
Furthermore, by Lemma \ref{lem:ConsequencesReductionGaussMap}, taking advantage of the null direction,
\begin{equation*}
\text{dist}(u_\theta, L \cap \Psi^{-1}(\tau)) \sim \text{dist}(\omega_\theta,S_\omega \cap \tau) \lesssim \angle(G^\lambda(\bar{x};\omega_\theta),V) \lesssim R^{-\frac{1}{2}+\delta_m},
\end{equation*}
where the last inequality is by the definition of $\mathbb{T}_{V,B,\tau}$.

We turn to the estimate of the second term in \eqref{eq:TriangleInequality}: Fix $u_* \in L \cap \Psi^{-1}(\tau)$. We note that $\text{dist}(u_*,A_u) = \text{dist}(u_*,A_{\bar{u}})$ for $\bar{u} = \frac{ | u_* |}{ | u |} u$ by null direction. Let $A_{\bar{u}} = u_0 + V_u$ for some linear subspace $V_u$. Now we note that the surface $L \cap \Psi^{-1}(\tau) \cap ( | u_* | \cdot \mathbb{S}^{n-2} )$, provided $\rho$ is large enough, can be written as a subset of the graph of a function $\psi: \mathcal{W} \to V_u^{\perp}$, where $\mathcal{W} \subseteq V_u$ is a subset around the origin of size $O(\rho^{-\frac{1}{2}+\delta_m})$. More precisely, we may write
\begin{equation*}
L \cap \Psi^{-1}(\tau) \cap ( \| u_* \| \cdot \mathbb{S}^{n-2} ) \subseteq \{ w +\psi(w) : w \in \mathcal{W} \} + u_0
\end{equation*}
with $\psi(0) = 0$ and $\nabla \psi(0) = 0$. The estimate now follows from Taylor expansion:
\begin{equation*}
| \psi(w) | = O(\rho^{-1+2 \delta_m}) = O(R^{-\frac{1}{2}+\delta_m}).
\end{equation*}
Here we used the hypothesis $R^{\frac{1}{2}} < \rho \ll R$ (cf. \cite{GuthHickmanIliopoulou2019}), which is absent for constant-coefficients.
\end{proof}

For the proof of the transverse equidistribution estimate in Lemma \ref{lem:TransverseEquidistributionEstimate}, we have to extend the estimate from a fixed vector space to a variety. The argument follows \cite[Section~8.4]{GuthHickmanIliopoulou2019} with the difference that the quantitative transversality (see \cite[Definition~8.11]{GuthHickmanIliopoulou2019}) mildly depends on the scale. This is due to the necessity of distinguishing between broad and narrow space-time frequencies and reflected in Lemma \ref{lem:QuantitativeTransverse}. The vector space $W$ constructed in Subsection \ref{subsection:GeometricPreliminaries} now satisfies
\begin{equation*}
\angle(v,w) \gtrsim K^{-4} \text{ for all } 0 \neq v \in T_{z} Z \text{ and } 0 \neq w \in W.
\end{equation*}

\begin{proof}[Proof of Lemma \ref{lem:TransverseEquidistributionEstimate}]
First, by Lemma \ref{lem:ConsequencesReductionGaussMap}, we have
\begin{equation*}
|G^\lambda(\bar{x};\theta) - G^\lambda(x;\theta)| \lesssim |x-\bar{x}|/\lambda \lesssim R^{-\frac{1}{2}+\delta_m}.
\end{equation*}
This yields
\begin{equation*}
\angle(G^\lambda(\bar{x};\theta), T_z Z) \lesssim R^{-\frac{1}{2}+\delta_m} \text{ for all } z \in Z \cap 2B.
\end{equation*}
Letting $V= T_z Z$, we have
\begin{equation*}
\mathbb{T}_{Z,B,\tau} \subseteq \mathbb{T}_{V,B,\tau}.
\end{equation*}
We can apply Lemma \ref{lem:TransverseEquidistributionEstimateLinearSubspace} to find a subspace $W$ such that
\begin{equation}
\label{eq:TransversalityVW}
\angle (V,W) \gtrsim K^{-4}
\end{equation}
and
\begin{equation}
\label{eq:TransverseEquidistributionSubspace}
\int_{\Pi \cap B(x_0,\rho^{\frac{1}{2}+\delta_m})} |T^\lambda g|^2 \lesssim_\delta R^{O(\delta_m)} \big( \frac{\rho}{R} \big)^{\frac{\dim W}{2}} \| g \|^{\frac{2 \delta}{1+\delta}}_{L^2} \big( \int_{\Pi \cap 2B} |T^\lambda g|^2 \big)^{\frac{1}{1+\delta}}
\end{equation}
for every affine subspace $\Pi$ parallel to $W$. By Remark \ref{remark:Flatness}, $(T_z Z, W)$ is a quantitatively transverse pair for all $z \in Z \cap 2B$. We have by (the proof of) \cite[Lemma~8.13]{GuthHickmanIliopoulou2019}
\begin{equation*}
\Pi \cap N_{\rho^{\frac{1}{2}+\delta_m}}(Z) \cap B \subseteq N_{C K^4 \rho^{\frac{1}{2}+\delta_m}}(\Pi \cap Z) \cap 2B.
\end{equation*}
Here is a minor change compared to \cite{GuthHickmanIliopoulou2019} as the size of the neighbourhood of $\Pi \cap Z$ is now increased by $K^4$. However, this factor can be absorbed into $R^{O(\delta_m)}$.
Since $\Pi \cap Z$ is a transverse complete intersection of dimension $\dim W + \dim Z - n$, a result due to Wongkew \cite{Wongkew1993} yields that $\Pi \cap N_{\rho^{\frac{1}{2}+\delta_m}}(Z) \cap B$ can be covered by
\begin{equation*}
O \big( R^{O(\delta_m)} \big( \frac{R}{\rho}\big)^{(\dim W + \dim Z - n)/2} \big) = O(R^{O(\delta_m)})
\end{equation*}
balls of radius $\rho^{\frac{1}{2}+\delta_m}$ because $K \lesssim R^\delta \ll \rho^{\delta_m}$. Now the argument is concluded by integrating over $\Pi$ parallel to $W$ and H\"ormander's $L^2$-bound (cf. Lemma \ref{lem:Hoermander}). For details we refer to \cite[p.~318]{GuthHickmanIliopoulou2019}.
\end{proof}

\section{Main inductive argument: Proof of the $k$-broad estimate}
\label{section:MainInductiveArgument}
The $k$-broad estimate is a consequence of the following claim, which is amenable to induction on scales and dimension. Let
\begin{equation*}
\bar{p}(k,n) = 2 \cdot \frac{n+k}{n+k-2}.
\end{equation*}
\begin{theorem}
\label{thm:MainInductionTheorem}
For $\varepsilon > 0$ sufficiently small, there are
\begin{equation*}
0<\delta \ll \delta_n \ll \delta_{n-1} \ll \ldots \ll \delta_1 \ll \varepsilon
\end{equation*}
and large dyadic parameters $\bar{A}_\varepsilon$, $\bar{C}_\varepsilon$, $D_{m,\varepsilon} \lesssim_\varepsilon 1$ and $\theta_m < \varepsilon$ such that the following holds. Suppose $Z=Z(P_1,\ldots,P_{n-m})$ is a transverse complete intersection with $\overline{\deg} Z \leq D_{m,\varepsilon}$. For all $2 \leq k \leq n$, $1 \leq A \leq \bar{A}_\varepsilon$ dyadic and $1 \leq K \leq R \leq \lambda$, the inequality
\begin{equation}
\label{eq:MainInductionEstimate}
\| T^\lambda f \|_{BL^p_{k,A}(B(0,R))} \lesssim_\varepsilon K^{\bar{C}_\varepsilon} R^{\theta_m + \delta(\log \bar{A}_\varepsilon - \log A) - e_{k,n}(p) + \frac{1}{2}} \| f \|_{L^2(A^{n-1})}
\end{equation}
holds whenever $f$ is concentrated on wave packets from $\mathbb{T}_Z$ and
\begin{equation}
\label{eq:Inductionp}
2 \leq p \leq \bar{p}_0(k,m) = 
\begin{cases}
\bar{p}(k,m), \quad \text{if } k <m, \\
\bar{p}(m,m) + \delta, \quad \text{if } k=m.
\end{cases}
\end{equation}
Above,
\begin{equation*}
e_{k,n}(p) = \frac{1}{2} \big( \frac{1}{2} - \frac{1}{p} \big) (n+k).
\end{equation*}
\end{theorem}
For future reference we denote by $E_{m,A}(R)$ the constant on the right-hand side of \eqref{eq:MainInductionEstimate}:
\begin{equation}
\label{eq:InductionQuantity}
E_{m,A}(R) = C_{m,\varepsilon} K^{\bar{C}_\varepsilon} R^{\varepsilon - c_n \delta_m + \delta(\log \bar{A}_\varepsilon - \log A) - e_{k,n}(p) + \frac{1}{2}}.
\end{equation}

\medskip

In the first step we reduce to $R \lesssim_\varepsilon \lambda^{1-\varepsilon}$ by covering $B(0,\lambda)$ with balls of radius $\lambda^{1-\varepsilon}$. The technical details are provided in \cite[Lemma~10.2]{GuthHickmanIliopoulou2019}. This reduction is necessary for polynomial approximation of the core curve $\gamma^\lambda_{\omega,v}$ uniformly in $R$.

Next, we set up the induction argument for $1 \leq R \lesssim_\varepsilon \lambda^{1-\varepsilon}$. For $\varepsilon >0$ sufficiently small, it is enough to consider $K \leq R^\delta$ by choosing $\bar{C}_\varepsilon$ sufficiently large (as the claim then follows from the trivial $L^1$-$L^\infty$-estimate and crude summation). We let furthermore
\begin{equation}
\label{eq:Parameters}
\begin{split}
D_{m,\varepsilon} = \varepsilon^{-\delta^{-(2n-m)}}, \quad \theta_m(\varepsilon) = \varepsilon - c_n \delta_m, \quad \bar{A}_\varepsilon = \lceil e^{\frac{10n}{\delta}} \rceil, \\
\delta_i = \delta_i(\varepsilon) = \varepsilon^{2i + 1} \text{ for all } i=1, \ldots, n, \text{ and } \delta = \delta(\varepsilon) \ll \delta_{n}. 
\end{split}
\end{equation}
The base case is given by $m \leq k-1$, and $A \geq 2^{10}$. For details we refer to \cite[Subsection~10.3]{GuthHickmanIliopoulou2019}.

\subsection{Inductive step} Let $2 \leq k \leq n-1$, $k \leq m \leq n$, and $K \leq R^\delta$. Assume, by way of induction hypothesis, that \eqref{eq:MainInductionEstimate} holds whenever $\text{dim} Z \leq m-1$ or the radial parameter is at most $\frac{R}{2}$. Fix $\varepsilon >0$, $1 < A \leq \bar{A}_\varepsilon$ and a transverse complete intersection $Z = Z(P_1,\ldots,P_{n-m})$ with $\overline{\deg} Z \leq D_{m,\varepsilon}$, where $\bar{A}_\varepsilon$ and $D_{m,\varepsilon}$ are as in \eqref{eq:Parameters}. Let $f$ be concentrated on wave packets from $\mathbb{T}_Z$. It suffices to show \eqref{eq:MainInductionEstimate} for $p = \bar{p}_0(k,m)$ by interpolation with the trivial $L^2$-bound. We recall the two cases to be analyzed:

\medskip

\textbf{The cellular case:} For any transverse complete intersection $Y^{l} \subseteq Z$ of dimension $1 \leq l \leq m-1$ and maximum degree at most $(D_{m,\varepsilon})^n$, the inequality
\begin{equation}
\label{eq:CellularCase}
\| T^\lambda f \|^p_{BL^p_{k,A}(N_{R^{\frac{1}{2} + \delta_m}/4}(Y^l) \cap B(0,R))} < c_{alg} \| T^\lambda f \|^p_{BL^p_{k,A}(B(0,R))}
\end{equation}
holds.

\medskip

\textbf{The algebraic case:} There exists a transverse complete intersection $Y^l \subseteq Z$ of dimension $1 \leq l \leq m-1$ of maximum degree at most $(D_{m,\varepsilon})^n$ such that
\begin{equation}
\label{eq:AlgebraicCase}
\| T^\lambda f \|^p_{BL^p_{k,A}(N_{R^{\frac{1}{2}+\delta_m}/4}(Y^l) \cap B(0,R))} \geq c_{alg} \| T^\lambda f \|^p_{BL^p_{k,A}(B(0,R))}.
\end{equation}
Here $c_{alg}>0$ depends on $n$ and $\varepsilon$.

\subsubsection{Cellular case}
The cellular case is as usually treated by induction on the radius. Via polynomial partitioning the $BL^p_{k,A}$-norm is equidistributed among the cells, and the induction closes.
This case is handled as in \cite[Section~10.5]{GuthHickmanIliopoulou2019}. We omit the details.

\subsubsection{Algebraic case}
 The algebraic case is more involved: $T^\lambda f$ can be regarded as concentrated near a low-dimensional and low degree variety $Y^l$ (for an oversimplification, one can think of a hyperplane). In the \emph{tangential subcase} the wave packets from $f$ are also tangential to this variety, then we can use induction on the dimension to conclude. In the \emph{transverse subcase}, many wave packets are transverse to $Y^l$. Then we can conclude via transverse equidistribution estimates. For the homogeneous phase functions the transverse equidistribution estimates are different from \cite[Section~10.6]{GuthHickmanIliopoulou2019}, and we give the details. 

\medskip
 
 Fix a transverse complete intersection $Y^l$ of dimension $1 \leq l \leq m-1$ of maximum degree $\overline{\deg} Y^l \leq (D_{m,\varepsilon})^n$, which satisfies \eqref{eq:MainInductionEstimate}. Let $R^{\frac{1}{2}} \ll \rho \ll R$ be such that $\rho^{\frac{1}{2}+\delta_l} = R^{\frac{1}{2}+\delta_m}$, and note that
\begin{equation*}
R \leq R^{2 \delta_l} \rho \text{ and } \rho \leq R^{-\delta_l/2} R.
\end{equation*}
Let $\mathcal{B}_\rho$ be a finitely overlapping cover of $B(0,R)$ by $\rho$-balls, and for each $B \in \mathcal{B}_\rho$ define
\begin{equation*}
\mathbb{T}_B = \{(\theta,v) \in \mathbb{T}: T_{\theta,v} \cap N_{R^{\frac{1}{2} + \delta_m}/4}(Y^l) \cap B \neq \emptyset \}
\end{equation*}
and
\begin{equation*}
f_B:= \sum_{(\theta,v) \in \T_B} f_{\theta,v}.
\end{equation*}
We have by the triangle inequality for broad norms
\begin{equation*}
\| T^\lambda f \|^p_{BL^p_{k,A}(B(0,R))} \lesssim \sum_{B \in \mathcal{B}_\rho} \| T^\lambda f_B \|^p_{BL^p_{k,A}(N_{R^{\frac{1}{2}+\delta_m}/4}(Y^l) \cap B)}
\end{equation*}
up to $\text{RapDec}(R) \| f \|^p_{L^2}$ on the right-hand side by the rapid decay of the wave packets away from the tubes.

For $B=B(y,\rho) \in \mathcal{B}_\rho$, let $\mathbb{T}_{B,tang}$ denote the set of all $(\theta,v) \in \mathbb{T}_B$ with the property that, whenever $x \in T_{\theta,v}$ and $z \in Y^l \cap B(y,2 \rho)$ satisfy $|x-z| \leq 2 \bar{C}_{tang} \rho^{\frac{1}{2}+\delta_l}$, it follows that
\begin{equation*}
\angle (G^\lambda(x;\omega_\theta), T_z Y^l) \leq \frac{1}{2} \bar{c}_{tang} \rho^{-\frac{1}{2}+\delta_l},
\end{equation*} 
where $\bar{C}_{tang}$ and $\bar{c}_{tang}$ are the constants appearing in the definition of tangency. Furthermore, let $\mathbb{T}_{B,trans} = \mathbb{T}_B \backslash \mathbb{T}_{B,tang}$ and define
\begin{equation*}
f_{B,tang} = \sum_{(\theta,v) \in \mathbb{T}_{B,tang}} f_{\theta,v} \text{ and } f_{B,trans} = \sum_{(\theta,v) \in \mathbb{T}_{B,trans}} f_{\theta,v}.
\end{equation*}
It follows that $f_B = f_{B,tang} + f_{B,trans}$ and, by the triangle inequality for broad norms, we conclude that
\begin{equation*}
\| T^\lambda f \|^p_{BL^p_{k,A}(B(0,R))} \lesssim \sum_{B \in \mathcal{B}_\rho} \| T^\lambda f_{B,tang} \|^p_{BL^p_{k,A/2}(B)} + \sum_{B \in \mathcal{B}_\rho} \| T^\lambda f_{B,trans} \|^p_{BL^p_{k,A/2}(B)}.
\end{equation*}
Either the tangential or the transverse contribution to the above sum dominates, and each case is treated separately.

\medskip

\textbf{Tangential subcase:} Suppose that the tangential term dominates and we have
\begin{equation}
\| T^\lambda f \|^p_{BL^p_{k,A}(B(0,R))} \lesssim \sum_{B \in \mathcal{B}_\rho} \| T^\lambda f_{B,tang} \|^p_{BL^p_{k,A/2}(B)}.
\end{equation}
This case can be handled by applying the dimensional induction hypothesis. We refer to \cite[pp.~345-346]{GuthHickmanIliopoulou2019} for the details.

\medskip

\textbf{Transverse sub-case:} In this case, we have
\begin{equation*}
\| T^\lambda f \|^p_{BL^p_{k,A}(B(0,R))} \lesssim \sum_{B \in \mathcal{B}_\rho} \| T^\lambda f_{B,trans} \|^p_{BL^p_{k,A/2}(B)}.
\end{equation*}

The strategy in the transverse case is to use induction on radius to show that for some $\overline{c}_\varepsilon > 0$ one has (redenoting $f_j$ for $f_{B_j,trans}$)
\begin{equation}
\label{eq:TransverseSubcaseKeyEstimate}
\| T^\lambda f_{j} \|_{BL^p_{k,A/2}(B)} \leq \overline{c}_\varepsilon E_{m,A}(R) \| f_{j} \|_{L^2(A^{n-1})}
\end{equation}
for all $B_j \in \mathcal{B}_\rho$.

Provided $\overline{c}_\varepsilon > 0$ is chosen sufficiently small, depending only on $n$ and $\varepsilon$, the proof is concluded by almost orthogonality of $f_{B,\text{trans}}$ (see \cite[Eq.~(10.23)]{GuthHickmanIliopoulou2019}):
\begin{equation*}
\| T^\lambda f \|_{BL^p_{k,A/2}(B(0,R))} \lesssim_\varepsilon \bar{c}_\varepsilon E_{m,A}(R) \| f \|^{1-\frac{2}{p}}_{L^2} \big( \sum_{B \in \mathcal{B}_\rho} \| f_{B,trans} \|^2_{L^2} \big)^{\frac{1}{p}} \leq E_{m,A}(R) \| f \|_{L^2}.
\end{equation*}
This estimate relies on Lemma \ref{lem:TransverseIntersectionPolynomialApproximant}.

The main obstacle to \eqref{eq:TransverseSubcaseKeyEstimate} is that we cannot expect $f_{j}$ to satisfy the hypothesis of Theorem \ref{thm:MainInductionTheorem} at scale $\rho$. The remedy is to break $f_{j}$ into pieces $f_{j,b}$, which are $\rho^{\frac{1}{2}+\delta_m}$-tangent to a translated variety of $Z+b$. To obtain a favorable estimate when reassembling the pieces, we use transverse equidistribution estimates. This still follows \cite{GuthHickmanIliopoulou2019}. Recall however, as argued in Section \ref{section:TransverseEquidistribution}, that after restricting to a ball of size $R^{\frac{1}{2}+\delta_m}$, the transverse equidistribution estimates do not always hold true. But if these do not hold true, then the contribution can actually be neglected. The first step is to separate the essentially and non-essentially contributing balls.

\bigskip

\textbf{Separating essentially and non-essentially contributing $R^{\frac{1}{2}+\delta_m}$-balls}: We cover $B_j$ by finitely overlapping $R^{\frac{1}{2}+\delta_m}$-balls $B_{j,k}$. Let $(\theta,v) \in \T_{Z,B_{j,k}}$ and $x \in T_{\theta,v} \cap N_{R^{\frac{1}{2}+\delta_m}}(Z) \cap B_{j,k}$, $z \in Z$ with $|x-z| \leq \bar{C}_{tang} R^{\frac{1}{2}+\delta_m}$. By definition of tangency, we have
\begin{equation*}
\angle (G^\lambda(x;\omega_\theta),T_{z} Z) \lesssim R^{-\frac{1}{2}+\delta_m}.
\end{equation*}
Let $V= T_z Z$. By Lemma \ref{lem:ConsequencesReductionGaussMap}, we have
\begin{equation*}
\angle (G^\lambda(\bar{x};\omega_\theta),V) \lesssim R^{-\frac{1}{2}+\delta_m}.
\end{equation*}
Now we consider the linearization $\tilde{\phi}(u) = \partial_{x_n} \phi^\lambda (\bar{x};\Psi^\lambda(u))$ around $\bar{x}$, the centre of $B_{j,k}$. We consider like in Section \ref{subsection:GeometricPreliminaries} a matrix $A \in \R^{(n-m) \times n}$ of maximal rank such that
\begin{equation*}
V= \{ x \in \R^n : A x = 0 \}, \quad L = \{ u \in B_{n-1}(0,2) \backslash B_{n-1}(0,1/2) : A(-\nabla_u \tilde{\phi}(u),1) = 0 \},
\end{equation*}
where $L$ denotes the set of $u$-frequencies with normal in $V$.

We apply the dichotomy of Section \ref{subsection:GeometricPreliminaries}: 
\begin{itemize}
\item[Case~I:] $L$ is contained in $O(1)$ sectors of size $1 \times K^{-2} \times \ldots \times K^{-2}$ by Lemma \ref{lem:NegligibleContributionBall}. This is referred to as Case I. In Subsection \ref{subsection:GeometricPreliminaries} this was called the case of narrow space-time frequencies. Note that if $L$ is contained in $O(1)$ $1 \times K^{-2} \times \ldots \times K^{-2}$-sectors, then so is $\bigcup \theta$ with $\angle (G^\lambda(\bar{x};\omega_\theta),V) \lesssim R^{-\frac{1}{2}+\delta_m}$ by Lemma \ref{lem:ConsequencesReductionGaussMap}. Consequently, by $2 \leq k \leq m$, Case I-balls can be neglected in the $k$-broad norm (see \eqref{eq:EssentialContribution} below).

\item[Case~II:] This was referred to as case of broad space-time frequencies. We consider the further refinement $\T_{V,B_{j,k},\tau}$ with $\tau$ a $\rho^{-\frac{1}{2}}$-sector. By Lemma \ref{lem:QuantitativeTransverse}, there is a quantitatively transverse subspace $W$ with $\bar{V} \oplus W = \R^n$ and
\begin{equation*}
\angle (V,W) \gtrsim K^{-4}.
\end{equation*}
$\bar{V}$ denotes the extension of a tangent space of $L$ from Subsection \ref{subsection:GeometricPreliminaries}.
\end{itemize}
We let $X_I$ and $X_{II}$ denote the union of balls $B_{j,k}$ from Cases I and II.

\bigskip

\textbf{Sorting into medium tubes:} Redenote $g = f_{j,trans}$.
Next, we use the sorting into medium tubes as in \cite[Section~9]{GuthHickmanIliopoulou2019}\footnote{The arguments in \cite[Section~9]{GuthHickmanIliopoulou2019} clearly only depend on the non-degeneracy condition $C1)$.}. The idea is to carry out a wave packet decomposition at smaller scale $\rho \ll R$ on the $\rho$-ball. This requires us to mildly modify the phase function, amplitude function, and input function $\phi,a,g \to \tilde{\phi},\tilde{a},\tilde{g}$ to center the $\rho$-ball at the origin. The wave packets at scale $\rho$ are denoted by $\tilde{\mathbb{T}}$ with indices $\tilde{\theta}$ (a $\rho^{-\frac{1}{2}}$-cap) and $\tilde{v} \in \rho^{\frac{1+\delta}{2}} \Z^{n-1}$. We can break the $R$-wave packet $(\theta,v) \in \mathbb{T}$ into smaller scale $\rho$-wave packets essentially contributing to $(\theta,v)$. These $\rho$-wave packets are denoted by $\tilde{\mathbb{T}}_{\theta,v}$ (see \cite[Lemma~9.1]{GuthHickmanIliopoulou2019}).

 We use the following facts from \cite{GuthHickmanIliopoulou2019}:
Firstly, we recall how the tangency properties are inherited. The $\rho$-wave packets obtained from a wave packet $R^{-\frac{1}{2}+\delta_m}$-tangential to a variety $Z$ are $\rho^{-\frac{1}{2}+\delta_m}$-tangential to a translate $Z-b$ (cf. \cite[Proposition~9.2]{GuthHickmanIliopoulou2019}).

There is a sorting $\big( \mathcal{T}_{\tilde{\theta},w} \big)_{\tilde{\theta},w}$, $\mathcal{T}_{\tilde{\theta},w} \subseteq \mathbb{T}$ with $\tilde{\theta}$ a $\rho^{-\frac{1}{2}}$-cap and $w \in R^{\frac{1+\delta}{2}} \Z^{n-1}$, which partitions $\mathbb{T}$ (see \cite[Section~9.3]{GuthHickmanIliopoulou2019}). The corresponding \emph{medium tubes} of length $\rho$ and width $R^{\frac{1}{2}+\delta}$ are given by
\begin{equation*}
T_{\tilde{\theta},w}= \big( \bigcup_{(\theta,v) \in \mathcal{T}_{\tilde{\theta},w}} T_{\theta,v} \big) \cap B(y,\rho).
\end{equation*}
We have the almost orthogonality:
\begin{equation*}
\| g \|_2^2 \sim \sum_{(\tilde{\theta},w) \in \mathcal{T}} \| g_{\tilde{\theta},w} \|_{2}^2,
\end{equation*}
and we have that for every $(\theta,v) \in \mathcal{T}_{\tilde{\theta},w}$ it holds
\begin{equation*}
\text{dist}_H(T_{\theta,v} \cap B(y,\rho), T_{\tilde{\theta},w} ) \lesssim R^{\frac{1}{2}+\delta}.
\end{equation*}

We can also consider the $\rho$-wave packet decomposition of $(g_{\tilde{\theta},w}) \tilde{\;}$. This is concentrated on scale $\rho$-wave packets belonging to $\bigcup_{(\theta,v) \in \mathcal{T}_{\tilde{\theta},w}} \tilde{\mathbb{T}}_{\theta,v}$. We obtain a covering $(\tilde{\mathcal{T}}_{\tilde{\theta},w})_{\tilde{\theta},w}$ of $\tilde{\mathbb{T}}$ by almost disjoint sets. Hence, the sorting into medium tubes is suitable to switch between $R$ and $\rho$-wave packets.

 We define the essential part like in \cite[p.~3582]{OuWang2022}:
\begin{equation*}
g_{ess} = \sum_{(\tilde{\theta},w) \in \mathcal{T}_{ess}} g_{\tilde{\theta},w} = g - \sum_{(\tilde{\theta},w) \in \mathcal{T}_{tail}} g_{\tilde{\theta},w},
\end{equation*}
where
\begin{equation*}
\begin{split}
\mathcal{T}_{ess} &= \{ (\tilde{\theta},w) : \, \exists (\theta,v) \in \mathcal{T}_{\tilde{\theta},w} :  T_{\theta,v} \cap X_{II} \neq \emptyset \}, \\
\mathcal{T}_{tail} &= \{ (\tilde{\theta},w) : \, \forall (\theta,v) \in \mathcal{T}_{\tilde{\theta},w} : T_{\theta,v} \cap X_{II} = \emptyset \}.
\end{split}
\end{equation*}
Like in \cite{OuWang2022}, we infer that
\begin{equation}
\label{eq:EssentialContribution}
\| T^\lambda g \|_{BL^p_{k,A/2}(B_j)} \leq \| T^\lambda g_{ess} \|_{BL^p_{k,A/4}(B_j)} + \text{RapDec}(R) \| f \|_{L^2}.
\end{equation}

\medskip

\textbf{Breaking the essential contribution into smaller neighbourhoods of translates:} As in \cite{GuthHickmanIliopoulou2019}, we choose a set of translates $\mathcal{B}$, so that we can write
\begin{equation}
\label{eq:TransverseSubcaseKeyEstimateII}
\| T^\lambda g_{ess} \|_{BL^p_{k,A/4}(B_j)} \lesssim (\log R)^2 \big( \sum_{b \in \mathcal{B}} \| T^\lambda g_{ess,b} \|^p_{BL^p_{k,A/4}(B_j)} \big)^{\frac{1}{p}},
\end{equation}
where each piece $g_{ess,b}$ is defined so that it is concentrated on scale $\rho$ wave packets, which are tangential to some translate $Z-y+b$ of $Z$. This construction is provided in \cite[Lemma~10.5]{GuthHickmanIliopoulou2019}.
The set of translates $\mathcal{B}$ are chosen so that effectively
\begin{itemize}
\item[(i)] $N_{R^{\frac{1}{2}+\delta_m}}(Z) \subseteq \bigcup_{b \in \mathcal{B}} N_{\rho^{\frac{1}{2}+\delta_m}}(Z-y+b)$,
\item[(ii)] $N_{\rho^{\frac{1}{2}+\delta_m}}(Z-y+b)$ are essentially disjoint.
\end{itemize}
Technically, this decomposition also involves to ``tune" the wave packets of $g_{ess,b}$ to a collection of balls $\mathcal{B}' \subseteq  B_{K^2} $. This quantifies (i) and (ii) and is necessary due to the definition of the broad norm. In the following we take the resulting almost orthogonality \eqref{eq:AlmostOrthogonalityTranslates} for granted (see \cite[Lemma~9.6]{GuthHickmanIliopoulou2019}).

\medskip

By modifying $T^\lambda$ and $g_{ess,b}$, we can center the ball $B_j$ at the origin:
\begin{equation*}
\| T^\lambda g_{ess,b} \|_{BL^p_{k,A/4}(B_j)} \lesssim (\log R)^2 \big( \sum_{b \in \mathcal{B}} \| \tilde{T}^\lambda (g_{ess,b}) \tilde{\;}  \|^p_{BL^p_{k,A/4}(B(0,\rho))} \big)^{\frac{1}{p}}.
\end{equation*}
Now we can apply the radial induction hypothesis:
\begin{equation*}
\big( \sum_{b \in \mathcal{B}} \| \tilde{T}^\lambda (g_{ess,b}) \tilde{\;} \|^p_{BL^p_{k,A/4}(B(0,\rho))} \big)^{\frac{1}{p}} \leq E_{m,A/4}(\rho) \big( \sum_{b \in \mathcal{B}} \| g_{ess,b} \|^p_{L^2} \big)^{\frac{1}{p}}.
\end{equation*}
Once we have proved
\begin{equation}
\label{eq:TransverseEquidistribution}
\big( \sum_{b \in \mathcal{B}} \| \tilde{g}_{ess,b} \|^p_{L^2} \big)^{\frac{1}{p}} \lesssim R^{O(\delta_m)} \big( \frac{\rho}{R} \big)^{(n-m) \big( \frac{1}{4}- \frac{1}{2p} \big)} \| g_{ess} \|_2,
\end{equation}
the computation to conclude the proof is like in \cite{GuthHickmanIliopoulou2019}. \eqref{eq:TransverseEquidistribution} is proved via interpolation between $p=2$ and $p=\infty$.

\medskip

\textbf{Reassembling the set of translates:} For $p=2$ we can argue like in \cite{GuthHickmanIliopoulou2019} based on the almost orthogonality of the $\rho$-wave packets contributing to $(g_{ess,b}) \tilde{\;}$ (here we also use the essential disjointness of the translates in (ii) above):
\begin{equation}
\label{eq:AlmostOrthogonalityTranslates}
\big( \sum_{b \in \mathcal{B}} \| g_{ess,b} \|^2_{L^2} \big)^{\frac{1}{2}} \lesssim \| g_{ess} \|_2.
\end{equation}

We turn to the proof of \eqref{eq:TransverseEquidistribution} for $p=\infty$, which requires an additional argument compared to \cite{GuthHickmanIliopoulou2019}. The crucial ingredient remains transverse equidistribution. By almost orthogonality of $(\tilde{\theta},w) \in \mathcal{T}$ and the definition of $\mathcal{T}_{ess}$ we have
\begin{equation*}
\| g_{ess,b} \|^2_{L^2} \sim \sum_{(\tilde{\theta},w) \in \mathcal{T}_{ess}} \| (g_{ess,\tilde{\theta},w})_b \tilde{\;} \|^2_{L^2}.
\end{equation*}

By construction of $g_{ess,\tilde{\theta},w}$ there is $(\theta,v) \in \mathcal{T}_{\tilde{\theta},w}$ such that $T_{\theta,v}$ intersects $X_{II}$. Let $B= B(\bar{x};R^{\frac{1}{2}+\delta_m})$ denote the corresponding ball in $X_{II}$. Since the Hausdorff distance between $T_{\theta_1,v_1}$ for further $(\theta_1,v_1) \in \mathcal{T}_{\tilde{\theta},w}$ is $\lesssim R^{\frac{1}{2}+\delta}$, we can apply the reverse H\"ormander estimate from \cite{GuthHickmanIliopoulou2019} at scale $\rho$ with $r \sim R^{\frac{1}{2}+\delta_m}$ to find that
\begin{equation}
\label{eq:InverseHoermanderApplication}
\| ( g_{ess,\tilde{\theta},w} )_b \tilde{\;} \|_{L^2}^2 \lesssim R^{-\frac{1}{2}-\delta_m} \| \tilde{T}_*^\lambda  (g_{ess,\tilde{\theta},w})_b \tilde{\;} \|_{L^2(B(\bar{x}-y;10 R^{\frac{1}{2}+\delta_m}))}^2.
\end{equation}
The reverse H\"ormander estimate \cite[Lemma~9.5]{GuthHickmanIliopoulou2019} relies on every $\rho$-wave packet of $(g_{ess,\tilde{\theta},w})_b \tilde{\;}$ intersecting the $R^{\frac{1}{2}+\delta_m}$-ball $B$. The proof follows from Fourier series expansion and the constant-coefficient analog. This yields a slightly different oscillatory integral operator $\tilde{T}_*^\lambda$, whose data however inherits the properties of $(\phi,a)$. We refer to \cite{GuthHickmanIliopoulou2019}.

Next, we can apply Lemma \ref{lem:TransverseEquidistributionEstimate} to find that\footnote{This holds up to $\text{RapDec}(R)$. Note the shift and localization to $N_{\rho^{\frac{1}{2}+\delta_m}}(Z+b)$ on the left-hand side, which is explained in \cite[Equation~(9.10)]{GuthHickmanIliopoulou2019}.}
\begin{equation}
\label{eq:TransverseEquidistributionApplication}
\| T_*^\lambda g_{ess,b,\tilde{\theta},w}  \|^2_{L^2(B(\bar{x};10 R^{\frac{1}{2}+\delta_m}) \cap N_{\rho^{\frac{1}{2}+\delta_m}}(Z+b))} \lesssim R^{\frac{1}{2}+O(\delta_m)} \big( \frac{\rho}{R} \big)^{\frac{n-m}{2}} \| ( g_{ess,\tilde{\theta},w} )_b \tilde{\;} \|^2_{L^2}.
\end{equation}
Now we can use the almost orthogonality induced by $\tilde{\mathcal{T}}_{\tilde{\theta},w}$ to find after taking \eqref{eq:InverseHoermanderApplication} and \eqref{eq:TransverseEquidistributionApplication} together:
\begin{equation*}
\begin{split}
\sum_{(\tilde{\theta},w) \in \mathcal{T}_{ess}} \| ( g_{ess,\tilde{\theta},w} )_b \tilde{\;} \|_{L^2}^2 &\lesssim R^{O(\delta_m)} \big( \frac{\rho}{R} \big)^{\frac{n-m}{2}} \sum_{(\tilde{\theta},w) \in \mathcal{T}_{ess}} \| (g_{ess,\tilde{\theta},w})_b \tilde{\;} \|_{L^2}^2 \\
&\lesssim R^{O(\delta_m)} \big( \frac{\rho}{R} \big)^{\frac{n-m}{2}} \| g_{ess,b} \|_{L^2}^2 \\
&\lesssim R^{O(\delta_m)} \big( \frac{\rho}{R} \big)^{\frac{n-m}{2}} \| g_{ess} \|_{L^2}^2.
\end{split}
\end{equation*}
This finishes the proof of \eqref{eq:TransverseEquidistribution}, and we can argue that induction closes like in \cite{GuthHickmanIliopoulou2019}.

$\hfill \Box$

\section{From $k$-broad to linear estimates}
\label{section:LinearEstimates}
In this section we deduce the linear estimates from the $k$-broad estimates by applying the Bourgain--Guth argument \cite{BourgainGuth2011}. We show the following proposition:
\begin{proposition}
\label{prop:LinearFromBroadEstimates}
Suppose that for all $K \geq 1$ and all $\varepsilon > 0$ any oscillatory integral operator $T^\lambda$ built from reduced data $(\phi,a)$ obeys the $k$-broad inequality
\begin{equation}
\label{eq:HypothesisKBroad}
\| T^\lambda f \|_{BL^p_{k,A}(B(0,R))} \lesssim_\varepsilon K^{C_\varepsilon} R^\varepsilon \| f \|_{L^p(A^{n-1})}
\end{equation}
for some fixed $k$, $A=A_\varepsilon$, $p$, $C_\varepsilon$, and all $R \geq 1$. If
\begin{equation}
\label{eq:ConditionsP}
p(k,n) \leq p \leq \frac{2n}{n-2}, \qquad p(k,n) = 
\begin{cases}
2 \cdot \frac{n-1}{n-2} \quad \text{ if } 2 \leq k \leq 3, \\
2 \cdot \frac{2n-k+1}{2n-k-1} \quad \text{ if } k > 3,
\end{cases}
\end{equation}
then any oscillatory integral operator with $C1)$ and $C2^+)$ phase $\phi$ and amplitude $a$ satisfies
\begin{equation}
\label{eq:LinearEstimateFromBroad}
\| T^\lambda f \|_{L^p(\R^n)} \lesssim_{\phi,\varepsilon,a} \lambda^\varepsilon \| f \|_{L^p(A^{n-1})}.
\end{equation}
\end{proposition}
From this proposition Theorem \ref{thm:LpLpEstimatesVariableCoefficients} is immediate by choosing $k = \frac{n+1}{2}$ for $n$ odd and $k = \frac{n}{2} + 1$ for $n$ even as $\max(p(k,n),\bar{p}(k,n))$ gives the lower bound for $p$ in Theorem \ref{thm:LpLpEstimatesVariableCoefficients}.
\subsection{Outline of the argument}
 For the proof we use induction on scales: $Q_{p,\delta}(R)$ will denote the infimum over all constants $C$ for which the estimate
\begin{equation*}
\| T^\lambda f \|_{L^p(B(0,r))} \leq C \| f \|_{L^p(A^{n-1})}
\end{equation*}
holds for $1 \leq r \leq R$ and all oscillatory integral operators built from a suitable class of phase functions. This class is supposed to be invariant under rescaling and amenable to narrow decoupling, which is explained below.

With this definition, it remains to prove that for $p$ as in Proposition \ref{prop:LinearFromBroadEstimates}
\begin{equation*}
Q_{p,\delta}(R) \lesssim_\varepsilon R^\varepsilon
\end{equation*}
for all $\varepsilon >0$ and $1 \leq R\leq \lambda$.\\
To this end, we decompose $B(0,R)$ into finitely overlapping balls $B_{K^2}$ of radius $K^2$ and estimate $\| T^\lambda f \|_{L^p(B_{K^2})}$. $f$ is decomposed into ``broad" and ``narrow" term. The narrow term is of the form
\begin{equation}
\label{eq:NarrowTerm}
\sum_{\substack{\tau \in V_a \\ \text{for some } a}} f_\tau, 
\end{equation}
consisting of contributions to $f$ from sectors for which $G^\lambda(\bar{x};\tau)$ makes a small angle with some member of a family of $(k-1)$-planes. Here $\bar{x}$ denotes the centre of $B_{K^2}$. The broad term consists of contributions to $f$ from the remaining sectors. One may choose the planes $V_1, \ldots , V_A$ so that the broad term can be bounded by the $k$-broad inequality. Thus, $f$ of the form  \eqref{eq:NarrowTerm} has to be analyzed. This is accomplished by narrow $\ell^p$-decoupling and rescaling. 
\subsection{Narrow decoupling and the induction quantity}
In this subsection we shall find an estimate for $V \subseteq \R^n$ an $m$-dimensional linear subspace:
\begin{equation*}
\big\| \sum_{\tau \in V} T^\lambda g_\tau \big\|_{L^p(B_{K^2})} \lesssim_\delta \max(1,K^{(m-2) \big( \frac{1}{2}-\frac{1}{p} \big)}) K^\delta \big( \sum_{\tau \in V} \| T^\lambda g_\tau \|^p_{L^p(w_{B_{K^2}})} \big)^{\frac{1}{p}} + l.o.t..
\end{equation*}
The additional assumptions to be imposed for the above estimate to hold will motivate our choice of induction quantity below.

In the translation-invariant case, e.g., with $\mathcal{E}$ as in \eqref{eq:ConeExtensionOperator}, this follows from the $\ell^2$-decoupling
\begin{equation*}
\big\| \sum_\tau \mathcal{E} g_\tau \big\|_{L^p(B_{K^2})} \lesssim_\delta K^\delta \big( \sum_{\tau \in V} \| \mathcal{E} g_\tau \|^2_{L^p(w_{B_{K^2}})} \big)^{\frac{1}{2}}
\end{equation*}
for $2 \leq p \leq \frac{2n}{n-2}$ and by counting the sectors $\tau$ such that $\angle(G(\tau),V) \leq K^{-2}$. This is carried out in \cite{OuWang2022}; see also \cite[Lemma~2.2]{Harris2019}. It appears that the argument from \cite{Harris2019,OuWang2022} does not suffice to count the sectors already for general translation-invariant phases under a convexity assumption, for which reason an additional technical assumption is introduced.

 Let $\phi: \R^{n-1} \backslash \{ 0 \} \to \R$ be a $1$-homogeneous and smooth function with $\text{supp}(\phi) \subseteq \Xi$.\footnote{Recall that $\Xi \subseteq A^{n-1}$ denotes a sector with small aperture $c \ll 1$ centred at $e_{n-1}$.} By Taylor expansion we find
 \begin{equation}
 \label{eq:TaylorExpansionHomogeneousPhase}
\begin{split}
\phi(\omega',\omega_{n-1}) &= \omega_{n-1} \phi(\omega'/\omega_{n-1},1) \\
 &= \omega_{n-1} \phi(e_{n-1}) + \partial_{\omega'} \phi(e_{n-1}) \omega' + \frac{\langle \partial^2_{\omega' \omega'} \phi(e_{n-1}) \omega', \omega' \rangle}{2 \omega_{n-1}} + K^{-4} E(\omega),
 \end{split}
\end{equation}
with $E(\omega)$ $1$-homogeneous. By comparison with Taylor's formula we have
\begin{equation*}
K^{-4} E(\omega) = \sum_{|\alpha| = 3} \frac{3}{\alpha !} \int_0^1 (1-s)^2 (\partial_{\omega'}^\alpha \phi)(\frac{s \omega'}{\omega_{n-1}},1) ds \frac{(\omega')^\alpha}{\omega_{n-1}^2}.
\end{equation*}

Gao \emph{et al.} \cite{GaoLiuMiaoXi2023} showed that if $E$ is bounded in $C^N$ one can recover the narrow decoupling of general homogeneous phases satisfying a convexity condition in the constant coefficient case. This leads to the following notion of $K$-flatness.
\begin{definition}
Let $\phi$ be like above. We say that $\phi$ is $K$-flat if $E$ like in \eqref{eq:TaylorExpansionHomogeneousPhase}
satisfies
\begin{equation*}
\| \partial^\alpha E \|_{L^\infty(\Xi)} \leq C_{flat}
\end{equation*}
for $0 \leq |\alpha| \leq N_0$.
\end{definition}
In this definition $N_0$ and $C_{flat}$ are universal constants, which allow us to extend the arguments of \cite{Harris2019,OuWang2022} for the circular cone to $K$-flat phase functions as shown in \cite{GaoLiuMiaoXi2023}.

\medskip

To apply narrow decoupling for the variable coefficient operator on a small $K^2$-ball with $K^2 \lesssim \lambda^{\frac{1}{2}-\varepsilon}$, we approximate the variable coefficient phase with a constant coefficient phase. Beltran--Hickman--Sogge \cite{BeltranHickmanSogge2020} worked out that this is possible by Taylor expansion.

We need the following notations: Let $\phi$ be a reduced phase and $\bar{x} \in \R^n$, which will be the centre of the small ball on which we want to apply decoupling. Recall that $u \mapsto \partial_x \phi^\lambda(\bar{x};\Psi^\lambda(\bar{x};u))$ is a graph parametrization of the hypersurface $\Sigma_{\bar{u}}$. We have
\begin{equation*}
\langle x, (\partial_x \phi^\lambda) (\bar{x};\Psi^\lambda(\bar{x};u)) \rangle = \langle x',u \rangle + x_n h_{\bar{x}}(u)
\end{equation*}
for all $x = (x',x_n) \in \R^n$ with $h_{\bar{x}}(u) = (\partial_{x_n} \phi^\lambda) (\bar{x};\Psi^\lambda(\bar{x};u))$. Recall that we can suppose $a(x;\omega) = a_1(x) a_2(\omega)$ by Fourier series expansion. Let $E_{\bar{x}}$ denote the extension operator associated to $\Sigma_{\bar{x}}$, given by
\begin{equation*}
E_{\bar{x}} g(x) = \int_{\R^{n-1}} e^{i(\langle x',u \rangle + x_n h_{\bar{x}}(u))} a_{\bar{x}}(u) g(u) du \text{ for all } x \in \R^n,
\end{equation*}
where $a_{\bar{x}}(u) = a_2 \circ \Psi^\lambda(\bar{x};u) |\det \partial_u \Psi^\lambda(\bar{x};u)|$. We recall how $T^\lambda$ is approximated by $E_{\bar{x}}$: Let $x \in B(\bar{x};K^2) \subseteq B(0,3\lambda/4)$. By change of variables $\omega = \Psi^\lambda(\bar{x};u)$ and a Taylor expansion of $\phi^\lambda$ around $\bar{x}$, we have
\begin{equation*}
T^\lambda f(x) = \int_{\R^{n-1}} e^{i( \langle x-\bar{x}, (\partial_x \phi^\lambda)(\bar{x};\Psi^\lambda(\bar{x};u)) \rangle + \mathcal{E}^\lambda_{\bar{x}}(x-\bar{x};u))} a_1^\lambda(x) a_{\bar{x}}(u) f_{\bar{z}}(u) du
\end{equation*} 
with $f_{\bar{x}} = e^{i \phi^\lambda(\bar{x};\Psi^\lambda(\bar{x};\cdot))} f \circ \Psi^\lambda(\bar{x};\cdot)$ and by Taylor expansion
\begin{equation*}
\mathcal{E}^\lambda_{\bar{x}}(v;u) = \frac{1}{\lambda} \int_0^1 (1-r) \langle (\partial_{xx} \phi)((\bar{x}+rv)/\lambda;\Psi^\lambda(\bar{x};u)) v, v \rangle dr.
\end{equation*}
Let $N \gg 1$ be a large constant to be specified. By the derivative bounds
\begin{equation*}
\sup_{(v;u) \in B(0,K^2) \times \text{supp} a_{\bar{x}}} | \partial^\beta_\omega \mathcal{E}_{\bar{x}}^\lambda(v;u) | \lesssim_N 1
\end{equation*}
and Fourier series expansion, the oscillation of $\mathcal{E}^\lambda_{\bar{x}}$ can be neglected. This yields the following lemma:
\begin{lemma}[{\cite[Lemma~2.6]{BeltranHickmanSogge2020}}]
\label{lem:LinearApproximationLemma}
Let $T^\lambda$ be an oscillatory integral operator built from reduced data $(\phi,a)$. Let $0 < \delta \leq 1/2$, $1 \leq K^2 \leq \lambda^{\frac{1}{2}-\delta}$ and $\bar{x}/\lambda \in X$ so that $B(\bar{x};K^2) \subseteq B(0,3\lambda/4)$.
\begin{itemize}
\item Then
\begin{equation}
\label{eq:LinearApproximationI}
\| T^\lambda f \|_{L^p(w_{B(\bar{x};K^2)})} \lesssim_N \| E_{\bar{x}} f_{\bar{x}} \|_{L^p(w_{B(0;K^2)})} + \lambda^{-\frac{\delta N}{2}} \| f \|_{L^2}
\end{equation}
holds provided that $N$ is sufficiently large depending on $n$, $\delta$, and $p$, and $w_{B_{K^2}} = (1 + K^{-2} |x-\bar{x}|)^{-N}$ is a rapidly decaying weight off $B_{K^2}$.

\item Suppose that $|\bar{x}| \leq \lambda^{1-\delta'}$. There exists a family of operators $\mathbf{T}^\lambda$ all with phase $\phi$ and of type $(1,1,C)$ data (see Definition \ref{def:TypeData}) such that
\begin{equation}
\label{eq:LinearApproximationII}
\| E_{\bar{x}} f_{\bar{x}} \|_{L^p(w_{B(0,K^2)})} \lesssim_N \| T^\lambda_* f \|_{L^p(w_{B(\bar{x};K^2)})} + \lambda^{- \frac{N \min(\delta,\delta')}{2}} \| f \|_2
\end{equation}
holds for some $T^\lambda_* \in \mathbf{T}^\lambda$. The family $\mathbf{T}^\lambda$ has cardinality $O_N(1)$ and is independent of $B(\bar{x};K^2)$.
\end{itemize}
\end{lemma}
To apply the narrow decoupling to $E_{\bar{x}} f_{\bar{x}}$, we need the constant coefficient phase
\begin{equation*}
h_{\bar{x}}(u) = \partial_{x_n} \phi^\lambda(\bar{x};\Psi^\lambda(\bar{x};u))
\end{equation*}
to be $K$-flat.
\begin{definition}
Let $K \gg 1$. We say that a reduced homogeneous phase $\phi: \R^n \times \R^{n-1} \backslash \{ 0 \} \to \R$ is $K$-flat, if all its constant-coefficient approximations $h_{\bar{x}}$ for $\bar{x} \in X$ are $K$-flat and $\phi$ satisfies
\begin{itemize}
\item for $1 \leq k \leq n-1$ and $\beta = (\beta',\beta_{n-1}) \in \N_0^{n-1}$ with $|\beta| \leq N_0+5$ and $|\beta'| \geq 2$:
\begin{equation}
\label{eq:KFlatPhaseI}
 \| \partial_{x_k} \partial^\beta_{\omega} \phi \|_{L^\infty(X \times \Xi)} \leq K^{-4},
\end{equation}
\item for $\beta = (\beta',\beta_{n-1}) \in \N_0^{n-1}$ with $|\beta| \leq N_0+5$ and $|\beta'| \geq 3$:
\begin{equation}
\label{eq:KFlatPhaseII}
\| \partial_{x_n} \partial^{\beta}_{\omega} \phi \|_{L^\infty(X \times \Xi)} \leq K^{-4}.
\end{equation}
\end{itemize}
\end{definition}

The derivative bounds are used to facilitate induction: They will allow us to argue that the rescaled phase function becomes ``flatter". The choice of size of $\beta'$ in \eqref{eq:KFlatPhaseI} and \eqref{eq:KFlatPhaseII} will become clear from the formula for the phase function after rescaling. The size conditions \eqref{eq:KFlatPhaseI} and \eqref{eq:KFlatPhaseII} (but not the number of derivatives!) are more restrictive than the definition of a reduced phase function. We remark that with this definition, Proposition \ref{prop:DecouplingNarrowTerm} now follows from the constant-coefficient decoupling and the approximation by constant-coefficient operators provided by the previous lemma.

Regarding the choice of $N$ in Lemma \ref{lem:LinearApproximationLemma}: We choose $N = N(\delta) = N(\varepsilon) \geq N_0+5$ (since $\delta = \delta(\varepsilon)$) such that the error term $\lambda^{-\delta N/2} \| f \|_2$ becomes manageable when we carry out the induction on scales.


\begin{proposition}[Narrow~variable~coefficient~decoupling]
\label{prop:DecouplingNarrowTerm}
Suppose that $T^\lambda$ is an oscillatory integral operator with reduced $C1)$ and $C2^+)$ phase $\phi$, which is $K$-flat, and let $B_{K^2} \subseteq B(0,\lambda^{1-\delta})$ with $1 \leq K^2 \leq \lambda^{\frac{1}{2}-\delta}$, $0<\delta \leq 1/2$. If $V \subseteq \R^n$ is an $m$-dimensional linear subspace, then for $2 \leq p \leq \frac{2n}{n-2}$ and any $\delta>0$ one has
\begin{equation*}
\begin{split}
\big\| \sum_{\tau \in V} T^\lambda g_\tau \big\|_{L^p(B_{K^2})} &\lesssim_{\delta,N} \max(1, K^{(m-2)\big( \frac{1}{2}- \frac{1}{p} \big)}) K^\delta \big( \sum_{\tau \in V} \| T_*^\lambda g_\tau \|_{L^p(w_{B_{K^2}})}^p \big)^{\frac{1}{p}} \\
&\quad + \lambda^{-\frac{\delta N}{2}} \| g \|_{L^2}
\end{split}
\end{equation*}
provided that $N$ is chosen sufficiently large depending on $n$, $\delta$, and $p$. Here, the sum ranges over sectors $\tau$ for which $\angle(G^\lambda(\bar{x};\tau),V) \leq K^{-2}$, where $\bar{x}$ is the centre of $B_{K^2}$ and $w_{B_{K^2}} = (1+K^{-2} |x-\bar{x}|)^{-N}$ is a rapidly decaying weight off $B_{K^2}$. $T^\lambda_*$ is an oscillatory integral operator with phase $\phi$ and some amplitude $a_*$ chosen from a family of $O_N(1)$ many amplitudes. The amplitudes $a_*$ are reduced after uniform decomposition of the support of $a_*$.
\end{proposition}
\begin{proof}
We apply \eqref{eq:LinearApproximationI} from Lemma \ref{lem:LinearApproximationLemma} to approximate $T^\lambda$ with $E_{\bar{x}}$:
\begin{equation}
\label{eq:ApproximationA}
\| \sum_{\tau \in V} T^\lambda f_\tau \|_{L^p(B_{K^2})} \lesssim_N \| E_{\bar{x}} \sum_{\tau \in V} f_{\bar{x},\tau} \|_{L^p(w_{B_{K^2}})} + \lambda^{-\frac{\delta N}{2}} \| f \|_2.
\end{equation}
Since by our assumption on $\phi$, the underlying phase $h_{\bar{x}}$ for $E_{\bar{x}}$ is $K$-flat. Moreover, the transformed functions $f_{\bar{x},\tau}$ are disjoint sectors by diffeomorphism property of $\Psi^\lambda(\bar{x};\cdot)$. Hence, we can apply the narrow decoupling for $K$-flat constant-coefficient phase functions by \cite{GaoLiuMiaoXi2023}:
\begin{equation}
\label{eq:NarrowDecouplingConstantCoefficients}
\| \sum_{\tau \in V} E_{\bar{x}} f_{\bar{x},\tau} \|_{L^p(w_{B_{K^2}})} \lesssim_\delta K^\delta \max(1,K^{(m-2) \big( \frac{1}{2}-\frac{1}{p} \big)}) \big( \sum_{\tau \in V} \| E_{\bar{x}} f_{\bar{x},\tau} \|^p_{L^p(w_{B_{K^2}})} \big)^{\frac{1}{p}}.
\end{equation}
We can apply \eqref{eq:LinearApproximationII} from Lemma \ref{lem:LinearApproximationLemma} to find:
\begin{equation}
\label{eq:ApproximationB}
\| E_{\bar{x}} f_{\bar{x},\tau} \|_{L^p(w_{B_{K^2}})} \lesssim_{\delta,N} \| T^\lambda_* f_\tau \|_{L^p(w_{B_{K^2}})} + \lambda^{-\frac{\delta N}{2}} \| f_\tau \|_{L^2}.
\end{equation}
Recall that $T^\lambda_*$ is a variable-coefficient extension operator from a family $\mathbf{T}^\lambda$, which is of size $O_N(1)$. The operators from $\mathbf{T}^\lambda$ have phase $\phi$, but possibly different amplitude $a_*$, which is independent of $\bar{x}$, but still satisfies the bounds
\begin{equation*}
|\partial_\xi^\alpha a_* | \lesssim_N 1 \text{ for } \alpha \in \N_0^{n-1}, \; |\alpha| \leq N.
\end{equation*}
Hence, taking the $p$-th power, summing over the sectors, and pigeonholing in $T^\lambda_*$ yields
\begin{equation*}
\big( \sum_\tau \| E_{\bar{x}} f_{\bar{x},\tau} \|_{L^p(w_{B_{K^2}})}^p \big)^{\frac{1}{p}} \lesssim_{\delta,N} \big( \sum_{\tau} \| T^\lambda_* f_\tau \|_{L^p(w_{B_{K^2}})}^p \big)^{\frac{1}{p}} + \lambda^{-\frac{\delta N}{2}} \| f \|_{L^p}.
\end{equation*}
We take \eqref{eq:ApproximationA}, \eqref{eq:NarrowDecouplingConstantCoefficients}, and \eqref{eq:ApproximationB} together to find
\begin{equation*}
\begin{split}
\big\| \sum_{\tau \in V} T^\lambda f_\tau \big\|_{L^p(B_{K^2})} &\lesssim_{N,\delta} K^\delta \max(1,K^{(m-2) \big( \frac{1}{2}-\frac{1}{p} \big)}) \big( \sum_\tau \| T^\lambda_* f_\tau \|^p_{L^p(w_{B_{K^2})})} \big)^{\frac{1}{p}} \\
&\quad + \lambda^{\frac{n}{2}-\frac{\delta N}{2}} \| f \|_p.
\end{split}
\end{equation*}
\end{proof}

We choose $N=5n/\delta$, which keeps the error term $\lambda^{-\frac{\delta N}{2}} \| g \|_{L^2}$ manageable even after summing over the balls. 

\medskip

The technicality of dealing with different amplitude functions is handled by considering an appropriate class of data $(\phi,a)$, for which the induction on scales is carried out:
\begin{definition}
For $1 \leq p \leq \infty$ and $R \geq 1$ let $Q_{p,\delta}(R)$ denote the infimum over all constants $C$ for which the estimate
\begin{equation*}
\| T^\lambda f \|_{L^p(B(0,r))} \leq C \| f \|_{L^p(A^{n-1})}
\end{equation*}
holds for $1 \leq r \leq R \leq \lambda$ and all oscillatory integral operators $T^\lambda$ built from reduced data $(\phi,a)$ with $\lambda^{\delta}$-flat phase function.
\end{definition}

\medskip

\subsection{Parabolic rescaling}

In this section we shall see how parabolic rescaling flattens the phase function. Moreover, we observe how for functions supported on a small sector, we can apply the estimate from the previous definition on the smaller scale.

\subsubsection{Change of spatial variables}
For the parabolic rescaling we also make use of a change of variables on the spatial side.

\begin{lemma}[Change~of~spatial~variables]
\label{lem:ChangeSpatialVariables}
Let $x = (x'',x_{n-1},x_n) \in \R^{n-2} \times \R \times \R$ and $\phi$ be a reduced phase function. For $x_n \in T$, $\omega \in \Xi$ there is a smooth mapping $\Upsilon(\cdot,x_n;\omega)$, which satisfies
\begin{equation*}
\partial_{\omega'} \phi(\Upsilon(x^\prime,x_n;\omega),x_n;\omega) = x^{\prime \prime} \text{ and } \phi(\Upsilon(x;\omega),x_n;\omega) = x_{n-1}.
\end{equation*}
Moreover, we have the uniform derivative bounds:
\begin{equation}
\label{eq:DerivativeBoundUpsilon}
|\partial_{x'} \Upsilon| \leq C_\Upsilon.
\end{equation}
\end{lemma}
\begin{proof}
We shall apply the implicit function theorem. We consider the defining equations
\begin{equation*}
F(\Upsilon(x',x_n;\omega),x_n;\omega) = 
\begin{pmatrix}
\partial_{\omega'} \phi(\Upsilon(x',x_n;\omega),x_n;\omega) \\
\phi(\Upsilon(x;\omega),x_n;\omega)
\end{pmatrix}
=
\begin{pmatrix}
x'' \\ x_{n-1}
\end{pmatrix}
\end{equation*}
and differentiate:
\begin{equation*}
\underbrace{\begin{pmatrix}
\partial^2_{x' \omega} \phi(\Upsilon(x',x_n;\omega),x_n;\omega) \\ \partial_{x'} \phi( \Upsilon( x',x_n;\omega), x_n;\omega)
\end{pmatrix}}_{A}
\partial_{x'} \Upsilon(x',x_n;\omega) = 1_{n-1}.
\end{equation*}
By $1$-homogeneity we have
\begin{equation*}
\partial_{x'} \phi(x;\omega) = \sum_{j=1}^{n-1} \omega_j \cdot \partial^2_{x' \omega_j} \phi(x;\omega).
\end{equation*}
Thus, for each $x_n \in T$ and $\omega \in \Xi$, the Jacobian determinant of the map $x' \mapsto ((\partial_{\omega'} \phi(x;\omega),\phi(x;\omega))$ is given by $\omega_{n-1} \cdot \det \partial^2_{\omega x'} \phi(x;\omega)$ and hence is non-vanishing. In fact, it is uniformly bounded from above and below for reduced phases.

Hence, $\partial_{x'} F(\Upsilon(x',x_n;\omega),x_n;\omega)$ is invertible, and therefore $\Upsilon(\cdot,x_n;\omega)$ exists for $x_n \in T$, $\omega \in \Xi$. We obtain by Cramer's rule
\begin{equation*}
\partial_{x'} \Upsilon(x',x_n;\omega) = 
\begin{pmatrix}
\partial^2_{x' \omega} \phi(\Upsilon(x',x_n;\omega)) \\ \partial_{x'} \phi(x',x_n;\omega)
\end{pmatrix}
^{-1} = \frac{\text{ad} (A)}{\det (A)}.
\end{equation*}
By reducedness of $\phi$, we have $\det A \sim 1$. Secondly, the components of $\partial^2_{x' \omega} \phi$ are uniformly bounded for a reduced phase function and by homogeneity,
\begin{equation*}
\partial_{x'} \phi(x',x_n;\omega) = \sum_{i=1}^{n-1} \omega_i \partial^2_{x' \omega_i} \phi(x',x_n;\omega).
\end{equation*}
This also gives a uniform estimate on $\partial_{x'} \phi$ for a reduced phase function.
\end{proof}

\subsubsection{Phase and amplitude function after rescaling}

We carry out the parabolic rescaling $T^\lambda g$ for a function $g$ supported in a $\rho^{-1}$-sector. In \cite{BeltranHickmanSogge2020} the phase function was computed. It was shown how after rescaling we find the bounds for higher order derivatives introduced in Section \ref{subsection:BasicReductions} to hold. For an arbitrary phase function $\phi$, $\rho$ has to be chosen large enough depending on $\phi$. For phase functions, which are reduced before rescaling, $\rho$ can be chosen uniform; see also Lemma \ref{lem:ReducedPhaseFunction} and its proof. Since we need some expressions from \cite{BeltranHickmanSogge2020} to verify that the phase is ``flattened" in the sense of \eqref{eq:KFlatPhaseI} and \eqref{eq:KFlatPhaseII} upon rescaling, some details are repeated. Furthermore, the amplitude and its derivatives satisfy stronger bounds aswell. However, the scale of the amplitude exceeds the scale of the phase function, but this can be remedied by an argument due to Guth--Hickman--Iliopoulou \cite{GuthHickmanIliopoulou2019}; see the end of the proof of Lemma \ref{lem:ParabolicRescaling}.

\medskip

Let $\omega \in B_{n-2}(0,1)$ with $(\omega,1)$ the centre of the $\rho^{-1}$-sector encasing the support of $g$:
\begin{equation*}
\text{supp} (g) \subseteq \{ (\xi',\xi_{n-1}) \in \R^{n-1} : \, 1/2 \leq \xi_{n-1} \leq 2 \text{ and } \big| \frac{\xi'}{\xi_{n-1}} - \omega \big| \leq \rho^{-1} \}.
\end{equation*}
We perform the change of variables 
\begin{equation*}
(\xi',\xi_{n-1}) = (\eta_{n-1} \omega + \rho^{-1} \eta', \eta_{n-1}),
\end{equation*}
after which follows
\begin{equation*}
T^\lambda g(x) = \int_{\R^{n-1}} e^{i \phi^\lambda(x;\eta_{n-1} \omega + \rho^{-1} \eta', \eta_{n-1})} a^\lambda(x;\eta_{n-1} \omega + \rho^{-1} \eta', \eta_{n-1}) \tilde{g}(\eta) d\eta,
\end{equation*}
where $\tilde{g}(\eta) = \rho^{-(n-2)} g(\eta_{n-1} \omega + \rho^{-1} \eta',\eta_{n-1})$ and $\text{supp} (\tilde{g}) \subseteq \Xi$. By Taylor expansion and homogeneity of the phase, we find
\begin{equation*}
\begin{split}
\phi(x;\eta_{n-1} \omega + \rho^{-1} \eta', \eta_{n-1}) &= \phi(x;\omega,1) \eta_{n-1} + \rho^{-1} \langle \partial_{\omega'} \phi (x;\omega,1), \eta' \rangle \\
&\quad + \rho^{-2} \int_0^1 (1-r) \langle \partial^2_{\omega' \omega'} \phi(x;\eta_{n-1} \omega + r \rho^{-1} \eta', \eta_{n-1}) \eta', \eta' \rangle dr.
\end{split}
\end{equation*}
Let $\Upsilon_\omega(y',y_n) = (\Upsilon(y',y_n;\omega,1),y_{n})$ and $\Upsilon^\lambda_\omega(y',y_n) = \lambda \Upsilon_\omega(y'/\lambda,y_n/\lambda)$ and consider anisotropic dilations
\begin{equation*}
D_\rho( y'', y_{n-1}, y_n) = (\rho y'',y_{n-1}, \rho^2 y_n) \text{ and } D'_{\rho^{-1}}(y'',y_{n-1}) = (\rho^{-1} y'', \rho^{-2} y_{n-1})
\end{equation*}
on $\R^n$ and $\R^{n-1}$, respectively. By definition of $\Upsilon$, we find
\begin{equation}
\label{eq:ChangeVariablesRescaling}
T^\lambda g \circ \Upsilon^\lambda_\omega \circ D_\rho = \tilde{T}^{\lambda/\rho^2} \tilde{g}
\end{equation}
where
\begin{equation*}
\tilde{T}^{\lambda/\rho^2} \tilde{g}(y) = \int_{\R^{n-1}} e^{i \tilde{\phi}^{\lambda/\rho^2}(y;\eta)} \tilde{a}^\lambda(y;\eta) \tilde{g}(\eta) d\eta
\end{equation*}
for the phase $\tilde{\phi}(y;\eta)$ given by
\begin{equation*}
\langle y',\eta \rangle + \int_0^1 (1-r) \langle \partial^2_{\xi' \xi'} \phi(\Upsilon_\omega(D'_{\rho^{-1}} y',y_{n}); \eta_{n-1} \omega + r \rho^{-1} \eta', \eta_{n-1}) \eta', \eta' \rangle dr
\end{equation*}
and the amplitude $\tilde{a}(y;\eta) = a(\Upsilon_\omega(D'_{\rho^{-1}}y',y_n);\eta_{n-1} \omega + \rho^{-1} \eta', \eta_{n-1})$. 

We make a harmless linear change of variables: Let $L \in GL(n-1;\R)$ be such that $L e_{n-1} = e_{n-1}$ and
\begin{equation*}
\partial^2_{\eta' \eta'} \partial_{y_n} \tilde{\phi}_L (0,0;e_{n-1}) = I_{n-2},
\end{equation*}
where
\begin{equation*}
\tilde{\phi}_L(y;\eta) = \tilde{\phi}(L^{-1} y',y_n; L \eta).
\end{equation*}
It suffices to analyze $\tilde{T}^{\lambda/\rho^2}_L \tilde{g}_L$ with $\tilde{T}^{\lambda/\rho^2}_L$ defined with respect to the data $(\tilde{\phi}_L,\tilde{a}_L)$ for $\tilde{\phi}_L$ as above, $\tilde{a}_L(y;\eta) = \tilde{a}(L^{-1} y',y_n;L \eta)$ and $\tilde{g}_L = | \det L| \tilde{g} \circ L$.\\
To see that $\tilde{\phi}_L$ is still a reduced phase, we use the representations
\begin{equation}
\label{eq:RepresentationTildePhiI}
\tilde{\phi}_L(y;\eta) = \rho^2 \phi(\Upsilon_\omega(D'_{\rho^{-1}} \circ L^{-1} y', y_n), y_n; \eta_{n-1} \omega + \rho^{-1} L' \eta', \eta_{n-1})
\end{equation}
and
\begin{equation}
\label{eq:RepresentationTildePhiII}
\langle y',\eta \rangle + \int_0^1 (1-r) \langle \partial^2_{\omega' \omega'} \phi(\Upsilon_\omega(D'_{\rho^{-1}} \circ L^{-1} y',y_n);\eta_{n-1} \omega + r \rho^{-1} L' \eta', \eta_{n-1}) L' \eta', L' \eta' \rangle dr,
\end{equation}
where $L'$ denotes the $(n-2) \times (n-2)$-submatrix of $L$, containing the first $n-2$ rows and columns. We can argue like in \cite{BeltranHickmanSogge2020}, that $\tilde{\phi}_L$ is again a reduced phase provided that $\phi$ was a reduced phase. Note that the notion of reduced phase is slightly weaker in \cite{BeltranHickmanSogge2020}. We have strengthened the second condition in $D1)$ for technical reasons.
For reduced phase functions the components of $L$ are uniformly bounded (see \cite{BeltranHickmanSogge2020}). Regarding $K$-flatness of the rescaled phase functions, we have the following:

\begin{lemma}
\label{lem:ReducedPhaseFunction}
Suppose that $\phi$ is a reduced $K$-flat phase function supported in a sector of aperture $\sim \rho^{-1}$. Then $\tilde{\phi}_L$ is a reduced $c K \rho^{\frac{1}{4}}$-flat phase function provided that $\rho$ is chosen large enough.
\end{lemma}

\begin{proof}
We begin with the verification of \eqref{eq:KFlatPhaseI}. Let 
\begin{equation*}
C_{m,\rho^{-1}} = \begin{cases}
\rho^{-1}, \quad m=1,\ldots,n-2,\\
\rho^{-2}, \quad m=n-1.
\end{cases}
\end{equation*}
From the representation \eqref{eq:RepresentationTildePhiI} we find for $k=1,\ldots,n-1$
\begin{equation*}
\begin{split}
\partial_{y_k} \tilde{\phi}_L &= \rho^2 \sum_{\ell = 1}^{n-1} \partial_{y_\ell} \phi(\Upsilon_\omega(D'_{\rho^{-1}} \circ L^{-1} y',y_n);\eta_{n-1} \omega + \rho^{-1} L' \eta', \eta_{n-1}) \\
&\qquad \times \frac{\partial}{\partial y_k} \Upsilon^\ell (D'_{\rho^{-1}} \circ L^{-1} y',y_n;\omega,1) \\
&= \rho^2 \sum_{\ell,m=1}^{n-1} \partial_{y_\ell} \phi(\Upsilon_\omega(D'_{\rho^{-1}} \circ L^{-1} y',y_n);\eta_{n-1} \omega + \rho^{-1} L \eta', \eta_{n-1}) \\
&\qquad \times \frac{\partial \Upsilon^\ell}{\partial y_m}(D'_{\rho^{-1}} \circ L^{-1} y',y_n;\omega,1) C_{m,\rho^{-1}} L^{-1}_{mk}.
\end{split} 
\end{equation*}
We take derivatives in $\eta'$, which gives factors of $\rho^{-1}$ and components of $L$: Recall that these are uniformly bounded. We obtain by $C_{m,\rho^{-1}} \leq \rho^{-1}$ and $|\partial \Upsilon| \leq C_\Upsilon$ for $2 \leq |\alpha| \leq N_0 + 5$ by hypothesis:
\begin{equation*}
\begin{split}
|\partial_{y_k} \partial^\alpha_{\eta'} \tilde{\phi}_L| &\leq \rho^{2} \rho^{-1} \rho^{-|\alpha|} C_L^{|\alpha|+1} \sum_{|\tilde{\alpha}| = |\alpha|} \sum_{\ell = 1}^{n-1} |\partial_{y_\ell} \partial_{\eta'}^{\tilde{\alpha}} \phi| \\
&\leq \rho^{1-|\alpha|} C_L^{|\alpha|+1} C_\Upsilon C_{N_0} K^{-4}.
\end{split}
\end{equation*}
Taking derivatives in $\eta_{n-1}$ only changes the result by a constant (use e.g. homogeneity in $\eta_{n-1}$): For $2 \leq |\alpha| \leq N_0 + 5$ and $0 \leq |\beta| \leq N_0 + 5 - |\alpha|$ we obtain
\begin{equation*}
|\partial_{y_k} \partial^\alpha_{\eta'} \partial_{\eta_{n-1}}^\beta \tilde{\phi}_L| \leq \rho^{1-|\alpha|} C_L^{|\alpha|+|\beta|+1} C_\Upsilon C_{N_0} K^{-4}.
\end{equation*}
For $\rho = \rho(N_0,n,C_\Upsilon)$ large enough this verifies \eqref{eq:KFlatPhaseI} with $c \rho^{\frac{1}{4}} K$ instead of $K$.

\medskip

We turn to the proof of \eqref{eq:KFlatPhaseII}. By the chain rule
\begin{equation*}
\begin{split}
\partial_{y_n} \tilde{\phi}_L(y;\eta) &= \rho^2 \sum_{k=1}^{n-1} \partial_{y_k} \phi(\Upsilon_\omega(D'_{\rho^{-1}} \circ L^{-1} y',y_n); \eta_{n-1} \omega + \rho^{-1} L' \eta', \eta_{n-1}) \\
&\quad \quad \times \frac{\partial \Upsilon^k_\omega}{\partial x_n}(D'_{\rho^{-1}} \circ L^{-1} y',y_n) \\
&\quad + \rho^2 \partial_{y_n} \phi(\Upsilon_\omega(D'_{\rho^{-1}} \circ L^{-1} y',y_n); \eta_{n-1} \omega + \rho^{-1} L' \eta', \eta_{n-1})
\end{split}
\end{equation*}
We find when taking derivatives in $\eta'$:
\begin{equation*}
|\partial^\alpha_{\eta'} \partial_{y_n} \tilde{\phi}_L(y;\eta)| \leq \rho^2 C_L^{|\alpha|} \rho^{-|\alpha|} C_\Upsilon \sum_{k=1}^n |\partial_{x_k} \partial^\alpha_{\eta'} \phi|.
\end{equation*}
For $|\alpha| \geq 3$ we find by hypothesis
\begin{equation*}
|\partial^\alpha_{\eta'} \partial_{y_n} \tilde{\phi}_L(y;\eta)| \leq \rho^{2-|\alpha|} C_L^{|\alpha|} C_\Upsilon C_{N_0,n} K^{-4}.
\end{equation*}
This suffices to conclude like above. Likewise the argument with additional derivatives in $\eta_n$ applies.

\medskip

To show $K$-flatness for the constant-coefficient approximations, we first suppose that $\psi(\bar{x};u) = u$. Then we find
\begin{equation*}
\begin{split}
h_{\bar{x}}(\eta) &= \eta_{n-1} \rho^2 \sum_{k=1}^{n-1} \partial_{x_k} \phi^\lambda(\Upsilon^\lambda_\omega(D'_{\rho^{-1}} \circ L^{-1} y',y_n);\omega + \rho^{-1} L' \frac{\eta'}{\eta_{n-1}},1) \\
&\qquad \times \frac{\partial \Upsilon^k_\omega}{\partial_{x_n}}(D'_{\rho^{-1}} \circ L^{-1} y',y_n) \\
&\quad + \eta_{n-1} \rho^2 \partial_{x_n} \phi^\lambda(\Upsilon^\lambda_\omega(D'_{\rho^{-1}} \circ L^{-1} y',y_n);\omega + \rho^{-1} \frac{L' \eta'}{\eta_{n-1}}, 1).
\end{split}
\end{equation*}
We need to prove boundedness of
\begin{equation*}
E_{h_{\bar{x}}}(\eta) = \rho K^4 \sum_{|\alpha| = 3} \frac{3}{\alpha!} \int_0^1 (1-s)^2 (\partial^\alpha_{\eta'} h_{\bar{x}})( \frac{s \eta'}{\eta_{n-1}}, 1) ds \frac{(\eta')^\alpha}{\eta_{n-1}^2}.
\end{equation*}
By \eqref{eq:KFlatPhaseI} and \eqref{eq:KFlatPhaseII} we find
\begin{equation}
\label{eq:KFlatnessApproximation}
|\partial^\alpha_\eta E_{h_{\bar{x}}}(\eta)| \lesssim 1 \text{ for } 0 \leq |\alpha| \leq N_0.
\end{equation}
However, in the general case $\psi(\bar{x};\cdot)$ is not the identity mapping, and we need to prove bounds for the derivatives up to order $N_0 + 3$.

We show that
\begin{equation*}
|\partial_u \psi(\bar{x};u)| \lesssim 1 \text{ and } |\partial_u^\alpha \psi(\bar{x};u)| \lesssim \rho^{-1} \text{ for } 2 \leq |\alpha | \leq N_0 +3.
\end{equation*}
The first estimate is immediate from
\begin{equation*}
\partial_{x'} \phi^\lambda(\bar{x};\psi(\bar{x};u)) = u \Rightarrow \partial^2_{x' \omega} \phi(\bar{x};\psi(\bar{x};u)) \partial_u \psi(\bar{x};u) = 1_{n-1}.
\end{equation*}
Hence, $\partial_u \psi(\bar{x};u)= (\partial^2_{x' \omega} \phi^\lambda(\bar{x};\psi(\bar{x};u)))^{-1}$. This proves the first estimate since $\phi$ is a reduced phase.

For the second estimate we consider
\begin{equation*}
\begin{split}
u_i &= \psi_i(\bar{x};u) + \sum_{\ell} C_{\ell,\rho^{-1}} C_L \int_0^1 (1-r) \sum_{j,k=1}^{n-1} \partial_{x_{\ell}'} \partial^2_{\omega_j' \omega_k'} \phi(\Upsilon_\omega(D'_{\rho^{-1}} \circ L^{-1} y',y_{n}); \\
&\quad \psi_{n-1} \omega + r \rho^{-1} \psi'(\bar{x};u) \psi_{n-1}(\bar{x};u)) \frac{\partial \Upsilon^{\ell}_\omega}{\partial x_i}(D'_{\rho^{-1}} \circ L^{-1} y',y_n) \psi_j(\bar{x};u) \psi_k(\bar{x};u) dr.
\end{split}
\end{equation*}
From this expression we can argue by induction that
\begin{equation*}
\partial_u^\alpha \psi_i(\bar{x};u) = O(\rho^{-1}) \partial_u^\alpha \psi_i(\bar{x};u) + O(\rho^{-1}) \text{ for } 2 \leq |\alpha| \leq N_0 + 3.
\end{equation*}
Here we use \eqref{eq:KFlatPhaseI}, Lemma \ref{lem:ChangeSpatialVariables}, and the chain rule. This yields $|\partial_u^\alpha \psi_i(\bar{x};u)| \lesssim_{\Upsilon,N_0,n} \rho^{-1}$ and extends \eqref{eq:KFlatnessApproximation} to the general case.
\end{proof}

\begin{lemma}[Parabolic~rescaling]
\label{lem:ParabolicRescaling}
Let $\text{supp} (f) \subseteq \Xi$ be supported in a $\rho^{-1}$-sector and $\phi$ be a reduced phase, that is $\lambda^{\delta}$-flat. Then, for any $1 \leq \rho \leq R \leq \lambda$:
\begin{equation}
\label{eq:LorentzRescaling}
\| T^\lambda f \|_{L^p(B(0,R))} \lesssim_{\delta'} R^{\delta'} Q_{p,\delta}(R/\rho) \rho^{\frac{2(n-1)}{p}-(n-2)} \| f \|_{L^p}.
\end{equation}
\end{lemma}

The proof combines arguments from \cite{BeltranHickmanSogge2020} and \cite{GuthHickmanIliopoulou2019}. With many preliminaries already settled in Lemma \ref{lem:ReducedPhaseFunction}, we shall be brief.

\begin{proof}
By a change of spatial variables, we find from \eqref{eq:RescalingChangeSpatialVariables}
\begin{equation}
\label{eq:RescalingChangeSpatialVariables}
\| T^\lambda g \|_{L^p(B_R)} \lesssim \rho^{\frac{n}{p}} \| \tilde{T}^{\lambda/\rho^2} \tilde{g} \|_{L^p((\Upsilon^\lambda_\omega \circ D_\rho)^{-1}(B_R))}.
\end{equation}
We have proved in Lemma \ref{lem:ReducedPhaseFunction} that the phase $\tilde{\phi}_L$ used for the operator $\tilde{T}$ is still reduced and $c \rho^{\frac{1}{4}} \lambda^\delta$ flat with $c$ a universal constant. In particular, for $\rho$ large enough, the phase function is $(\lambda/\rho^2)^\delta$-flat.

We want to use the definition of $Q_{p,\delta}(R/\rho^2)$ with $\lambda/\rho^2$ playing the role of $\lambda$. However, $(\Upsilon^\lambda_\omega \circ D_{\rho})^{-1} (B_R)$ is an approximate ellipsoid of dimensions $\sim R/\rho \times \ldots \times R/\rho \times R  \times R/\rho^2$ and possibly not contained anymore in $B(0,\lambda/\rho^2)$. Nonetheless, in \cite{GuthHickmanIliopoulou2019} was argued for a different class of phase functions that the estimate
\begin{equation*}
\| T^\lambda f \|_{L^p(D_{\textbf{R}})} \lesssim_{\delta'} Q_{p,\delta}(R) R^{\delta'} \| f \|_{L^p}
\end{equation*}
still holds for an ellipsoid with $1 \leq R' \leq R$ and $R \leq \lambda$:
\begin{equation*}
D_{\textbf{R}} = \{ x \in \R^n : \big( \frac{|x'|}{R'} \big)^2 + \big( \frac{|x_n|}{R} \big)^2 \leq 1 \}.
\end{equation*}
The argument is based on discretization by the essentially constant property and orthogonality between balls of size $R$, on which the estimate at smaller scale can be applied. In the present context the balls of size $R/\rho^2$ are centered at the origin such that the containment in $B(0,\lambda/\rho^2)$ is achieved.
The arguments from \cite[Section~11.2]{GuthHickmanIliopoulou2019} extend without significant change to the present case, for which reason the details are omitted. This completes the proof.
\end{proof}

\subsection{Proof of Proposition \ref{prop:LinearFromBroadEstimates}}

With the narrow decoupling and parabolic rescaling at hand, we can now derive linear estimates from broad estimates:
\begin{proof}[Proof of Proposition \ref{prop:LinearFromBroadEstimates}]
Let $\varepsilon > 0$. By interpolation it suffices to prove the linear estimate for $p$ satisfying the additional constraint
\begin{equation*}
p(k,n) < p.
\end{equation*}

In the first step, for $\lambda \gg 1$, we carry out a parabolic rescaling depending on the phase so that we can reduce to $\lambda^{\tilde{\delta}}$-flat phase functions by Lemma \ref{lem:ReducedPhaseFunction}. This loses a factor $C_\phi \lambda^{10 n \tilde{\delta} }$ by partitioning $\Xi$ into sectors, which will be admissible provided that 
\begin{equation}
\label{eq:InitialRescaling}
10 n \, \tilde{\delta} \leq \varepsilon/10.
\end{equation}
We shall next prove that $Q_{p,\tilde{\delta}}(R) \leq C_\varepsilon R^{\frac{3\varepsilon}{4}}$. In the following let $K=K_0 R^\eta \leq \lambda^{\tilde{\delta}}$ with $K_0$ and $\eta$ to be determined (see \eqref{eq:ChoiceK}).

\medskip

\noindent By the assumed $k$-broad estimate, we find
\begin{equation}
\label{eq:kBroadEstimateI}
\sum_{\substack{B_{K^2} \in \mathcal{B}_{K^2}, \\ B_{K^2} \cap B(0,R) \neq \emptyset}} \min_{V_1,\ldots,V_A} \max_{\tau \notin V_a} \int_{B_{K^2}} |T^\lambda f_\tau |^p \leq \tilde{C}_\varepsilon K^{C_\varepsilon} R^{\frac{p \varepsilon}{2}} \| f \|_{L^p(A^{n-1})}^p,
\end{equation}
where $V_1$,...,$V_A$ are $(k-1)$-planes and $\tau \notin V_a$ is short-hand for
\begin{equation*}
\angle (G^\lambda(\bar{x};\tau),V_a) > K^{-2},
\end{equation*}
with $\bar{x}$ being the centre of $B_{K^2}$. Moreover, we can suppose that $R \leq \lambda^{1-\delta_1}$ by covering $B(0,\lambda)$ with $R$-balls and losing an additional factor $\lambda^{10 n \delta_1}$, which is admissible provided that
\begin{equation}
\label{eq:ConstraintDeltaPrime}
10 n \delta_1 \leq \frac{\varepsilon}{10}.
\end{equation}

We choose $V_1$,...,$V_A$ for each $B_{K^2}$, which attains the minimum in \eqref{eq:kBroadEstimateI}. By this, we may write
\begin{equation*}
\int_{B_{K^2}} |T^\lambda f |^p \leq K^{10n} \max_{\tau \notin V_a} \int_{B_{K^2}} |T^\lambda f_\tau|^p + \sum_{a=1}^A \int_{B_{K^2}} \big| \sum_{\tau \in V_a} T^\lambda f_\tau \big|^p.
\end{equation*}
By summing over $B_{K^2}$ and using \eqref{eq:kBroadEstimateI}, we find
\begin{equation*}
\int_{B(0,R)} |T^\lambda f|^p \leq K^{10 n} \tilde{C}_\varepsilon K^{C_\varepsilon} R^{p \varepsilon/2} \| f \|^p_{L^p} + \sum_{\substack{B_{K^2} \in \mathcal{B}_{K^2}, \\ B_{K^2} \cap B(0,R) \neq \emptyset}} \sum_{a=1}^A \int_{B_{K^2}} \big| \sum_{\tau \in V_a} T^\lambda f_\tau \big|^p.
\end{equation*}
By the decoupling result Proposition \ref{prop:DecouplingNarrowTerm}, we find for any $\delta_1 > 0$, provided that $K \leq \lambda^{\tilde{\delta}}$ and $B_{K^2} \subseteq B(0,\lambda^{1-\delta_1})$,
\begin{equation*}
\int_{B_{K^2}} \big| \sum_{\tau \in V_a} T^\lambda f_\tau \big|^p \lesssim_{\delta_1} K^{\max((k-3)(\frac{p}{2}-1),0) + \delta_1} \sum_{\tau \in V_a} \int_{\R^n} |T^\lambda_* f_\tau|^p w_{B_{K^2}} + \lambda^{p \big( \frac{n}{2}- \frac{\delta_1 N}{2} \big)} \| f \|^p_{L^p}
\end{equation*}
and summing over $a$ and $B_{K^2}$, we find 
\begin{equation}
\label{eq:NarrowAfterDecoupling}
\begin{split}
\sum_{B_{K^2} \in \mathcal{B}_{K^2}} \sum_{a=1}^A \int_{B_{K^2}} \big| \sum_{\tau \in V_a} T^\lambda f_\tau \big|^p &\lesssim_{\delta_1} K^{\max((k-3)(p/2-1),0)+ \delta_1} \sum_{\tau: K^{-1}} \int_{B(0,2R)} |T_*^\lambda f_\tau |^p \\
&\quad + \| f \|^p_{L^p}.
\end{split}
\end{equation}
Here we make use of the choice $N=N(\delta_1)$ large enough. We shall choose $\delta_1 = \delta_1(\varepsilon,p,k,n)$. This requires to possibly increase $N$ above the needs of Section \ref{section:MainInductiveArgument}.

The separated expressions $T_*^\lambda f_\tau$ are amenable to Lemma \ref{lem:ParabolicRescaling}, which gives
\begin{equation}
\label{eq:RescaledVersions}
\int_{B(0,2R)} |T^\lambda_* f_\tau|^p \lesssim_{\delta_2} (Q_{p,\tilde{\delta}}(R))^p R^{\delta_2} K^{2(n-1)-(n-2)p} \| f_\tau \|^p_{L^p}.
\end{equation}
Let
\begin{equation*}
-e(k,p) = 2(n-1) - (n-2)p + \max((k-3) \big( \frac{p}{2} - 1 \big), 0).
\end{equation*}

Plugging \eqref{eq:RescaledVersions} into \eqref{eq:NarrowAfterDecoupling}, we find
\begin{equation*}
\int_{B(0,R)} |T^\lambda f|^p \leq (K^{10 n} \tilde{C}_\varepsilon K^{C_\varepsilon} R^{p \varepsilon/2} + C_{\delta_1,\delta_2} (Q_{p,\tilde{\delta}}(R))^p R^{\delta_2} K^{-e(k,p)+\delta_1} ) \| f \|^p_{L^p(A^{n-1})}.
\end{equation*}
This yields
\begin{equation}
\label{eq:InductionLoop}
(Q_{p,\tilde{\delta}}(R))^p \leq K^{10 n} \tilde{C}_\varepsilon K^{C_\varepsilon} R^{p \varepsilon/2} + C_{\delta,\delta'} (Q_{p,\tilde{\delta}}(R))^p R^\delta K^{-e(k,p)+\delta'}.
\end{equation}
Since $p$ is as in \eqref{eq:ConditionsP}, we find $e(k,p)>0$, and may choose $\delta_1 = \min( e(k,p)/2, \frac{\varepsilon}{100 n})$, so that the $K$ exponent in the second term on the right-hand side is negative and \eqref{eq:ConstraintDeltaPrime} is satisfied.

Moreover, we can choose $\delta_2$ small enough such that 
\begin{equation}
\label{eq:ChoiceDelta}
\frac{2 \delta_2}{e(k,p)} C_\varepsilon \leq \frac{p \varepsilon}{8} \text{ and } \frac{10n \delta_2}{2 e(k,p)} \leq \frac{\varepsilon}{60} \leq \frac{p \varepsilon}{8}.
\end{equation}
If
\begin{equation}
\label{eq:ChoiceK}
K=K_0 R^{\frac{2 \delta_2}{e(k,p)}}
\end{equation}
 for a sufficiently large $K_0$, depending on $\varepsilon$, $\delta_2=\delta_2(\varepsilon)$, $p$ and $n$, \eqref{eq:ChoiceDelta} ensures for the first term on the right-hand side of \eqref{eq:InductionLoop}:
\begin{equation*}
K^{10 n} K^{C_\varepsilon} R^{p \varepsilon/2} \leq \tilde{C}_\varepsilon K_0^{10n + C_\varepsilon} R^{\frac{ 3p \varepsilon}{4}} = \tilde{D}_\varepsilon R^{\frac{ 3p \varepsilon}{4}}.
\end{equation*}
We find for the second term on the right-hand side of \eqref{eq:InductionLoop}:
\begin{equation*}
\begin{split}
C_{\delta_1,\delta_2} (Q_{p,\tilde{\delta}}(R))^p R^\delta K^{-e(k,p) + \delta_1} &\leq C_{\delta_1,\delta_2} (Q_{p,\tilde{\delta}}(R))^p R^{\delta_2} (K_0 R^{\frac{2 \delta_2}{e(k,p)}} )^{-\frac{e(k,p)}{2}} \\
&= C_{\delta_1,\delta_2} K_0^{-\frac{e(k,p)}{2}} (Q_{p,\tilde{\delta}}(R))^p \leq \frac{1}{2} (Q_{p,\tilde{\delta}}(R))^p.
\end{split}
\end{equation*}
The ultimate estimate follows from choosing $K_0=K_0(\delta_1,\delta_2,k,p,n) = K_0(\varepsilon)$ large enough.

It remains to make sure that $K \leq \lambda^{\tilde{\delta}}$ to apply the narrow decoupling result: By choosing $\tilde{\delta}= \frac{3 \delta_2}{e(k,p)}$ and $\lambda \geq \lambda_0(\varepsilon)$ such that $\lambda^{\tilde{\delta}} \geq K_0 \lambda^{\frac{2 \delta_2}{e(k,p)}}$, the proof is complete because \eqref{eq:InitialRescaling} is ensured by
 \begin{equation*}
10 n \tilde{\delta} = \frac{30 n \delta_2}{e(k,p)} = 6  \frac{10 n \delta_2}{2 e(k,p)} \leq \frac{\varepsilon}{10}.
 \end{equation*}
\end{proof}

\begin{remark}
\label{rem:RefinedLpLpEstimate}
We can similarly prove an estimate
\begin{equation*}
\| T^\lambda f \|_{L^p(B(0,R))} \lesssim_{\varepsilon,\phi,a} R^\varepsilon \| f \|_{L^p},
\end{equation*}
for which we have to modify the induction quantity to work with $R^{\tilde{\delta}}$-flat phase functions.
\end{remark}

\section{$\varepsilon$-removal away from the endpoint}
\label{section:epsRemoval}
In the following we prove the estimate
\begin{equation}
\label{eq:GlobalEstimate}
\| T^\lambda f \|_{L^p(\R^n)} \lesssim_{\phi,a} \| f \|_{L^p(A^{n-1})}
\end{equation}
for $p > p_n$ with $p_n$ defined in \eqref{eq:PolynomialPartitioningRange}. The argument is essentially well-known in the literature \cite{Tao1998,Tao1999,GuthHickmanIliopoulou2019} and we shall be brief. The detailed argument from \cite{GuthHickmanIliopoulou2019} cannot be applied directly because it relies on non-degenerate curvature properties $H2)$ of the phase function. However, we shall see that the partial non-degeneracy
\begin{equation}
\label{eq:OneNondegenerateCurvature}
\exists \text{ non-vanishing eigenvalue of } \partial^2_{\omega \omega} \langle \partial_x \phi^\lambda(x;\omega), G^\lambda(x;\omega_0) \rangle \vert_{\omega = \omega_0}
\end{equation}
suffices for the argument. In the following we suppose that the phase $\phi$ satisfies the non-degeneracy $C1)$ and \eqref{eq:OneNondegenerateCurvature}. We shall prove that, if for $\bar{p} \geq 2$ and for all $\varepsilon > 0$ the estimate
\begin{equation}
\label{eq:LocalEstimate}
\| T^\lambda f \|_{L^p(B_R)} \lesssim_{\varepsilon,\phi,a} R^\varepsilon \| f \|_{L^p(A^{n-1})}
\end{equation}
holds for all $p \geq \bar{p}$, all $R$-balls $B_R$, and any amplitude, then we find the global estimate \eqref{eq:GlobalEstimate} to hold for all $p > \bar{p}$. The following notion plays an important role in the argument:
\begin{definition}[Tao \cite{Tao1999}]
Let $R \geq 1$. A collection $\{B(x_j,R)\}_{j=1}^N$ of $R$-balls in $\R^d$ is sparse if $\{x_1,\ldots,x_N\}$ are $(RN)^{\bar{C}}$-separated. Here $\bar{C} \geq 1$ is a fixed constant, chosen large enough to satisfy the requirements of the forthcoming argument.
\end{definition}
Like in previous instances of the argument, it suffices to analyze sparse families of balls.
\begin{lemma}[{\cite[Lemma~12.2]{GuthHickmanIliopoulou2019}}]
\label{lem:SparseReduction}
To prove \eqref{eq:GlobalEstimate} for all $p > \bar{p}$, it suffices to show that for all $\varepsilon>0$ the estimate
\begin{equation}
\label{eq:SparseEstimate}
\| T^\lambda f \|_{L^{\bar{p}}(S)} \lesssim_{\varepsilon,\phi,a} R^\varepsilon \| f \|_{L^{\bar{p}}(A^{n-1})}
\end{equation}
holds whenever $R \geq 1$ and $S \subseteq \R^n$ is a union of $R$-balls belonging to a sparse collection, for any choice of amplitude function.
\end{lemma}
The key ingredient in the proof of Lemma \ref{lem:SparseReduction} is the following covering lemma due to Tao \cite{Tao1998}:
\begin{lemma}[Covering~lemma,~\cite{Tao1998,Tao1999}]
\label{lem:CoveringLemma}
Suppose that $E \subseteq \R^n$ is a finite union of $1$-cubes and $N \geq 1$. Define the radii $R_j$ inductively by $R_0 =1$ and $R_j=(R_{j-1} |E|)^{\bar{C}}$ for $1 \leq j \leq N$. Then, for each $1 \leq j \leq N$, there exists a family of sparse collections $(\mathcal{B}_{j,\alpha})_{\alpha \in A_j}$ of balls of radius $R_j$ such that the index sets $A_k$ have cardinality $O(|E|^{1/N})$ and
\begin{equation*}
E \subseteq \bigcup_{j=1}^{N} \bigcup_{\alpha \in A_j} S_{j,\alpha},
\end{equation*}
where $S_{j,\alpha}$ is the union of all the balls belonging to the family $\mathcal{B}_{j,\alpha}$.
\end{lemma}
With Lemma \ref{lem:CoveringLemma} at hand, the proof of Lemma \ref{lem:SparseReduction} from \cite{GuthHickmanIliopoulou2019} applies. It remains to establish the estimates for $T^\lambda$ over sparse collections of $R$-balls.
\begin{lemma}
Let $\phi$ satisfy $C1)$ and \eqref{eq:OneNondegenerateCurvature}. If $p \geq \bar{p}$, then the estimate
\begin{equation*}
\| T^\lambda f \|_{L^p(S)} \lesssim_{\varepsilon,\phi,a} R^\varepsilon \| f \|_{L^p}
\end{equation*}
holds for all $ \varepsilon > 0$ whenever $S \subseteq \R^n$ is a union of $R$-balls belonging to a sparse collection.
\end{lemma}
\begin{proof}
Let $(B(x_j,R))_{j=1}^N$ be a sparse collection of balls. We can suppose that $R \ll \lambda$ and that all $B(x_j,R)$ intersect the $x$-support of $a^\lambda$. 
Fix $\eta \in C^\infty_c(\R^{n-1})$ satisfying $0 \leq \eta \leq 1$, $\text{supp}(\eta) \subseteq B^{n-1}$ and $\eta(z) = 1$ for all $z \in B(0,1/2)$. For $R_1 := CNR$, where $C \geq 1$ is a large constant, define $\eta_{R_1}(z) = \eta(z/R_1)$. Let $\psi \in C^\infty_c(\R^{n-1})$ satisfy $0 \leq \psi \leq 1$, $\text{supp}(\psi) \subseteq \Omega$ and $\psi(\omega) = 1$ for $\omega$ belonging to the $\omega$-support of $a^\lambda$. Fix $1 \leq j \leq N$ and write
\begin{equation*}
e^{i \phi^\lambda(x_j;\cdot)}\psi f = P_j f + (e^{i \phi^\lambda(x_j;\cdot)} \psi f - P_j f) =: P_j f + f_{j,\infty},
\end{equation*}
where $P_j f = \hat{\eta}_{R_1} * [e^{i \phi^\lambda(x_j;\cdot)} \psi f]$. 
Arguing like in \cite{GuthHickmanIliopoulou2019}, it suffices to show that
\begin{equation*}
\big( \sum_{j=1}^N \| P_j f \|^p_{L^p} \big)^{\frac{1}{p}} \lesssim \| f \|_{L^p}.
\end{equation*}
This follows via interpolation between $p=2$ and $p= \infty$. For $p= \infty$, this is a consequence of Young's inequality. The estimate for $p=2$ is by duality equivalent to
\begin{equation*}
\big\| \sum_{j=1}^N e^{-2 \pi i \phi^\lambda(x_j;\cdot)} \psi \cdot [\hat{\eta}_{R_1} * g_j] \big\|_{L^2(\R^{d-1})} \lesssim \big( \sum_{j=1}^N \| g_j \|^2_{L^2} \big)^{\frac{1}{2}}.
\end{equation*}
By squaring the left-hand side, we find
\begin{equation*}
\begin{split}
\sum_{j_1,j_2=1}^N \int_{\R^{d-1}} \overline{G_{j_1,j_2}(\omega)} \hat{\eta}_{R_1} * g_{j_1}(\omega) \overline{\hat{\eta}_{R_1} * g_{j_2}(\omega)} d\omega \\
 \text{ with } G_{j_1,j_2}(\omega) = e^{i (\phi^\lambda(x_{j_1};\omega) - \phi^\lambda(x_{j_2};\omega))} \psi^2(\omega).
 \end{split}
\end{equation*}
By the Van der Corput-lemma \cite[Proposition~5,~p.~342]{Stein1993}, we still find due to \eqref{eq:OneNondegenerateCurvature}
\begin{equation*}
| \check{G}_{j_1,j_2}(z)| \lesssim |x_{j_2}-x_{j_1}|^{-\frac{1}{2}}.
\end{equation*}
This suffices to estimate the absolute value and straight-forward summation of the off-diagonal terms. The contribution of the diagonal terms is easier to estimate. For details we refer to \cite{GuthHickmanIliopoulou2019}.

\end{proof}

\section{Improved local smoothing for Fourier integral operators}
\label{section:LocalSmoothing}

In this section we improve $L^p$-smoothing estimates for solutions to wave equations on compact Riemannian manifolds $(M,g)$ with $\dim (M) \geq 3$: We consider
\begin{equation}
\label{eq:WaveEquationCompactManifold}
\left\{ \begin{array}{cl}
\partial_t^2 u - \Delta_g u &= 0, \quad (x,t) \in \R \times M, \\
u(\cdot,0) &= f_0, \qquad \dot{u}(\cdot,0) = f_1
\end{array} \right.
\end{equation}
with the solution $u$ to \eqref{eq:WaveEquationCompactManifold} given by
\begin{equation*}
u(t) = \cos(t \sqrt{- \Delta_g}) f_0 + \frac{\sin(t \sqrt{- \Delta_g})}{\sqrt{- \Delta_g}} f_1.
\end{equation*}

Parametrices for the half-wave equation are provided by Fourier integral operators (FIOs); see below.

\medskip

 Gao \emph{et al.} \cite{GaoLiuMiaoXi2023} improved the previous results on Euclidean local smoothing for $d \geq 3$ and $p \leq \frac{2(d+1)}{d-1}$ due to Bourgain--Demeter by a broad--narrow iteration. Presently, we extend their arguments to the variable coefficient case.
  
Let $d \geq 3$ and
\begin{equation}
\label{eq:RangeSmoothingFIO}
p_d =
\begin{cases}
2 \cdot \frac{3d+5}{3d+1} \text{ for } d \text{ odd}, \\
2 \cdot \frac{3d+6}{3d+2} \text{ for } d \text{ even}.
\end{cases}
\end{equation}
We show the following:
\begin{theorem}[Improved local smoothing on compact manifolds]
\label{thm:ImprovedLocalSmoothing}
Let $(M,g)$ be a compact Riemannian manifold with $\dim(M) \geq 3$. Let $\bar{s}_p=(d-1)|\frac{1}{2}-\frac{1}{p}|$, $p_d \leq p < \infty$ with $p_d$ as in \eqref{eq:RangeSmoothingFIO} and $\sigma < \frac{2}{p} - \frac{1}{2}$. Let $u$ be a solution to \eqref{eq:WaveEquationCompactManifold}.
Then, the following estimate holds:
\begin{equation}
\label{eq:LocalSmoothingCompactManifold}
\| u \|_{L_t^p([1,2], L_x^p(M))} \lesssim_{M,g,p,\sigma} \| f_0 \|_{L^p_{\bar{s}_p-\sigma}(M)} + \| f_1 \|_{L^p_{\bar{s}_p-\sigma+1}(M)}.
\end{equation}
\end{theorem}
Like in \cite{GaoLiuMiaoXi2023}, we can interpolate with the trivial $L^2$-estimate and the sharp local smoothing estimates for $p \geq \frac{2(d+1)}{d-1}$ following from decoupling due to Beltran--Hickman--Sogge \cite{BeltranHickmanSogge2020} to obtain an extended range of estimates. This gives the same range of estimates like in \cite[Corollary~1.3]{GaoLiuMiaoXi2023}.

\medskip

It is well-known (cf. \cite[Chapter~4]{Sogge2017}, \cite[p.~224]{MinicozziSogge1997}) that local parametrices for \eqref{eq:WaveEquationCompactManifold} take the form of FIOs
\begin{equation}
\label{eq:FIO}
(\mathcal{F} f)(x,t) = \int_{\R^d} e^{i \phi(x,t;\xi)} a(x,t;\xi) \hat{f}(\xi) d\xi
\end{equation}
with phase functions $\phi \in C^\infty(\R^{d+1} \times \R^d \backslash \{ 0 \})$, which are $1$-homogeneous in $\xi$ and satisfy $C1)$ and $C2^+)$. $a \in S^0(\R^{d+1})$ is a symbol of order zero. Recall that this means $a \in C^\infty(\R^{d+1} \times \R^{d+1})$ and
\begin{equation*}
\forall \alpha, \beta \in \N_0^{d+1}: \; |\partial_{(x,t)}^\alpha \partial_\xi^\beta a(x,t,\xi)| \lesssim_{\alpha,\beta} \langle \xi \rangle^{-|\beta|}.
\end{equation*}
Moreover, we assume that $a$ is compactly supported in $(x,t) \in \R^{d+1}$.
 
\medskip 
 
It turns out that for the proof of Theorem \ref{thm:ImprovedLocalSmoothing}, it suffices to prove bounds for rescaled operators
\begin{equation}
\label{eq:RescaledFIO}
(\mathcal{F}^\lambda f)(x,t) = \int_{\R^d} e^{i \phi^\lambda(x,t;\xi)} a^\lambda(x,t;\xi) \hat{f}(\xi) d\xi
\end{equation}
with $a^\lambda$ and $\phi^\lambda$ defined like in previous sections. Theorem \ref{thm:ImprovedLocalSmoothing} is a consequence of the following (cf. \cite[Section~3]{BeltranHickmanSogge2020}):

\begin{proposition}
\label{prop:SmoothingFIO}
Let $\mathcal{F}$ be an FIO as in \eqref{eq:FIO} and $p_d$ as in \eqref{eq:RangeSmoothingFIO}. Then, the following local smoothing estimate holds for $p_d \leq p < \infty$:
\begin{equation}
\label{eq:SmoothingFIOs}
\| \mathcal{F}^\lambda f \|_{L^p_{t,x}(\R^{d+1})} \lesssim_{\varepsilon,\phi,a} \lambda^{d \big( \frac{1}{2} - \frac{1}{p} \big) + \varepsilon} \| f \|_{L^p(\R^d)}.
\end{equation}
\end{proposition}

Proposition \ref{prop:SmoothingFIO} improves on the previously best estimates due to Beltran--Hickman--Sogge \cite{BeltranHickmanSogge2020}, which read
\begin{equation*}
\| \mathcal{F}^\lambda f \|_{L^p_{t,x}(\R^{d+1})} \lesssim_{\varepsilon, \phi, a} \lambda^{\frac{d-1}{2} \big( \frac{1}{2} - \frac{1}{p} \big)+ \frac{1}{p} + \varepsilon} \| f \|_{L^p(\R^d)}
\end{equation*}
for $2 \leq p \leq \frac{2(d+1)}{d-1}$. Beltran--Hickman--Sogge \cite{BeltranHickmanSogge2020} extended the decoupling inequalities in the constant coefficient case \cite{BourgainDemeter2015} to variable coefficients. This argument also yields local smoothing estimates for FIOs, which do not satisfy the convexity condition $C2^+)$. Indeed, the FIOs, for which decoupling yields the sharp smoothing estimates (cf. \cite[Section~4]{BeltranHickmanSogge2020}), are the ones with $d$ odd, and
\begin{equation*}
\partial^2_{\xi \xi} \langle \partial_{(x,t)} \phi(x,t;\xi), G_0(x,t;\xi_0) \rangle \vert_{\xi = \xi_0}
\end{equation*}
having $\frac{d-1}{2}$ positive and $\frac{d-1}{2}$ negative eigenvalues. Recall that the generalized Gauss map is defined by
\begin{equation*}
G_0(x,t;\xi) = \partial^2_{(x,t) \, \xi_1} \phi(x,t;\xi) \wedge \ldots \wedge \partial^2_{(x,t) \, \xi_d} \phi(x,t;\xi) \in \bigwedge_{i=1}^d \R^{d+1} \simeq \R^{d+1}.
\end{equation*}

For the proof of Proposition \ref{prop:SmoothingFIO}, we run almost the same iteration as in the proof of Theorem \ref{thm:LpLpEstimatesVariableCoefficients}. The following lemma based on finite speed of propagation allows us to convert $L^2$-estimates for $T^\lambda$ into $L^p$-estimates for $\mathcal{F}^\lambda$ by Plancherel's theorem and H\"older's inequality (see \cite[Section~2.2]{GaoLiuMiaoXi2023}):
\begin{lemma}
\label{lem:FiniteSpeedPropagation}
Let $(\phi,a)$ be reduced data and $\psi \in \mathcal{S}(\R^d)$ such that $\text{supp} (\hat{\psi}) \subseteq B(0,1)$, $\sum_{\ell \in \Z^d} \psi(x-\ell) \equiv 1$ for any $x \in \R^d$. Assume $\text{supp} (\hat{f}) \subseteq A^d$. Then, for any $\delta > 0$, the following estimate holds true:
\begin{equation}
\label{eq:FiniteSpeedPropagation}
\begin{split}
|\mathcal{F}^\lambda f(x,t)| &\lesssim_{\delta} |\mathcal{F}^\lambda (\psi_{R^{1+\delta}(x_0)} f)(x,t)| \\
&\quad + \text{RapDec}(R) \sum_{|\ell| > R^{\delta}} (1+|\ell|)^{-M} \| f |\psi_{\ell}(\cdot-x_0)|^{\frac{1}{2}} \|_{L^p(w_{B_R^d(x_0)})}
\end{split}
\end{equation}
for $(x,t) \in B(x_0,R) \times [-R,R]$, $1<p<\infty$, where
\begin{equation*}
\psi_{R^{1+\delta}(x_0)}(x) = \sum_{|\ell| < R^\delta} \psi(R^{-1}(x-x_0) - \ell).
\end{equation*}
\end{lemma}
\begin{proof}
The claim follows from a kernel estimate. We have
\begin{equation*}
\begin{split}
\mathcal{F}^\lambda(x,t) &= \int e^{i \phi^\lambda(x,t;\xi)} a^\lambda(x,t;\xi) \hat{f}(\xi) d\xi \\
&= \frac{1}{(2 \pi)^d} \int e^{i(\phi^\lambda(x,t;\xi) - \langle y,\xi \rangle)} a^\lambda(x,t;\xi) f(y) dy d\xi.
\end{split}
\end{equation*}
We set $K^\lambda(x,t;y) = \int e^{i(\phi^\lambda(x,t;\xi) - \langle y, \xi \rangle)} a^\lambda(x,t;\xi) d\xi$. Let $\Phi^\lambda(x,y,\xi,t) = \phi^\lambda(x,t;\xi) - \langle y, \xi \rangle$. We have
\begin{equation*}
\begin{split}
\nabla_\xi \Phi^\lambda(x,y,\xi,t) &= \nabla_\xi \phi^\lambda(x,t;\xi) - y, \\
\nabla_\xi \phi^\lambda(x,0;\xi) = x, &\quad \nabla_\xi \phi^\lambda(x,t;\xi) = \frac{1}{\lambda} \nabla_\xi \int_0^{\lambda t} \partial_t \phi^\lambda(x,s;\xi) ds + x.
\end{split}
\end{equation*}
By $|\nabla_\xi \partial_t \phi(x,t;\xi)| \lesssim 1$ for a reduced phase function, we find for $|t| \leq R$ and $|x-y| \geq R^{1+\delta}$ rapid decay by non-stationary phase. We have the estimate
\begin{equation*}
|K^\lambda (x,t;y)| \leq C_N (1+R|x-y|)^{-N}.
\end{equation*}
For reduced data $(\phi,a)$, $C_N$ can be chosen uniformly.
This follows from estimates, which hold because $a$ is reduced and due to Lemma \ref{lem:ReducedPhaseTimeDerivatives}:
\begin{equation*}
\| \partial_\xi^\alpha a(x,\xi) \|_{L^\infty(X \times \Xi)} \leq C_{\text{amp}} \text{ and } \| \partial_t \partial_{\xi}^{\alpha} \phi \|_{L^\infty(X \times \Xi)} \lesssim_{N} 1 \text{ for } \alpha \in \N_0^d, \, |\alpha| \leq N.
\end{equation*}
 This yields \eqref{eq:FiniteSpeedPropagation}.
\end{proof}

By the same arguments as in Section \ref{section:LinearEstimates}, we can show the following narrow decoupling (cf. Proposition~\ref{prop:DecouplingNarrowTerm}):
\begin{proposition}
\label{prop:NarrowDecoupling}
Let $B_{K^2} \subseteq B(0,\lambda^{1-\delta'}) \subseteq \R^{d+1} $ be a $K^2$-ball centred at $\bar{x} \in \R^{d+1}$. Let $(\phi,a)$ be reduced and $\phi$ be a $K$-flat phase function. Let $k \geq 3$ and $V$ be a $(k-1)$-dimensional vector space. Suppose that $\text{supp}(\hat{f}) \subseteq \bigcup_\nu S_{\nu}$, where $(S_\nu)_{\nu}$ is a family of $K^{-1}$-sectors $S_{\nu}$ centred at $\tau_\nu$ such that $\angle(G^\lambda(\bar{x};\tau_\nu),V) \leq K^{-2}$. Let $\hat{f}_\nu$ be the localization of $\hat{f}$ to $S_\nu$. Then, the following estimate holds:
\begin{equation}
\label{eq:NarrowDecoupling}
\| T^\lambda \hat{f} \|_{L^p(B_{K^2})} \lesssim_\delta K^\delta K^{\max((k-3) \big( \frac{1}{2}-\frac{1}{p} \big),0)} \big( \sum_\nu \| T_*^\lambda \hat{f}_{\nu} \|^p_{L^p(w_{B_{K^2}})} \big)^{1/p} + \lambda^{- \frac{\min( \delta, \delta') N}{2}} \| f \|_{L^2}
\end{equation}
for $ 2 \leq p \leq \frac{2(k-1)}{k-3}$. Above $w_{B_{K^2}}$ denotes a weight, which decays polynomially off $B_{K^2}$. The operator $T_\lambda^*$ is built from a reduced amplitude function $a^*$ and the same phase $\phi$. $a^*$ can be chosen from a family of size $O_N(1)$. The family is uniform over balls $B_{K^2} \subseteq B(0,\lambda^{1-\delta'})$.
\end{proposition}

\begin{remark}
Note that \eqref{eq:NarrowDecoupling} can be written as
\begin{equation*}
\| \mathcal{F}^\lambda f \|_{L^p(B_{K^2})} \lesssim_\delta K^\delta K^{\max((k-3) \big( \frac{1}{2}-\frac{1}{p} \big),0)} \big( \sum_\nu \| \mathcal{F}^\lambda_* f_\nu \|^p_{L^p(w_{B_{K^2}})} \big)^{1/p} + \lambda^{- \frac{\min( \delta, \delta') N}{2}} \| f \|_{L^2}
\end{equation*}
simply by definition of $T^\lambda$ and $\mathcal{F}^\lambda$. We used the formulation in Proposition \ref{prop:NarrowDecoupling} to highlight the parallel to Proposition \ref{prop:DecouplingNarrowTerm}.
\end{remark}

The discrepancy in the amplitude functions will be remedied by carrying out the induction on scales for an appropriate class of phase and amplitude functions.

\begin{definition}
 Let $Q_{p,\delta}(R)$ be the infimum over all constants such that
\begin{equation*}
\| \mathcal{F}^\lambda f \|_{L^p(B(0,R))} \leq Q_{p,\delta}(R) R^{d \big( \frac{1}{2} - \frac{1}{p} \big)} \| f \|_{L^p}
\end{equation*}
for $1 \leq R \leq \lambda$ and all FIOs $\mathcal{F}$ with reduced $(\phi,a)$ and $\lambda^\delta$-flat phase $\phi$. 
\end{definition}

\medskip

As further ingredient we use the following parabolic rescaling for FIOs.

\begin{lemma}[Parabolic rescaling for FIOs]
\label{lem:RescalingFIOs}
Let $(\phi,a)$ be reduced with $\phi$ a $\lambda^\delta$-flat phase function and $\hat{g}$ supported in a $\rho^{-1}$-sector in $A^d$. Then, for any $1 \leq \rho \leq R \leq \lambda$, the following estimate holds:
\begin{equation}
\label{eq:ParabolicRescalingFIO}
\| \mathcal{F}^\lambda g \|_{L^p(B(0,R))} \lesssim_{\delta'} R^{\delta'} Q_{p,\delta}(R/\rho^2) R^{d \big( \frac{1}{2} - \frac{1}{p} \big)} \rho^{\frac{2(d+1)}{p} - d} \| g \|_{L^p}.
\end{equation}
\end{lemma}
\begin{proof}
The proof has much in common with the proof of Lemma \ref{lem:ParabolicRescaling}. However, after rescaling, we use almost orthogonality in space-time, which comes from finite speed of propagation (cf. Lemma \ref{lem:FiniteSpeedPropagation}). Let $\omega \in B_{d-1}(0,1)$ with $(\omega,1)$ the centre of the $\rho^{-1}$-slab encasing the support of $\hat{g}$:
\begin{equation*}
\text{supp}(\hat{g}) \subseteq \{ (\xi',\xi_d) \in \R^d \, : \, 1/2 \leq \xi_d \leq 2 \text{ and } \big| \frac{\xi'}{\xi_d} - \omega \big| \leq \rho^{-1} \}.
\end{equation*}
We perform the change of variables:
\begin{equation*}
(\xi',\xi_d) = (\eta_d \omega + \rho^{-1} \eta', \eta_d),
\end{equation*}
after which it follows that
\begin{equation*}
(\mathcal{F}^\lambda g)(x,t) = \int_{\R^d} e^{i \phi^\lambda(x,t;\eta_d \omega + \rho^{-1} \eta', \eta_d)} a^\lambda(x,t;\eta_d \omega + \rho^{-1} \eta', \eta_d) \hat{\tilde{g}}(\eta) d\eta,
\end{equation*}
where $\hat{\tilde{g}}(\eta) = \rho^{-(d-1)} \hat{g}(\eta_d \omega + \rho^{-1} \eta', \eta_d)$ and $\text{supp}(\hat{\tilde{g}}) \subseteq \Xi$. By Taylor expansion and homogeneity of the phase, we find
\begin{equation*}
\begin{split}
\phi(x,t;\eta_d \omega + \rho^{-1} \eta', \eta_d) &= \phi(x,t;\omega,1) \eta_d + \rho^{-1} \langle \partial_{\xi'} \phi(x,t;\omega,1), \eta' \rangle \\
&\quad + \rho^{-2} \int_0^1 (1-r) \langle \partial^2_{\xi' \xi'} \phi(x,t;\eta_d \omega + r \rho^{-1} \eta', \eta_d) \eta', \eta' \rangle dr.
\end{split}
\end{equation*}
Let $\Upsilon_\omega(x,t) = (\Upsilon(x,t;\omega,1),x_d)$ and $\Upsilon^\lambda_\omega(x,t) = \lambda \Upsilon_\omega(x/\lambda,t/\lambda)$ and consider anisotropic dilations
\begin{equation*}
D_\rho (x',x_d,t) = (\rho x', x_d, \rho^2 t) \text{ and } D'_{\rho^{-1}} (x',x_d) = (\rho^{-1} x', \rho^{-2} x_d)
\end{equation*}
on $\R^{d+1}$ and $\R^d$, respectively. By definition of $\Upsilon$, we find
\begin{equation*}
\mathcal{F}^\lambda g \circ \Upsilon^\lambda_\omega \circ D_\rho = \tilde{\mathcal{F}}^{\lambda /\rho^2} \tilde{g},
\end{equation*}
where
\begin{equation*}
\tilde{\mathcal{F}}^{\lambda / \rho^2} \tilde{g}(y,\tau) = \int_{\R^d} e^{i \tilde{\phi}^{\lambda/\rho^2}(y,\tau;\eta)} \tilde{a}^{\lambda}(y,\tau;\eta) \hat{\tilde{g}}(\eta) d\eta
\end{equation*}
for the phase $\tilde{\phi}(y,\tau;\eta)$ given by
\begin{equation*}
\langle y, \eta \rangle + \int_0^1 (1-r) \langle \partial^2_{\xi' \xi'} \phi(\Upsilon_\omega(D'_{\rho^{-1}} y,\tau); \eta_d \omega + r \rho^{-1} \eta', \eta_d) \eta', \eta' \rangle dr
\end{equation*}
and the amplitude
\begin{equation*}
\tilde{a}(y,\tau;\eta) = a(\Upsilon_\omega (D'_{\rho^{-1}} y;\tau);\eta_d \omega + \rho^{-1} \eta', \eta_d).
\end{equation*}
By change of space-time variables, we find
\begin{equation}
\label{eq:RescalingFIOI}
\| \mathcal{F}^\lambda g \|_{L^p(B_R)} \lesssim \rho^{\frac{d+1}{p}} \| \tilde{\mathcal{F}}^{\lambda/\rho^2} \tilde{g} \|_{L^p((\Upsilon^\lambda_\omega \circ D_\rho)^{-1}(B_R))}.
\end{equation}
Note that $(\Upsilon^\lambda_\omega \circ D_\rho)^{-1}(B_R)= D_R$ is roughly a set of size $R/\rho \times \ldots \times R/\rho \times R \times R/\rho^2$. We want to apply an estimate at a smaller scale $R/\rho^2$, to which end we use finite speed of propagation: Since the time-scale is $R/\rho^2$, we can decompose $\sim R/\rho \times \ldots R/\rho \times R \times R/\rho^2$ into balls of size $R/\rho^2$ and correspondingly, the support of $\tilde{g}$ into balls of size $R/\rho^2$. We have almost orthogonality between the localized pieces by Lemma \ref{lem:FiniteSpeedPropagation}, which yields
\begin{equation}
\label{eq:RescalingFIOII}
\| \tilde{\mathcal{F}}^{\lambda/\rho^2} \tilde{g} \|_{L^p(D_R)} \lesssim_{\delta'} R^{\delta'} Q_{p,\delta}(R/\rho^2) (R/\rho^2)^{d \big( \frac{1}{2} - \frac{1}{p} \big) } \| \tilde{g} \|_{L^p}.
\end{equation}
Since $\tilde{g}(x) = g(\rho x', x_d - \omega x')$, we find
\begin{equation}
\label{eq:RescalingFIOIII}
\| \tilde{g} \|_{L^p} = \rho^{- \frac{d-1}{p}} \| g \|_{L^p}.
\end{equation}
Taking \eqref{eq:RescalingFIOI}, \eqref{eq:RescalingFIOII}, and \eqref{eq:RescalingFIOIII} together, 
\eqref{eq:ParabolicRescalingFIO} holds.
\end{proof}

We are ready for the proof of the following proposition:
\begin{proposition}
\label{prop:NarrowIterationSmoothingFIO}
Let $d \geq 3$, $2 \leq k \leq d$, and $\lambda \geq 1$. Let
\begin{equation*}
\bar{p}(k,d) \leq p \leq 
\begin{cases}
\infty, \quad 2 \leq k \leq 3, \\
2 \cdot \frac{k-1}{k-3}, \quad k \geq 4,
\end{cases}
 \text{ with }
\bar{p}(k,d) = 
\begin{cases}
\frac{2(d+1)}{d}, \quad &k=2, \\
2 \cdot \frac{2d-k+5}{2d-k+3}, \quad &k \geq 3.
\end{cases}
\end{equation*}
Suppose that for all $\varepsilon >0$ FIOs with reduced $(\phi,a)$ obey the following $k$-broad estimate for all $1 \leq K \leq R \leq \lambda$ and choice of $A= A_\varepsilon$
\begin{equation*}
\| \mathcal{F}^\lambda f \|_{BL^p_{k,A}(B_R)} \leq \tilde{C}_\varepsilon K^{C_\varepsilon} R^\varepsilon R^{d \big( \frac{1}{2} - \frac{1}{p} \big) } \| f \|_{L^p}.
\end{equation*}
Let $(\tilde{\phi},\tilde{a})$ with $\tilde{\phi}$ a $1$-homogeneous phase function that satisfies $C1)$ and $C2^+)$ and $\tilde{a} \in S^0(\R^{2d+1})$, which is compactly supported in $(x,t)$. There is $D_{\varepsilon,\tilde{\phi},\tilde{a}}$ such that for the FIO $\tilde{\mathcal{F}}$ built from $(\tilde{\phi},\tilde{a})$ the following estimate holds:
\begin{equation}
\label{eq:SmoothingFIONarrow}
\| \tilde{\mathcal{F}}^\lambda f \|_{L^p(\R^{d+1})} \leq D_{\varepsilon,\tilde{\phi},\tilde{a}} \lambda^{d \big( \frac{1}{2} - \frac{1}{p} \big) + \varepsilon} \| f \|_{L^p(\R^d)}.
\end{equation}
\end{proposition}
The $k$-broad estimate is valid by Theorem \ref{thm:kBroadEstimate} (see also Remark \ref{rem:RefinedLpLpEstimate}) and an application of Lemma \ref{lem:FiniteSpeedPropagation}.

Proposition \ref{prop:SmoothingFIO} follows from Proposition \ref{prop:NarrowIterationSmoothingFIO} by choosing $k= \frac{d+5}{2}$ for $d$ odd and $k= \frac{d+4}{2}$ for $d$ even. Hence, the proof of Theorem \ref{thm:ImprovedLocalSmoothing} will be complete once Proposition \ref{prop:NarrowIterationSmoothingFIO} is proved.

\begin{proof}[Proof~of~Proposition~\ref{prop:NarrowIterationSmoothingFIO}]
The proof has much in common with the proof of Proposition \ref{prop:LinearFromBroadEstimates}, and we shall be brief. By one parabolic rescaling depending on the phase as in the beginning of the proof of Proposition \ref{prop:LinearFromBroadEstimates}, we can suppose that $(\phi,a)$ is $\lambda^{\tilde{\delta}}$-flat, and $R \leq \lambda^{1-\frac{\varepsilon}{10d}}$.

Recall that $Q_{p,\tilde{\delta}}(R)$ denotes the smallest constant such that for reduced $(\phi,a)$ with $\lambda^{\tilde{\delta}}$-flat phase functions, we have
\begin{equation*}
\| \mathcal{F}^\lambda f \|_{L^p(B(0,R))} \leq R^{d \big( \frac{1}{2} - \frac{1}{p} \big)} Q_{p,\tilde{\delta}}(R) \| f \|_{L^p}.
\end{equation*}
It suffices to prove that for any $\varepsilon > 0$ there is $C_\varepsilon > 0$ such that $Q_{p,\tilde{\delta}}(R) \leq C_\varepsilon R^\varepsilon$ for any $R \leq \lambda^{1-\frac{\varepsilon}{10 d}}$. We shall choose 
\begin{equation}
\label{eq:ChoiceKFIO}
K=K_0 R^{\eta}
\end{equation}
 with $K_0$ and $\eta$ to be determined later. This is similar to the argument in the proof of Proposition \ref{prop:LinearFromBroadEstimates}.

For a given ball $B_{K^2} \subseteq B(0,R)$, let $V_1,\ldots,V_A$ be $(k-1)$-dimensional linear subspaces which achieve the minimum in the definition of the $k$-broad norm. Then,
\begin{equation*}
\begin{split}
\int_{B_{K^2}^{d+1}} | \mathcal{F}^\lambda f(x,t)|^p dx dt &\lesssim K^{O(1)} \max_{\tau \notin V_{\ell}} \int_{B_{K^2}^{d+1}} | \mathcal{F}^\lambda f^\tau(x,t)|^p dx dt \\
&\quad + \sum_{a=1}^A \int_{B_{K^2}^{d+1}} \big| \sum_{\tau \in V_{a}} \mathcal{F}^\lambda f^\tau(x,t) \big|^p dx dt.
\end{split}
\end{equation*}
Summing over a finitely overlapping family $\big( B_{K^2} \big) = \mathcal{B}_{K^2}$ covering $B(0,R)$ yields
\begin{equation*}
\begin{split}
\int_{B(0,R)} | \mathcal{F}^\lambda f(x,t)|^p dx dt &\lesssim K^{O(1)} \sum_{B_{K^2} \in \mathcal{B}_{K^2}} \min_{V_1,\ldots,V_A} \max_{\tau \notin V_{\ell}} \int_{B_{K^2}} | \mathcal{F}^\lambda f^\tau(x,t)|^p dx dt \\
&\quad + \sum_{B_{K^2} \in \mathcal{B}_{K^2}} \sum_{a=1}^A \int_{B_{K^2}} \big| \sum_{\tau \in V_{\ell}} \mathcal{F}^\lambda f^\tau(x,t)|^p dx dt.
\end{split}
\end{equation*}
By the broad norm estimate, we find
\begin{equation*}
\sum_{B_{K^2} \in \mathcal{B}_{K^2}} \min_{V_1,\ldots,V_A} \max_{\tau \notin V_{\ell}} \int_{B_{K^2}} | \mathcal{F}^\lambda f^\tau(x,t)|^p dx dt \leq \tilde{C}_\varepsilon K^{C_{\varepsilon}} R^{\frac{\varepsilon p}{2}} R^{dp \big( \frac{1}{2} - \frac{1}{p} \big)} \| f \|_{L^p}^p. 
\end{equation*}
The narrow contribution is estimated by Proposition \ref{prop:NarrowDecoupling}:\footnote{Here we omit the error term, which can be handled like in the proof of Proposition \ref{prop:LinearFromBroadEstimates}.}
\begin{equation}
\label{eq:NarrowEstimateI}
\begin{split}
\sum_{a=1}^A \int_{B_{K^2} \in \mathcal{B}_{K^2} } \big| \sum_{\tau \in V_{a}} \mathcal{F}^\lambda f^\tau(x,t) \big|^p dx dt &\leq C_{\delta_1} K^{\delta_1} K^{\max( (k-3)\big( \frac{1}{2} - \frac{1}{p} \big) p, 0)} \\
&\qquad \times \sum_{\tau} \int_{\R^{d+1}} w_{B_{K^2}} |\mathcal{F}^\lambda f^\tau(x,t)|^p dx dt.
\end{split}
\end{equation}
Summing over $B_{K^2} \in \mathcal{B}_{K^2}$ in \eqref{eq:NarrowEstimateI}, we find
\begin{equation*}
\begin{split}
\sum_{B_{K^2} \in \mathcal{B}_{K^2}} \sum_{a=1}^A \int_{B_{K^2}^{d+1}} \big| \sum_{\tau \in V_{\ell}} \mathcal{F}^\lambda f^\tau(x,t) \big|^p dx dt &\leq C_{\delta_1} K^{\delta_1} K^{\max( (k-3) \big( \frac{1}{2} - \frac{1}{p} \big) p, 0)} \\
&\quad \times \sum_\tau \int_{\R^{d+1}} w_{B(0,2R)} |\mathcal{F}^\lambda f^\tau(x,t) |^p dx dt.
\end{split}
\end{equation*}
By Lemma \ref{lem:RescalingFIOs}, we find
\begin{equation}
\label{eq:NarrowEstimateII}
\begin{split}
&\quad \int_{B(0,4R)} |\mathcal{F}^\lambda f^\tau(x,t) |^p dx dt \\
&\lesssim_{\delta_2} K^{-2d \big(\frac{1}{2} - \frac{1}{p} \big) p + 2 } Q_{p,\tilde{\delta}}^p (R/K^2) R^{dp \big( \frac{1}{2} - \frac{1}{p} \big) + \delta_2} \| f^\tau \|^p_p + \text{RapDec}(R) \| f \|_p^p.
\end{split}
\end{equation}
We have the following estimate for $2 \leq p \leq \infty$:
\begin{equation}
\label{eq:SectorSummation}
\big( \sum_\tau \| f^\tau \|^p_p \big)^{\frac{1}{p}} \lesssim \| f \|_p.
\end{equation}
For $p=2$ this holds by Plancherel's theorem, for $p=\infty$ by a kernel estimate, and the remaining cases are covered by interpolation.

Hence, summing \eqref{eq:NarrowEstimateII} over $\tau$ yields by \eqref{eq:SectorSummation}
\begin{equation*}
\begin{split}
\int_{B_R^{d+1}} |\mathcal{F}^\lambda f(x,t)|^p dx dt &\leq \tilde{C}_\varepsilon K^{O(1) + C_\varepsilon} R^{dp \big( \frac{1}{2}- \frac{1}{p} \big) + \frac{\varepsilon p}{2}} \| f \|_{L^p}^p \\
 &\quad + C_{\delta_1,\delta_2} K^{\delta_1} R^{dp \big( \frac{1}{2} - \frac{1}{p} \big) + \delta_2} K^{-e(p,k,d)} Q_{p,\tilde{\delta}}^p(R/K^2) \| f \|_{L^p}^p
 \end{split}
\end{equation*}
with
\begin{equation*}
e(p,k,d) = \min \{ 2d \big( \frac{1}{2} - \frac{1}{p} \big) p - 2, 2d \big( \frac{1}{2} - \frac{1}{p} \big) p - 2 - (k-3) \big( \frac{1}{2} - \frac{1}{p} \big) p \}.
\end{equation*}
We find $e(p,k,d) \geq 0$, if
\begin{equation*}
p \geq 
\begin{cases}
 \frac{2(d+1)}{d}, \quad &k=2, \\
 2 \cdot \frac{2d-k+5}{2d-k+3}, \quad &k \geq 3.
\end{cases}
\end{equation*}
By the definition of $Q_{p,\tilde{\delta}}$, we have
\begin{equation*}
Q_{p,\tilde{\delta}}^p(R) \leq K^{O(1)} \tilde{C}_{\varepsilon} K^{C_\varepsilon} R^{\frac{\varepsilon p}{2}} + C_{\delta_1,\delta_2} R^{\delta_2 } Q_{p,\tilde{\delta}}^p(R) K^{-e(p,k,d) + \delta_1}.
\end{equation*}
Similar to the end of the proof of Proposition \ref{prop:LinearFromBroadEstimates}, we can choose $\delta_1(\varepsilon)$, $\delta_2(\varepsilon)$, and $K_0$ and $\eta$ (from \eqref{eq:ChoiceKFIO}) to finish the proof.
\end{proof}

\section*{Acknowledgements}

Funded by the Deutsche Forschungsgemeinschaft (DFG, German Research Foundation) -- Project-ID 258734477 -- SFB 1173. I am indebted to the anonymous referee for a very careful reading of the manuscript and insightful comments, which led to many improvements.

\bibliographystyle{plain}

\begin{thebibliography}{10}

\bibitem{Bejenaru2017}
Ioan Bejenaru.
\newblock The optimal trilinear restriction estimate for a class of
  hypersurfaces with curvature.
\newblock {\em Adv. Math.}, 307:1151--1183, 2017.

\bibitem{BeltranHickmanSogge2020}
David Beltran, Jonathan Hickman, and Christopher~D. Sogge.
\newblock Variable coefficient {W}olff-type inequalities and sharp local
  smoothing estimates for wave equations on manifolds.
\newblock {\em Anal. PDE}, 13(2):403--433, 2020.

\bibitem{BeltranHickmanSogge2018}
David Beltran, Jonathan Hickman, and Christopher~D. Sogge.
\newblock Sharp local smoothing estimates for fourier integral operators.
\newblock In {\em Geometric Aspects of Harmonic Analysis}, pages 29--105, Cham,
  2021. Springer International Publishing.

\bibitem{BennettCarberyTao2006}
Jonathan Bennett, Anthony Carbery, and Terence Tao.
\newblock On the multilinear restriction and {K}akeya conjectures.
\newblock {\em Acta Math.}, 196(2):261--302, 2006.

\bibitem{Bourgain1991}
J.~Bourgain.
\newblock {$L^p$}-estimates for oscillatory integrals in several variables.
\newblock {\em Geom. Funct. Anal.}, 1(4):321--374, 1991.

\bibitem{BourgainDemeter2015}
Jean Bourgain and Ciprian Demeter.
\newblock The proof of the {$l^2$} decoupling conjecture.
\newblock {\em Ann. of Math. (2)}, 182(1):351--389, 2015.

\bibitem{BourgainGuth2011}
Jean Bourgain and Larry Guth.
\newblock Bounds on oscillatory integral operators based on multilinear
  estimates.
\newblock {\em Geom. Funct. Anal.}, 21(6):1239--1295, 2011.

\bibitem{CarlesonSjoelin1972}
Lennart Carleson and Per Sj\"{o}lin.
\newblock Oscillatory integrals and a multiplier problem for the disc.
\newblock {\em Studia Math.}, 44:287--299. (errata insert), 1972.

\bibitem{GaoLiuMiaoXi2020}
Chuanwei {Gao}, Bochen {Liu}, Changxing {Miao}, and Yakun {Xi}.
\newblock {Square function estimates and Local smoothing for Fourier Integral
  Operators}.
\newblock {\em arXiv e-prints}, page arXiv:2010.14390, October 2020.

\bibitem{GaoLiuMiaoXi2023}
Chuanwei Gao, Bochen Liu, Changxing Miao, and Yakun Xi.
\newblock Improved local smoothing estimate for the wave equation in higher
  dimensions.
\newblock {\em J. Funct. Anal.}, 284(9):Paper No. 109879, 2023.

\bibitem{GarrigosSeeger2009}
Gustavo Garrig\'{o}s and Andreas Seeger.
\newblock On plate decompositions of cone multipliers.
\newblock {\em Proc. Edinb. Math. Soc. (2)}, 52(3):631--651, 2009.

\bibitem{Guth2016}
Larry Guth.
\newblock A restriction estimate using polynomial partitioning.
\newblock {\em J. Amer. Math. Soc.}, 29(2):371--413, 2016.

\bibitem{Guth2018}
Larry Guth.
\newblock Restriction estimates using polynomial partitioning {II}.
\newblock {\em Acta Math.}, 221(1):81--142, 2018.

\bibitem{GuthHickmanIliopoulou2019}
Larry Guth, Jonathan Hickman, and Marina Iliopoulou.
\newblock Sharp estimates for oscillatory integral operators via polynomial
  partitioning.
\newblock {\em Acta Math.}, 223(2):251--376, 2019.

\bibitem{GuthKatz2015}
Larry Guth and Nets~Hawk Katz.
\newblock On the {E}rd{\H{o}}s distinct distances problem in the plane.
\newblock {\em Ann. of Math. (2)}, 181(1):155--190, 2015.

\bibitem{GuthWangZhang2020}
Larry Guth, Hong Wang, and Ruixiang Zhang.
\newblock A sharp square function estimate for the cone in {$\Bbb {R}^3$}.
\newblock {\em Ann. of Math. (2)}, 192(2):551--581, 2020.

\bibitem{Harris2019}
Terence L.~J. Harris.
\newblock Improved decay of conical averages of the {F}ourier transform.
\newblock {\em Proc. Amer. Math. Soc.}, 147(11):4781--4796, 2019.

\bibitem{HickmanIliopoulou2022}
Jonathan Hickman and Marina Iliopoulou.
\newblock Sharp {$L^p$} estimates for oscillatory integral operators of
  arbitrary signature.
\newblock {\em Math. Z.}, 301(1):1143--1189, 2022.

\bibitem{Hoermander1973}
Lars H\"{o}rmander.
\newblock Oscillatory integrals and multipliers on {$FL\sp{p}$}.
\newblock {\em Ark. Mat.}, 11:1--11, 1973.

\bibitem{Lee2020}
Jungjin Lee.
\newblock A trilinear approach to square function and local smoothing estimates
  for the wave operator.
\newblock {\em Indiana Univ. Math. J.}, 69(6):2005--2033, 2020.

\bibitem{Lee2006}
Sanghyuk Lee.
\newblock Linear and bilinear estimates for oscillatory integral operators
  related to restriction to hypersurfaces.
\newblock {\em J. Funct. Anal.}, 241(1):56--98, 2006.

\bibitem{LeeVargas2012}
Sanghyuk Lee and Ana Vargas.
\newblock On the cone multiplier in {$\Bbb{R}^3$}.
\newblock {\em J. Funct. Anal.}, 263(4):925--940, 2012.

\bibitem{MinicozziSogge1997}
William~P. Minicozzi, II and Christopher~D. Sogge.
\newblock Negative results for {N}ikodym maximal functions and related
  oscillatory integrals in curved space.
\newblock {\em Math. Res. Lett.}, 4(2-3):221--237, 1997.

\bibitem{Miyachi1980}
Akihiko Miyachi.
\newblock On some estimates for the wave equation in {$L\sp{p}$} and
  {$H\sp{p}$}.
\newblock {\em J. Fac. Sci. Univ. Tokyo Sect. IA Math.}, 27(2):331--354, 1980.

\bibitem{MockenhauptSeegerSogge1993}
Gerd Mockenhaupt, Andreas Seeger, and Christopher~D. Sogge.
\newblock Local smoothing of {F}ourier integral operators and
  {C}arleson-{S}j\"{o}lin estimates.
\newblock {\em J. Amer. Math. Soc.}, 6(1):65--130, 1993.

\bibitem{OuWang2022}
Yumeng Ou and Hong Wang.
\newblock A cone restriction estimate using polynomial partitioning.
\newblock {\em J. Eur. Math. Soc. (JEMS)}, 24(10):3557--3595, 2022.

\bibitem{Peral1980}
Juan~C. Peral.
\newblock {$L\sp{p}$} estimates for the wave equation.
\newblock {\em J. Functional Analysis}, 36(1):114--145, 1980.

\bibitem{SeegerSoggeStein1991}
Andreas Seeger, Christopher~D. Sogge, and Elias~M. Stein.
\newblock Regularity properties of {F}ourier integral operators.
\newblock {\em Ann. of Math. (2)}, 134(2):231--251, 1991.

\bibitem{Sogge1991}
Christopher~D. Sogge.
\newblock Propagation of singularities and maximal functions in the plane.
\newblock {\em Invent. Math.}, 104(2):349--376, 1991.

\bibitem{Sogge1999}
Christopher~D. Sogge.
\newblock Concerning {N}ikod\'{y}m-type sets in {$3$}-dimensional curved
  spaces.
\newblock {\em J. Amer. Math. Soc.}, 12(1):1--31, 1999.

\bibitem{Sogge2017}
Christopher~D. Sogge.
\newblock {\em Fourier integrals in classical analysis}, volume 210 of {\em
  Cambridge Tracts in Mathematics}.
\newblock Cambridge University Press, Cambridge, second edition, 2017.

\bibitem{Stein1993}
Elias~M. Stein.
\newblock {\em Harmonic analysis: real-variable methods, orthogonality, and
  oscillatory integrals}, volume~43 of {\em Princeton Mathematical Series}.
\newblock Princeton University Press, Princeton, NJ, 1993.
\newblock With the assistance of Timothy S. Murphy, Monographs in Harmonic
  Analysis, III.

\bibitem{Taberner1985}
Bartolome~Barcelo Taberner.
\newblock On the restriction of the {F}ourier transform to a conical surface.
\newblock {\em Trans. Amer. Math. Soc.}, 292(1):321--333, 1985.

\bibitem{Tao2003}
T.~Tao.
\newblock A sharp bilinear restrictions estimate for paraboloids.
\newblock {\em Geom. Funct. Anal.}, 13(6):1359--1384, 2003.

\bibitem{Tao1998}
Terence Tao.
\newblock The weak-type endpoint {B}ochner-{R}iesz conjecture and related
  topics.
\newblock {\em Indiana Univ. Math. J.}, 47(3):1097--1124, 1998.

\bibitem{Tao1999}
Terence Tao.
\newblock The {B}ochner-{R}iesz conjecture implies the restriction conjecture.
\newblock {\em Duke Math. J.}, 96(2):363--375, 1999.

\bibitem{Tao2001}
Terence Tao.
\newblock Endpoint bilinear restriction theorems for the cone, and some sharp
  null form estimates.
\newblock {\em Math. Z.}, 238(2):215--268, 2001.

\bibitem{Wisewell2005}
L.~Wisewell.
\newblock Kakeya sets of curves.
\newblock {\em Geom. Funct. Anal.}, 15(6):1319--1362, 2005.

\bibitem{Wolff2000}
T.~Wolff.
\newblock Local smoothing type estimates on {$L^p$} for large {$p$}.
\newblock {\em Geom. Funct. Anal.}, 10(5):1237--1288, 2000.

\bibitem{Wolff2001}
Thomas Wolff.
\newblock A sharp bilinear cone restriction estimate.
\newblock {\em Ann. of Math. (2)}, 153(3):661--698, 2001.

\bibitem{Wongkew1993}
Richard Wongkew.
\newblock Volumes of tubular neighbourhoods of real algebraic varieties.
\newblock {\em Pacific J. Math.}, 159(1):177--184, 1993.

\end{thebibliography}

\end{document}